\documentclass[12pt]{amsart}
\usepackage[scale=.80]{geometry}
\usepackage{geometry} 
\usepackage{graphicx}
\usepackage{bbm}
\usepackage{mathrsfs}
\theoremstyle{plain}

\usepackage{CJK}
\usepackage{amsfonts,amssymb,amsmath,mathrsfs,multirow,booktabs}

\usepackage{color,latexsym,amsfonts}
 \newtheorem{theorem}{Theorem}[section]
 \newtheorem{lemma}[theorem]{Lemma}
 \newtheorem{corollary}[theorem]{Corollary}
 \newtheorem{proposition}[theorem]{Proposition}
 \newtheorem{example}[theorem]{Example}
 \newtheorem{Definition}[theorem]{Definition}
 \newtheorem{condition}[theorem]{Condition}
\theoremstyle{remark}
\newtheorem{remark}[theorem]{Remark}

 \def\beqlb{\begin{eqnarray}}\def\eeqlb{\end{eqnarray}}
 \def\beqnn{\begin{eqnarray*}}\def\eeqnn{\end{eqnarray*}}

 \def\<{\langle}\def\>{\rangle}

 \def\eqref#1{{\rm(\ref{#1})}}

 \def\qed{\hfill$\Box$\medskip}

\def\<{\left<}\def\>{\right>}

\newfam\msbmfam\font\tenmsbm=msbm10\textfont
\msbmfam=\tenmsbm\font\sevenmsbm=msbm7
\scriptfont\msbmfam=\sevenmsbm

\def\<{\left<}\def\>{\right>}

\def\({\left(}\def\){\right)}

\title[$\Lambda$-Wright-Fisher processes with selection \& simple EFC processes]{ \bf  On the boundary classification of $\Lambda$-Wright-Fisher processes with frequency-dependent selection}
\usepackage[pagebackref,colorlinks,citecolor=Plum,urlcolor=Periwinkle,linkcolor=DarkOrchid]{hyperref}
\newcommand{\ddr}{\mathrm{d}}
\usepackage[dvipsnames]{xcolor}
\usepackage{tikz}

\keywords{{$\Lambda$-Wright-Fisher process}, {selection}, {$\Lambda$-coalescent}, {fragmentation}, {explosion}, {coming down from infinity}, {entrance boundary}, {regular boundary}, {continuous-time Markov chains}.}
\subjclass[2010]{60J80, 60J70, 60J90, 92D25}
\begin{document}
\maketitle
\centerline{\large  Cl\'ement Foucart \footnote{Universit\'e  Sorbonne Paris Nord and Paris 8, Laboratoire Analyse, G\'eom\'etrie $\&$ Applications, UMR 7539. Institut Galil\'ee, 99 avenue J.B. Cl\'ement, 93430 Villetaneuse, France, foucart@math.univ-paris13.fr} and Xiaowen Zhou \footnote{Department of Mathematics and Statistics, Concordia University, 1455 De Maisonneuve Blvd. W., Montreal, Canada, xiaowen.zhou@concordia.ca}
}
\begin{center}
\today

\end{center}
\begin{abstract}
We construct extensions of the pure-jump $\Lambda$-Wright-Fisher processes with frequency-dependent selection ($\Lambda$-WF processes with selection) beyond their first passage time at the boundary $1$. We show that they satisfy some duality relationships with the block counting process of simple exchangeable fragmentation-coalescence processes (EFC).  
One-to-one correspondences between the nature of the boundary $1$ of the $\Lambda$-WF process with selection and the boundary $\infty$ of the block counting process are established. New properties for the $\Lambda$-WF processes with selection and the block counting processes of the simple EFC processes are deduced from these correspondences. Some conditions are provided for the selection to be either weak enough for boundary $1$ to be an exit boundary or strong enough for $1$ to be an entrance boundary. When the measure $\Lambda$ and the selection mechanism satisfy some regular variation properties, conditions are found in order that the extended $\Lambda$-WF process with selection makes excursions out from the boundary $1$ before getting absorbed at $0$. In the latter process, $1$ is a transient regular reflecting boundary. This corresponds to a new phenomenon for the deleterious allele which can spread into the population in a set of times of zero Lebesgue measure, before vanishing in  finite time almost surely.
\end{abstract}

%\section{Road map}
%\begin{enumerate}
%\item In the case $\theta^{\star}<1$ is it true that $\mathbb{E}_\infty(\tau_1)<\infty$? We want to show that if the process comes down from infinity then it is positive recurrent on $\mathbb{N}$ or $\bar{\mathbb{N}}$.
%\item 
%\end{enumerate}
\section{Introduction}
The $\Lambda$-Wright-Fisher processes represent the evolution of the frequency of a \textit{neutral} allele (or type) in a two-allele model evolving by resampling. A well-known result in the coalescent theory states that any $\Lambda$-Wright-Fisher process satisfies a certain \textit{duality relationship} with a $\Lambda$-coalescent, see Donnelly and Kurtz \cite{DonnKurtz} and Bertoin and Le Gall \cite{MR1990057}.  A consequence of this duality is that the $\Lambda$-Wright-Fisher process gets absorbed in finite time at one of its boundaries if and only if the $\Lambda$-coalescent comes down from infinity. This event, called \textit{fixation} in the genetics terminology mirrors the fact that even though no allele presents any advantage, the resampling of the population can reduce allelic diversity by the law of chance.  This phenomenon is called \textit{random genetic drift}. We refer the reader  to, for instance, Etheridge's book \cite{etheridge2011}. In this article, a generalisation of the $\Lambda$-Wright-Fisher process taking into account an extra force of selection is considered. 

Selection dynamics are typically modelled deterministically, so that the frequency of a type evolves both due to the resampling and due to a frequency-dependent term modeling how deleterious  the allele considered is. 
%Many authors since the nineties have looked for such more advanced models. 
Recently  Gonz\'alez and Span\`o \cite{zbMATH06873684} have established that discrete Wright-Fisher models with frequency-dependent selection can be rescaled to converge towards certain Markov processes called $\Xi$-Wright-Fisher process with frequency-dependent selection. In the latter work the approach of Neuhauser and Krone \cite{zbMATH01078990} for modeling logistic selection is generalized by relating selection events with multiple (and not only binary) branching events in the ancestral genealogy. We shall focus on the simpler setting of $\Lambda$-Wright-Fisher process with frequency-dependent selection ($\Lambda$-WF processes with selection). In those processes, resampling events are simple in the sense that they only involve one fraction of the population.%   The main object of study is the $\Lambda$-Wright-Fisher process with frequency-dependent selection. We call them for short $\Lambda$-WFs with selection. 

Let $\Lambda$ be a finite measure over $(0,1)$. Let $\mu$ be a finite measure on $\mathbb{N}:=\{1,2,\ldots\}$. Denote by $f$ the generating function of the probability measure $\xi(\cdot)=\mu(\cdot)/\mu(\mathbb{N})$ over $\mathbb{N}$, for all $x\in [0,1]$, $f(x):=\sum_{k=1}^{\infty}x^{k}\xi(k)$. Consider the following stochastic equation
\begin{align}\label{SDEselecintro}X_t(x)=x+\int_{0}^{t}\int_{0}^{1}\int_{0}^{1}z\left(\mathbbm{1}_{\{v\leq X_{s-}(x)\}}-X_{s-}(x)\right)&\bar{\mathcal{M}}(\ddr s,\ddr v, \ddr z)\\ \nonumber 
&+\mu(\mathbb{N})\int_{0}^{t}(f(X_{s}(x))-X_{s}(x))\ddr s,\end{align}
where $\mathcal{M}$ is a Poisson point process on $\mathbb{R}_+\times [0,1]\times [0,1]$ with intensity $m(\ddr t,\ddr v, \ddr z)=\ddr t\otimes \ddr v\otimes z^{-2}\Lambda(\ddr z)$ and $\bar{\mathcal{M}}$ stands for the compensated measure $\bar{\mathcal{M}}=\mathcal{M}-m$.
%% and $(B_t,t\geq 0)$ stands for an independent Brownian motion. 
Notice that $f(0)=f(1)=0$ and for all $x\in [0,1]$, $f(x)-x\leq 0$ so that the drift term is negative.   In the general case, $\Lambda$-WF processes may have a diffusion part. We focus in this work on the case of a measure $\Lambda$ on $[0,1]$ with no mass at $0$.  
%We assume in this work that this is not the case, so that the measure $\Lambda$ on $[0,1]$ has no mass at $0$. 

A solution $(X_t(x),t\geq 0)$ to \eqref{SDEselecintro}  is valued in $[0,1]$ and its dynamics can be understood as follows. Imagine a population of constant size $1$, whose individuals carry at any time one allele among a set of two alleles $\{a,A\}$. 
%Prior to a reproduction event in the population, an individual is picked at random in the population, accordingly to the current distribution of the alleles, and has for progeny a certain fraction $z$ of the population, ``chosen"  with intensity $z^{-2}\Lambda(\ddr z)$. Those individuals inherit then the allele of their parent, and no allele presents an advantage over the other one (the model is said to be neutral). 
The process $(X_t(x),t\geq 0)$ follows the frequency of allele $a$ when initially the proportion of individuals carrying allele $a$ is of size $x$. The time-dynamics of $(X_t(x),t\geq 0)$ consists of two parts: 
\begin{itemize}
\item the resampling which is governed by the Poisson random measure $\mathcal{M}$: for any  $(t,v,z)$ atom of $\mathcal{M}$,\\
\begin{itemize}
\item 
%if $v\leq X_{t-}(x)$, then allele $a$ is sampled and  a fraction $z\in (0,1)$ of the population carrying allele $A$ at time $t-$ takes the allele $a$ at time $t$. The proportion of allele $a$ increases:
if $v\leq X_{t-}(x)$, then allele $a$ is sampled and  a fraction $z\in (0,1)$ of the alleles $A$ at time $t-$ is replaced by the allele $a$ at time $t$. The frequency of allele $a$ increases:
\[X_t(x)=z\big(1-X_{t-}(x)\big)+X_{t-}(x),\]

\item if $v>X_{t-}(x)$, then allele $A$ is sampled and a fraction $z\in (0,1)$ of the alleles $a$ at time $t-$ is replaced by the allele $A$ at time $t$. The frequency of allele $a$ decreases:
\[X_t(x)=(1-z)X_{t-}(x),\]
\end{itemize}
\item the selection which is modeled by function $f$  characterizes the disadvantage of allele $a$:  the frequency of allele $a$ decreases continuously in time along the negative deterministic drift: \[\mu(\mathbb{N})\big(f(X_{t}(x))-X_t(x)\big)\ddr t.\]
\end{itemize}
%\begin{itemize}
%\item if $v\leq X_{t-}(x)$, the individual $v$ sampled at time $t-$ carries allele $a$ and has for progeny a fraction $z\in (0,1)$ of the population. The proportion of allele $a$ increases:
%\[X_t(x)=z(1-X_{t-}(x))+X_{t-}(x),\]
%\item if $v>X_{t-}(x)$, the individual $v$ sampled at time $t-$ carries allele $A$ and has for progeny a fraction $z$ of the population. The proportion of allele $a$ decreases:
%\[X_t(x)=(1-z)X_{t-}(x),\]
%\end{itemize}
%\item the selection which is modeled by function $f$  characterizes the disadvantage of allele $a$:  the frequency of allele $a$ decreases continuously in time along the negative deterministic drift: \[\mu(\mathbb{N})(f(X_{t}(x))-X_t(x))\ddr t.\]
%\end{itemize}

When $f$ vanishes, the drift term in \eqref{SDEselecintro} governing the selection disappears and the solution of \eqref{SDEselecintro} becomes the classical $\Lambda$-Wright-Fisher process, see Bertoin and Le Gall \cite{LGB2} and Dawson and Li \cite{DawsonLi}. In particular, when there is no selection term, the SDE \eqref{SDEselecintro} has a pathwise unique strong solution and the boundaries $0$ and $1$ are both absorbing whenever they are reached. The event of absorption at $1$ is called \textit{fixation} of the allele $a$. It corresponds to the fact that all individuals have the common type $a$ in finite time almost surely. 
%This reflects that $\Lambda$-Wright-Fisher processes are fundamental models for which the phenomenon of \textit{random genetic drift} occurs. That is to say, although no allele in the population is advantaged, one allele will fix (in finite time) only due to the law of chance.  
Bertoin and Le Gall \cite{LGB2}  have established that in the setting with no selection, if $\Lambda(\{1\})=0$ then the event of \textit{fixation} at one of the boundaries
%$i\in \{0,1\}$, namely \begin{equation}\label{deffixation}\{\exists \ t_{i}>0; \ X_t(x)=i \text{ for all } t\geq t_{i}\}\end{equation} 
occurs if and only if the measure $\Lambda$ satisfies the following condition
\begin{equation}\label{cdipsi}
\sum_{n\geq 2}\frac{1}{\Phi(n)}<\infty,
\end{equation}
where for any $n\geq 2$,
\begin{equation}\label{phi2}
\Phi(n)=\int_{(0,1)}\left((1-x)^{n}+nx-1\right)x^{-2}\Lambda(\ddr x).
\end{equation}
Condition \eqref{cdipsi} is perhaps better known in the coalescent theory, as the necessary and sufficient condition for the $\Lambda$-coalescent to come down from infinity. We will return to this fact later. We assume from now on that $\Lambda(\{1\})=0$, so that no single resampling event can give to all individuals the same allele. 

One of the first models generalizing the $\Lambda$-Wright-Fisher process by incorporating selection is  perhaps the logistic case for which $f(x)=x^2$ and the drift term in the SDE \eqref{SDEselecintro} takes the form $-\alpha x(1-x)$ with $\alpha=\mu(\mathbb{N})$.  In this setting  the measure $\mu$ reduces to a Dirac mass at $2$. Such processes have been studied 
%Some properties on the long-term behavior of $(X_t,t\geq 0)$ can be deduced from this identity. For instance, if one shows that the process $(N_t^{(n)},t\geq 0)$ can be started from  $n=\infty$ and has $\infty$ as an instantaneous entrance boundary; namely $N_t^{(\infty)}<\infty$ for all $t>0$, then by letting $n$ to $\infty$ in \eqref{dualmoment1}, we see as in Equation \eqref{classicid} that the process $(X_t(x),t\geq 0)$ gets absorbed at $1$ in finite time with positive probability. 
%In the specific binary case, it is known that the process $(N_t^{(\infty)},t\geq 0)$ comes down from infinity as soon as the pure $\Lambda$-coalescent does.  We refer for instance to 
by Baake et al. \cite{baake2016}, Bah and Pardoux \cite{zbMATH06403252}, Etheridge and Griffiths \cite{zbMATH05832932}, Griffiths \cite{zbMATH06387922} and Foucart \cite{zbMATH06346921}. 
% This entails that in the classical model with ``binary" selection, 
%Bah and Pardoux \cite[Theorem 4.3]{zbMATH06403252} have established that in the logistic case, fixation occurs in the process if and only if \eqref{cdipsi} is satisfied. So that despite that  allele $a$ is deleterious, the population has a positive probability to concentrate on  allele $a$ in a finite time. 
The SDE \eqref{SDEselecintro}  has been considered in the study of discrete branching processes with interactions by  Gonz\'alez et al. \cite{Gonzalesetal} with function $f$ satisfying $f'(1-)<\infty$ (i.e. $\mu$ has a first moment).
% with a more general drift subject to a global Lipschitz condition,
%See e.g. \cite[Theorem 4.9 page 107]{Beres2}. Still in the binary case, 
%When \eqref{cdipsi} is violated, an other parameter measuring the strength of the $\Lambda$-resampling, has to be taken into account, see \cite{zbMATH06387922} and the references therein. 
%

%Our main objective is to study the possible boundary behaviors of processes solving \eqref{SDEselecintro}. Notice that when the function $f$ is not smooth enough, pathwise uniqueness of the solution to \eqref{SDEselecintro} is not guaranteed. Typically if the function $f$ is not Lipschitz at $1$, the drift could be strong enough to push the process away from $1$. The behavior of the (positive) function $x\mapsto x-f(x)$ near $1$ reflects the strength of the selective advantage of allele $A$ over $a$.

%Unless explicitly mentioned, the measure $\Lambda$ will always satisfy the condition \eqref{cdipsi}. % and we shall also assume that there is no Kingman part, i.e. $\Lambda(\{0\}=0$. 
The behavior of the positive function $x\mapsto x-f(x)$ near $1$ reflects the strength of the selective advantage of allele $A$ over $a$.   The question addressed in the present article is to see whether a  selection term  can overcome the $\Lambda$-resampling mechanism and prevent  fixation of the deleterious allele $a$. We will find indeed new phenomena occurring in the presence of certain strong selection. In particular, despite the strength of the resampling rule under the condition \eqref{cdipsi}, we shall find regimes for which even though the population starts entirely with the deleterious allele, its frequency will vanish in finite time almost surely.  Our main results concern specifically the case for which 

\begin{equation}\label{regularlambda}
\Lambda(\ddr z)=h(z)\ddr z,\text{ for } x\in [0,x_0] \text{ with } h(z)z^{\beta}\underset{x\rightarrow 0+}{\longrightarrow}\rho, 
\end{equation}
\text{ and } 
\begin{equation}\label{regularmu}
\mu(\mathbb{N})\big(x-f(x)\big)\underset{x\rightarrow 1-}{\sim} \sigma (1-x)^{\alpha},
\end{equation}
for some $x_0\in (0,1]$, $\alpha,\beta\in (0,1)$ and $\sigma,\rho>0$. In this setting, when $\beta=1-\alpha$, we shall find regimes in which all individuals carry the deleterious allele at a set of times of negligible Lebesgue measure, before the selection starts to act effectively and that the deleterious allele vanishes, see Theorem \ref{fixationat0} and Theorem \ref{regularcase}. 

In a more theoretic wording, when the measure $\mu$ is heavy-tailed, the drift term in \eqref{SDEselecintro} becomes non-Lipschitz at $1$, namely $f'(1)=\infty$, and pathwise uniqueness of the solution to the SDE \eqref{SDEselecintro} might not hold. 
%Typically if the function $f$ is not Lipschitz at $1$, the drift could be strong enough to push the process away from $1$, so that several solutions of \eqref{SDEselecintro} might exist depending on the nature of the boundary $1$.
%This is for instance the case in the pure deterministic case: the ODE
%\[X_{t}(x)=x+\mu(\mathbb{N})\int_{0}^{t}\big(f(X_s(x))-X_s(x)\big)\ddr s\]
%has two distinct solutions started from $1$ when $\int^1 \frac{\ddr x}{x-f(x)}<\infty$, $X_t:=1$ for all $t\geq 0$, and $X_t:=\underset{x\rightarrow 1-}{\lim} X_t(x)$ for all $t\geq 0$. The latter integral converges for example when \eqref{regularmu} is satisfies with $\alpha<1$.
We shall be interested in the different ways of extension of the minimal process solution to \eqref{SDEselecintro} after it has reached boundary $1$. Following Feller's terminology for diffusions, see e.g. Karlin and Taylor's book \cite[Chapter 15, Section 6]{MR611513},
a boundary point is said to be natural if the process can not reach the boundary and can not leave it. The boundary is an exit if the process  can reach the boundary but can not leave it. Symmetrically, it is said to be an entrance if the process can not access the boundary but leaves it; and finally the boundary is regular if the process enters into it and is able to get out from it.

Our method relies on the study of an extension of the minimal process, solution to \eqref{SDEselecintro}, after it has reached boundary $1$ constructed in the following way. We first look at processes, solution to the Equation \eqref{SDEselecintro} with an additional drift term  $-\lambda X_t\ddr t$  with $\lambda>0$. This drift can be seen as modeling \textit{mutation} from the deleterious allele $a$ to the advantaged one $A$. We shall see that under the assumption that there is no Kingman component, i.e. $\Lambda(\{0\})=0$, those processes, call them $X^{\lambda}$'s, can all be started from boundary $1$. Our core object of study is the limit process that arises when the parameter $\lambda$  tends to $0$ (i.e. the mutation becomes very low).  Hence define formally the limit process $X^{\mathrm{r}}$ as $X^{\mathrm{r}}_t:=\underset{\lambda \rightarrow 0+}{\lim} X^{\lambda}_t$ for all $t\geq 0$. The convergence will be made precise later and we shall see that the process $X^{\mathrm{r}}$ is extending the minimal solution to \eqref{SDEselecintro}. In order to classify its boundaries, we will investigate some duality relationships that it necessarily satisfies with respect to certain continuous-time Markov chains arising as the functional of the block counting process $N$ of a \textit{simple} exchangeable fragmentation-coalescence (EFC) process. These duality relations can be seen as generalization of the well-known moment duality between a $\Lambda$-WF process and the block counting process of a $\Lambda$-coalescent. Basic facts about moment-duality in the simpler case without selection are recalled in Section \ref{backgroundWF}. 

EFC processes have been introduced by Berestycki in \cite{MR2110018}. They are partition-valued processes whose laws are invariant under the action of finite permutation and whose blocks (i.e. equivalence classes) can coalesce as in an exchangeable coalescent and fragmentate as in an homogeneous fragmentation. Simple EFC processes form a subclass of EFC processes in which coalescence occurs as in a $\Lambda$-coalescent and fragmentation occurs at finite rate.  
%More precisely, independently of each others, each infinite block is fragmentated into $k+1$ sub-blocks (hence creating $k$ new blocks) at rate $\mu(k)$ with $k\in \bar{\mathbb{N}}:=\{1,2,\cdots, \infty\}$.
More backgrounds on EFC processes are provided in Section \ref{backgroundEFC}. Simple EFCs and their block counting processes have been studied in Foucart \cite{cdiEFC} and Foucart and Zhou \cite{explosion}. The boundary $\infty$ has been investigated in those articles and sufficient conditions for the boundary $\infty$ to be entrance, exit or regular have been identified. We shall transfer the results on EFC processes to results for $\Lambda$-WF processes with selection and vice versa, see Section \ref{application}.

%{\red{
%In the sequel when saying that a boundary is regular \textit{non-absorbing}, we mean that the process  indeed leaves the boundary  after hitting it.
In the sequel, we say that a boundary is \textit{absorbing} if when started from the boundary, the process stays at the boundary at any future time.  
%{\red Given $\mathbb{P}(X^{\mathrm{r}}_{\tau_1+t}(x)\neq 1|\tau_1<\infty)=\mathbb{P}(X^{\mathrm{r}}_{t}(1)<1)>0$, is it necessary that $\mathbb{P}(X^{\mathrm{r}}_{\tau_1+t}(x)\neq 1|\tau_1<\infty)=\mathbb{P}(X^{\mathrm{r}}_{t}(1)<1)=1$?  }
So that an exit boundary is absorbing and a regular boundary is absorbing if it is subject to the prescription that the process fixes at the boundary once it is attained. In other words, the process with a {\it regular absorbing} boundary is stopped at the boundary.
%In particular, boundary $1$ is regular absorbing if $\tau_1<\infty$ with positive probability and
%% $\tau_1:=\inf\{t>0, X_t^{\mathrm{r}}(x)=1\}<\infty$ with positive probability, and on this event,
%\begin{center} $\mathbb{P}(X^{\mathrm{r}}_{\tau_1+t}(x)=1 \text{ for all }t\geq 0|\tau_1<\infty)=\mathbb{P}(X^{\mathrm{r}}_{t}(1)=1 \text{ for all }t\geq 0)=1$.\end{center}
When saying that a boundary is \textit{non-absorbing}, we mean that the process, when  started from the boundary, leaves it at some future time with positive probability. An entrance boundary is thus non-absorbing as well as a regular boundary when the process is not stopped at it.  
%In the latter case, the boundary is said to be regular non absorbing.
%}}

%In order to study the different possible boundary conditions of the $\Lambda$-Wright-Fisher process with selection at boundary $1$, we will establish not one, but two relationships of duality, see Theorems \ref{thmmomentdual1} and \ref{thmmomentduality2}. One for the process $(X_t^{\mathrm{min}},t\geq 0)$ solution to \eqref{SDEselecintro} that is absorbed  after it has hitted the boundary 1, and one for the process $(X_t^{\mathrm{r}},t\geq 0)$, which we need to construct rigorously, that might get out from the boundary $1$. We provide now more details on the contents of the article and the role played by duality. 

%We shall establish in Theorem \ref{thmmomentdual1} and Theorem \ref{thmmomentduality2} the following one-to-one correspondences between boundaries conditions.
%{\red{
%We highlight on the fact that by regular non-absorbing, it is meant that the process leaves the regular boundary.  
In a more formal setting, 
%The possible boundary behaviors of the $\Lambda$-Wright-Fisher process with selection can be set more formally as follows.
for any $x\in [0,1]$, let $\tau_1$ be the first hitting time of boundary $1$, i.e. $\tau_1:=\inf\{t>0: X_t^{\mathrm{r}}(x)=1\}\in [0,\infty]$. The $\Lambda$-WF process with selection has boundary $1$ accessible if for any $x\in (0,1]$, $\mathbb{P}_x(\tau_1<\infty)>0$. Furthermore, the boundary  $1$ is
\begin{itemize}
\item[-] exit: if $1$ is accessible and when $\tau_1<\infty$, for all $t\geq \tau_1$, $X^{\mathrm{r}}_t(1)=1$ a.s;
\item[-] entrance: if $1$ is not accessible and for all $t>0$, $X^{\mathrm{r}}_t(1)<1$ a.s;
\item[-] regular non-absorbing: $1$ is accessible and
%\begin{center} 
%$\mathbb{P}(X^{\mathrm{r}}_{\tau_1+t}(x)\neq 1 \text{ for some }t >0|\tau_1<\infty)=$
$\mathbb{P}(X^{\mathrm{r}}_{t}(1)<1\text{ for some }t>0)>0$; 
%\end{center}
\item[-] natural: if $1$ is not accessible and for all $t\geq 0$, $X_t^{\mathrm{r}}(1)=1$ almost surely.
\end{itemize}
As explained before, a boundary is said to be regular absorbing if the boundary is regular but the process is stopped after reaching it. In the case of the $\Lambda$-WF process with selection, the process with boundary $1$ absorbing is $(X_{t\wedge \tau_1}^{\mathrm{r}},t\geq 0)$. The latter will coincide with the process $(X_t^{\mathrm{min}},t\geq 0)$ solution to \eqref{SDEselecintro} that is absorbed  after it has hitted the boundary $1$. In order to classify the different possibilities, we will establish not one, but two relationships of duality, see Theorems \ref{thmmomentdual1} and \ref{thmmomentduality2}. The first duality relation is for the process $(X_t^{\mathrm{min}},t\geq 0)$ and the second is for the process $(X_t^{\mathrm{r}},t\geq 0)$ that might get out from the boundary $1$, but which we need to construct rigorously.
%We provide now more details on the contents of the article and the role played by duality. 

The two duality relationships  have for consequence the following one-to-one correspondences between boundaries conditions. Denote by $N$ the block counting process of the simple EFC process with coalescence measure $\Lambda$ and splitting measure $\mu$, see Section \ref{backgroundEFC}.
%, and by $X$ the $\Lambda$-WF process with selection. 
\begin{table}[htpb]
\begin{center}
\begin{tabular}{|c|c|}
\hline
Boundary condition on $X$ & Boundary condition on $N$\\
\hline
 $1$ exit & $\infty$  entrance \\
\hline
 $1$ entrance & $\infty$  exit \\ 
\hline
 $1$ regular non-absorbing & $\infty$  regular absorbing\\ 
\hline
 $1$ regular absorbing & $\infty$  regular non-absorbing \\
\hline
$1$ natural & $\infty$  natural\\
\hline
\end{tabular}
\vspace*{2mm}
\caption{Boundaries of $X$ and boundaries of $N$.}
\label{correspondance}
\end{center}
\end{table}
%The only main assumption in what follows is that there is no diffusion part so that $\Lambda(\{0\})=0$ and and no sudden replacement of all types by one type, i.e. $\Lambda(\{1\})=0$. 
%}}
% {\red We may give a more detailed description of regular non-absorbing here boundary here. Can the process stay for an exponential time at the boundary?
%What is the difference between regular and regular non-absorbing? }

The correspondences  stated in Table \ref{correspondance} between entrance and exit boundaries for processes satisfying a duality relationship have been observed in other contexts, see the seminal work of Cox and R\"osler \cite{zbMATH03828986}.  Reminiscent classifications between boundaries have been established in Foucart \cite{zbMATH07055671} for logistic continuous-state branching processes, see also Hermann and Pfaffelhuber \cite{Hermann} and Berzunza Ojeda and Pardo \cite{Berzunza}.  

Since there are several possible ways to leave a regular boundary, see e.g. \cite[Chapter 15, Section 8]{MR611513} in the case of diffusions processes, Table \ref{correspondance} does not precise completely the behavior of the process at the boundary when it is regular non-absorbing. 

Recall that a regular boundary is said to be \textit{reflecting} when the set of times at which the process lies at the boundary, has a zero Lebesgue measure.  A regular boundary is also said to be \textit{regular for itself} if the process started from the boundary returns immediately to it almost surely. 
%To the best of our knowledge Table \ref{correspondancereg} appears to be a new observation for processes in duality with respect to the moment function $g:(n,x)\mapsto x^n$.
%\newpage
In the same fashion, we classify the exit and entrance boundaries by saying that boundary $1$ is an instantaneous entrance if it is an entrance and the first entrance time in $[0,1)$, $\tau^{1}:=\inf\{t>0: X^{\mathrm{r}}_t(x)<1\}$ satisfies  for any $t>0$, $\mathbb{P}_x(\tau^{1}\leq t)\longrightarrow 1$, as $x$ goes to $1$. The boundary $\infty$ is an instantaneous exit if it is an exit and the first explosion time $\zeta_\infty:=\inf\{t>0: N^{(n)}_t=\infty\}$ satisfies for any $t>0$, $\mathbb{P}_n(\zeta_\infty \leq t)\longrightarrow 1$, as $n$ goes to $\infty$. Similar definitions hold for instantaneous exit boundary $1$ and instantaneous entrance boundary $\infty$.

We shall see in Theorem \ref{regularforitself} and Proposition \ref{instantanprop} that the following one-to-one correspondences hold :
\begin{table}[htpb]
\begin{center}
\begin{tabular}{|c|c|}
\hline
Boundary  of $X$ & Boundary of $N$   \\
\hline
$1$ regular reflecting & $\infty$  regular for itself\\
\hline
$1$ regular for itself & $\infty$  regular reflecting \\ 
\hline
$1$ instantaneous entrance & $\infty$  instantaneous exit \\ 
\hline
$1$ instantaneous exit & $\infty$  instantaneous entrance \\ 
\hline
\end{tabular}
\vspace*{2mm}
\caption{regular for itself/regular reflecting}
\label{correspondancereg}
\end{center}
\end{table}

We address now the long-term behavior of the extended $\Lambda$-Wright-Fisher process with selection $(X^{r}_t,t\geq 0)$. Recall that by fixation, we mean that all individuals get one of the two alleles and keep it forever. When $1$ is an exit, fixation of the deleterious allele has a positive probability to occur. When the boundary $1$ is regular non-absorbing or entrance, fixation at $1$ can not occur, and we shall actually see that there is almost sure fixation of the advantageous allele, see Theorem \ref{fixationat0}. Namely, if  the process  $(X^{\mathrm{r}}_t,t\geq 0)$ has boundary $1$ entrance or regular non-absorbing, then almost surely
\begin{center} $\exists \ t_0>0; X^{\mathrm{r}}_t(x)=0, \ \forall t\geq t_0$.\end{center} 
%In other words, the deleterious allele disappears in finite time. The fourth case in Table \ref{correspondance} 
%The  process $(X^{\mathrm{r}}_t,t\geq 0)$ can thus make excursions out from the boundary $1$ before vanishing, so that the deleterious allele can spread in all the population. 

Table \ref{correspondance} and Table \ref{correspondancereg} provide a theoretical classification of the boundaries. When there is selection, no necessary and sufficient conditions entailing that boundary $1$ of $X$ or boundary $\infty$ of $N$ is of a given type are known.  In Section \ref{application}, we design explicit sufficient conditions on the resampling measure $\Lambda$ and the selection function $f$ for each possible boundary condition.
% of the $\Lambda$-Wright-Fisher process with selection.
They are obtained  via the correspondences stated in Table \ref{correspondance} and Table \ref{correspondancereg}, by transfering previous results on the boundary $\infty$ of the block counting process $N$, obtained in \cite{explosion}, to results on the boundary $1$ of $X^{\mathrm{r}}$. Sufficient conditions on the function $f$ entailing that the boundary $1$ is an instantaneous entrance or an exit are given in Theorem \ref{theorementrancef} and Theorem \ref{suffcondexitf} respectively. When the conditions \eqref{regularlambda} and \eqref{regularmu} are satisfied, the nature of the boundary depends on the parameters $\alpha,\beta,\sigma$ and $\rho$. In particular, we show that when $\beta=1-\alpha$ and $\frac{1}{(1-\alpha)(2-\alpha)}<\sigma/\rho<\frac{\pi}{(2-\alpha)\sin(\pi \alpha)}$, the process $(X_t^{\mathrm{r}},t\geq 0)$ has boundary $1$ regular reflecting and regular for itself, see Theorem \ref{regularcase} and Theorem \ref{regularforitselfstable}.

The paper is organized as follows. Backgrounds on $\Lambda$-Wright-Fisher processes with and without selection are provided in Section \ref{backgroundWF}. We verify in particular that the SDE \eqref{SDEselecintro} admits a unique solution up to the first hitting time of the boundaries. The moment-duality between the $\Lambda$-Wright-Fisher process without selection and the process counting the number of blocks in a $\Lambda$-coalescent is recalled. Consequences of this duality relationships for the boundaries of the $\Lambda$-Wright-Fisher process without selection are reviewed. In Section \ref{backgroundEFC}, we briefly recall the notion of exchangeable fragmentation-coalescence processes and describe the process of its number of blocks. Results  from \cite{cdiEFC} are gathered. In Section \ref{sec}, two natural extensions of the minimal $\Lambda$-Wright-Fisher process are studied. We then provide a theoretical study of their possible behaviors at the boundary $1$. In Section \ref{application}, we apply the results obtained in  \cite{explosion} to the $\Lambda$-Wright-Fisher processes with selection and provide explicit sufficient conditions for each boundary condition.
%Finally, Section \ref{conclusion} is a conclusion and addresses the limitations and scope of our techniques. Some open problems are proposed.  
\\

\textbf{Notation.} We denote by $C([0,1])$ the space of continuous functions on $[0,1]$, $C^{2}((0,1))$ and $C^2_c([0,1])$ are respectively the space of twice continuously differentiable functions on $(0,1)$ and the space of twice continuously differentiable functions whose derivatives have compact support included in $(0,1)$. The integrability of a function $g$ in a left-neighbourhood of $a\in (0,\infty]$ is denoted by $\int^{a-}g(x)\ddr x<\infty$. We set $\bar{\mathbb{N}}:=\{1,2,\ldots, \infty\}$, the one-point compactification of $\mathbb{N}$, the latter is compact for the metric $d(n,m):=|n-m|$ for $n,m\in \mathbb{N}$ and $d(\infty,n):=1/n$ for any $n\in \bar{\mathbb{N}}$, where by convention $1/\infty =0$. The space of continuous functions on $\bar{\mathbb{N}}$ is denoted by $C(\bar{\mathbb{N}})$. A function $f$ belongs to $C(\bar{\mathbb{N}})$ if and only if $f(n)\underset{n\rightarrow \infty}{\longrightarrow} f(\infty)$. We denote by $\mathbb{P}_z$ the law of the process under consideration started from $z$, the corresponding expectation is $\mathbb{E}_z$.

\section{Background on $\Lambda$-WFs with selection and simple EFCs}\label{background}
\subsection{$\Lambda$-Wright-Fisher processes with selection and $\Lambda$-coalescent}\label{backgroundWF}
%If one starts initially from a population with two types $a$, $A$ (or two alleles in the genetics terminology), the frequency of a given type, say $a$, is known to be a Markov process valued in $[0,1]$ satisfying the following stochastic differential equation (SDE): 
%\begin{align}\label{WFSDE} X_t(x)=x+\int_{0}^{t}\int_{0}^{1}\int_{0}^{1}z\left(\mathbbm{1}_{\{v\leq X_{s-}(x)\}}-X_{s-}(x)\right)&\bar{\mathcal{M}}(\ddr s,\ddr v, \ddr z)
%\end{align}
%where $x$ is the initial proportion of individuals with allele $a$, $\mathcal{M}$ is a Poisson point process on $\mathbb{R}_+\times [0,1]\times [0,1]$ with intensity $m(\ddr t,\ddr v, \ddr z)=\ddr t\otimes \ddr v\otimes z^{-2}\Lambda(\ddr z)$ and $\bar{\mathcal{M}}$ stands for the
%compensated measure $\bar{\mathcal{M}}=\mathcal{M}-m$. 
%Note that $\int_{0}^{1}\left(\mathbbm{1}_{\{v\leq X_{s-}(x)\}}-X_{s-}\right)\ddr v=0$ so that the compensation plays no genuine role in the SDE \eqref{SDEselecintro}.
\subsubsection{A stochastic differential equation with jumps}
We introduce the class of $\Lambda$-Wright-Fisher processes with frequency-dependent selection. As we shall use it later, we slightly generalize SDE \eqref{SDEselecintro} by allowing the generating function  driving selection to be defective, namely such that $f(1)<1$.
%We postpone a discussion on the meaning in terms of genetics models of our results to the end of Section \ref{application}.  
Let $\mu$ be a finite measure over $\bar{\mathbb{N}}:=\{1,2,\ldots,\infty\}$ and $\Lambda$ be a finite measure over $[0,1]$ with $\Lambda(\{0\})=\Lambda(\{1\})=0$. Let $f$ be the generating function of the probability law $\mu(\cdot)/\mu(\bar{\mathbb{N}})$ over $\bar{\mathbb{N}}$: for all $x\in [0,1]$, $f(x)=\sum_{k\in \mathbb{N}}x^{k}\mu(k)/\mu(\bar{\mathbb{N}})$. When the function $f$ is defective, $1-f(1)>0$ and this term corresponds to the mass at infinity for the probability distribution associated to $f$.  Consider the stochastic equation
\begin{align}\label{SDEselec}X_t(x)=x+\int_{0}^{t}\int_{0}^{1}\int_{0}^{1}z\left(\mathbbm{1}_{\{v\leq X_{s-}(x)\}}-X_{s-}(x)\right)&\bar{\mathcal{M}}(\ddr s,\ddr v, \ddr z)\\ \nonumber 
&+\mu(\bar{\mathbb{N}})\int_{0}^{t}(f(X_{s}(x))-X_{s}(x))\ddr s.\end{align}
When there is no selection, i.e. $f\equiv 0$ and the drift term in Equation  \eqref{SDEselec} vanishes,  the process $(X_t(x),t\geq 0)$ valued in $[0,1]$, is a martingale (this property in terms of the population model can be thought as the neutrality assumption between the two alleles) and has both boundaries $0$ and $1$ absorbing. Existence and weak uniqueness of the solution to the SDE  \eqref{SDEselec}, when $f\equiv 0$, has been established by  Bertoin and Le Gall \cite{LGB2} through a martingale problem.  Set $q(v,x):=\mathrm{1}_{\{ v\leq x\}}-x$ for any $x\in [0,1]$, the generator of the $\Lambda$-Wright-Fisher process without selection is the operator $\mathcal{A}$ defined as follows
\begin{align}\label{generator}
\mathcal{A}g(x)&:=\int_{[0,1]\times [0,1]}\left(g(x+zq(v,x))-g(x)-zq(v,x)g'(x)\right)z^{-2}\Lambda(\ddr z).
\end{align}
Dawson and Li \cite{DawsonLi} have studied the SDE \eqref{SDEselec} through  techniques different from Bertoin and Le Gall. Among other results, it is established that under some assumptions on the drift term, \eqref{SDEselec} admits a flow of strong solutions $(X_t(x),t\geq 0,x\geq 0)$ for which $x\mapsto X_t(x)$ is c\`adl\`ag. 

Consider now the setting with selection. Any process solution to \eqref{SDEselec} has for generator $\mathcal{A}^{\mathrm{s}}$ acting on $C^{2}((0,1))$, by \begin{equation}\label{generatorWFs}\mathcal{A}^{\mathrm{s}}g(x):=\mathcal{A}g(x)+\mu(\bar{\mathbb{N}})(f(x)-x)g'(x) \text{ for any }x\in (0,1).
\end{equation}

\begin{lemma}\label{minimalprocess} Let $f$ be any generating function (possibly defective). There exists a unique strong solution to
 \eqref{SDEselec} up to the first hitting time of the boundaries. 
\end{lemma}
%{\red Do you mean $\left(g(X_t)-\int_{0}^{t}\mathcal{A}^{\mathrm{s}}g(X_s)\ddr s,t<\tau\right)1_{t<\tau} $ is a martingale?: no, I mean that $M_t:=g(X_t)-\int_{0}^{t}\mathcal{A}^{\mathrm{s}}g(X_s)\ddr s$ is a martingale for $t<\tau$: $\forall s<t<\tau$, $\mathbb{E}[M_{t}|\mathcal{F}_s]=M_s$. }
%\[\red \text{replace with}\quad g(X_{t\wedge\tau})-\int_{0}^{t\wedge\tau}\mathcal{A}^{\mathrm{s}}g(X_s)\ddr s, \quad t\geq 0 ?\]
%{\red{if we take yours, we need to define $X_\tau$ since $\tau<t$ could happen. I think, we should introduce a cemetery state $\partial\notin [0,1]$, so that on the event $\tau<\infty$, $X_t=\partial$ for all $t\geq \tau$ and hence $\mathcal{A}^{\mathrm{s}}f(\partial):=0$ for any $f:[0,1]\cup \{\partial\}\mapsto \mathbb{R}$, with $f(\partial)=\partial$. This allows us then to replace the condition $t<\tau$, by a stopped martingale as you suggest. It think this is indeed better to do this, and clarifies the difference between stopped and killed but I prefer to think about it again, what do you think?.}}
%\end{lemma}
\begin{proof} 
%We establish that there exists a unique strong solution $(X_t,t<\tau)$  to \eqref{SDEselec} until time $\tau:=\tau_0\wedge \tau_1$. This entails that  $(\mathrm{MP})$ is well-posed. 
%process $(X^{\mathrm{min}}_t,t<\tau)$ is  the unique solution to 
For any $n\geq 1$, one can find a Lipschitz function $b^{n}$ on $[0,1]$ such that $b^{n}(x):=\mu(\bar{\mathbb{N}})(f(x)-x)$ if $0\leq x\leq 1-1/n$ and $b^{n}(x):=0$ if $x\geq 1-1/2n$. 
Consider the stochastic equation
\begin{equation}\label{sdepre}X_t(x)=x+\int_{0}^{t}\int_{0}^{1}\int_{0}^{1}z\left(\mathbbm{1}_{\{v\leq X_{s-}(x)\}}-X_{s-}(x)\right)\bar{\mathcal{M}}(\ddr s,\ddr v, \ddr z)+\int_0^{t}b^{n}(X_s)\ddr s.
\end{equation}
We verify now that the latter equation has a pathwise unique strong solution by applying \cite[Theorem 5.1]{li2012}. We follow the notation of \cite{li2012} and check Conditions (3.a), (3.b) and (5.a). Set $U_0=[0,1]\times [0,1]$, for any $u\in U_0$, denote the coordinates of $u$ by  $u=(v,z)$ and let $\mu_0(\ddr u)=\mu_0(\ddr v,\ddr z):=\ddr v \otimes z^{-2}\Lambda(\ddr z)$. Define $g_0(x,u)=z(\mathbbm{1}_{\{v\leq x\}}-x)$ and note that the stochastic equation above corresponds to the stochastic equation (2.1) in \cite{li2012} with $\sigma\equiv 0, g_1\equiv 0$ and $b(x):=b^{n}(x)\leq 0$. Since $b^{n}$ is Lipschitz on $[0,1]$, condition (3.a) is satisfied (with $r_m(z)=r_1(z):=L_n z$ for any $z$ and any $m$, and where $L_n$ is the Lipschitz constant of $b_n$). Condition (3.b) is verified in \cite[Corollary 6.2]{li2012}. It remains only to check Condition (5.a). For any $x\in \mathbb{R}$, $g_0(x,(u,z))^2\leq z^2$ and $b^{n}(x)^2\leq 1$, hence \[\int_0^1 g_0(x,(u,z))^2 \mu_0(\ddr u)+b^{n}(x)^2\leq \int_0^1 z^{2}z^{-2}\Lambda(\ddr z)+1= \Lambda((0,1))+1\leq K(1+x^2).\]
Denote by $(X_t^{n},t\geq 0)$ the solution of \eqref{sdepre}. Let $\tau_{1-1/n}:=\inf\{t>0: X_t^{n}> 1-1/n\}$. By pathwise uniqueness if $m<n$ then $X_t^{m}=X_t^{n}$ for $t\leq \tau_{1-1/m}\wedge \tau_0$. Note that $\tau_{1-1/m}\underset{m\rightarrow \infty}{\longrightarrow} \tau_1$; thus we can define a process $(X_t,t<\tau)$ such that $X_t=X_t^{n}$ for all $t\leq \tau_{1-1/n}$ and all $n\geq 1$. The process is solution to \eqref{SDEselec} and uniqueness of this minimal solution plainly holds.\qed 
\end{proof}
\begin{remark}\label{Lipschitzrem}
When the selection term satisfies the condition $f'(1)<\infty$ (i.e. $\mu$ has a finite mean), the function $x\mapsto f(x)-x$ is Lipschitz over $[0,1]$, and \cite[Theorem 2.1]{DawsonLi} ensures that there is a unique pathwise strong solution to \eqref{SDEselec}. 
%\begin{remark}
It is worth noticing that in this case, since the constant process $1$ is solution to \eqref{SDEselec}, pathwise uniqueness entails in particular that if the boundary $1$ is reached then the process is absorbed at $1$. 
% When $f'(1-)=\infty$, the drift term in \eqref{SDEselec} is not Lipschitz over $[0,1]$. We shall see that the pathwise uniqueness of the solution to \eqref{SDEselec} does not always hold. 
%The existence and unicity of the solution is only guaranteed until the first time of hitting boundary $1$, $\tau_1:=\inf\{t\geq 0, X_t=1\}$. 
\end{remark}

We will be interested in cases where $f$ is not globally Lipschitz on $[0,1]$, i.e. $f'(1)=\infty$ and in possible extensions of the minimal process after time $\tau$. 
%\begin{definition} 
We call minimal $\Lambda$-Wright-Fisher process with selection governed by $f$ and denote by $(X_t^{\mathrm{min}},t\geq 0)$ the process solution to \eqref{SDEselec} that is stopped at the boundary after time $\tau:=\inf\{t>0: X_t \notin (0,1)\}$.
%\end{definition}
A consequence of Lemma \ref{minimalprocess} is that the minimal process, or minimal solution, is the unique solution to the following (stopped) martingale problem : 
\[(\mathrm{MP}): \forall g\in C^2_c([0,1]), \quad \left(g(X^{\mathrm{min}}_{t\wedge \tau})-\int_{0}^{t}\mathcal{A}^{\mathrm{s}}g(X^{\mathrm{min}}_{s\wedge \tau})\ddr s, t\geq 0\right) \text{ is a martingale,}\]
where $X^{\min}_{t}=1$ if $t\geq \tau_1=\tau_1\wedge \tau_0$ and $X^{\min}_{t}=0$ if $t\geq \tau_0=\tau_1\wedge \tau_0$.
%\left(g(X_t)-\int_{0}^{t}\mathcal{A}^{\mathrm{s}}g(X_s)\ddr s,t<\tau\right) \text{ is a local martingale}.
%The solution of $(\mathrm{MP})$ is denoted by $(X^{\mathrm{min}}_t,t\geq 0)$ and called the minimal solution. 
%We call\textit{ minimal process} and denote by $(X^{\mathrm{min}}_t,t\geq 0)$ the solution to \eqref{SDEselec} that is killed after having reached one of its boundaries. {\red{Namely, let $\partial\notin [0,1]$ be a cemetery point and $\tau:=\inf\{t>0, X^{\mathrm{min}}_t\notin (0,1)\}$, then for any $t\geq \tau$, $X^{\mathrm{min}}_t=\partial$.}} 
By extension of the minimal process solution of \eqref{SDEselec}, we mean here a process $(X_t,t\geq 0)$ with infinite life-time such that the process $(X_{t},t<\tau)$ has  same law as $(X^{\mathrm{min}}_{t},t<\tau)$. 

An important tool used repeatedly in the proofs is a comparison theorem due to Dawson and Li for solutions of \eqref{SDEselec}. Consider two Lipschitz generating functions $f_1$ and $f_2$ such that $f_1\leq f_2$ (thus by assumption $f'_i(1-)<\infty$ for $i=1,2$) and two initial values in $[0,1]$, $x_1\leq x_2$, then by \cite[Theorem 2.2]{DawsonLi} (here $\sigma=g'_1=g''_1=0$ and the conditions (2-a)-(2-d) are satisfied) if $(X^i_t(x_i),t\geq 0)$ is the solution of \eqref{SDEselec} with $f=f_i$ for $i=1,2$ and initial value $x_i$ then almost surely for all $t\geq 0$, $X^{1}_t(x_1)\leq X^{2}_t(x_2)$.
 
%\nonumber\\
%&+\mu(\mathbb{N})\int_{0}^{t}(f(X_{s-}(x))-X_{s-}(x))\ddr s.
%$\{X_t(x)=1\}$ is known
%Many works since the nineties have been devoted to incorporate other forces than resampling in the evolution of the types, such as selection and mutation. 
\subsubsection{$\Lambda$-Wright-Fisher processes without selection}
A fundamental property of $\Lambda$-Wright-Fisher processes without selection is their link with the processes called $\Lambda$-coalescents. Those processes are valued in the space of partitions of $\mathbb{N}$ and are evolving by multiple (not simultaneous) mergings of equivalence classes (called blocks). The $\Lambda$-coalescents can be thought as representing the genealogy backwards in time of the ancestral lineages in the population evolving by resampling, see \cite{MR1990057}. More backgrounds about $\Lambda$-coalescents are provided in Section \ref{backgroundEFC}. Let $(\Pi(t),t\geq 0)$ be a $\Lambda$-coalescent and $(\#\Pi(t),t\geq 0)$ be its block counting process. This is a Feller process valued in $\bar{\mathbb{N}}$
%with the same negative jumps as the block counting process of a $\Lambda$-coalescent and additional splitting into $k+1$ blocks with at constant rate $\mu(k)$ along each block independently: 
%is a Markov chain 
whose generator is $\mathcal{L}^{c}$ with for any $g:\mathbb{N}\mapsto \mathbb{R}$
\begin{equation}\label{coalpart}\mathcal{L}^c g(n):=\sum_{k=2}^{n}\binom{n}{k}\lambda_{n,k}(g(n-k+1)-g(n)), 
\end{equation}
\text{where } \begin{equation}\label{lambdank} \lambda_{n,k}:=\int_0^1 z^{k}(1-z)^{n-k}x^{-2}\Lambda(\ddr z) \text{ for } 2\leq k\leq n. \end{equation}
%and
%\begin{equation}\label{fragpart}\mathcal{L}^f g(n):=\sum_{k=1}^{\infty}n\mu(k)(g(n+k)-g(n))+n\mu(\infty)(g(\infty)-g(n)).
%\end{equation} 
%The process $(N^{(n)}_t,t\geq 0)$ has been intensively studied since the 2000s. Recall $\Phi$ in \eqref{phi2} and Condition \eqref{cdipsi}. For any $n\geq 2$, one has $\Phi(n)=\sum_{k=2}^{n} \binom{n}{k}\lambda_{n,k}(k-1)$ and 
Recall Condition \eqref{cdipsi}. Schweinsberg \cite{CDI} has established that \eqref{cdipsi} is necessary and sufficient for $\infty$ to be an entrance boundary of the block counting process, in this case although $\#\Pi(0)=\infty$, for any $t>0$ $\#\Pi(t)<\infty$  almost surely.  
%Recall Condition \eqref{cdipsi}. Schweinsberg \cite{CDI} has established that \eqref{cdipsi} is necessary and sufficient for the $\Lambda$-coalescent to come down from infinity, namely when $\#\Pi(0)=\infty$, for any $t>0$ $\#\Pi(t)<\infty$  almost surely. Equivalently $\infty$ is an entrance boundary of the block counting process.
%Set for any $n\geq 2$ \begin{equation}\label{phi}\Phi(n):=\sum_{k=2}^{n} \binom{n}{k}\lambda_{n,k}(k-1).
%\end{equation}  The pure $\Lambda$-coalescent comes down from infinity, i.e. started from $N_0=\infty$, $N_t<\infty$ for any $t>0$ almost surely if and only if \begin{equation}\label{cdi}\sum_{n=2}^{\infty}\frac{1}{\Phi(n)}<\infty. 
%\end{equation} 

The following identity, established for instance in \cite[Theorem 1, Equation (8)]{LGB2}, links the block counting process of a pure $\Lambda$-coalescent with the $\Lambda$-Wright-Fisher process without selection through a moment-duality relationship. For any $x\in [0,1]$ and $n\in \bar{\mathbb{N}}$, if $\#\Pi(0)=n$ then
\begin{equation}\label{dualclassic}\mathbb{E}[X_t(x)^n]=\mathbb{E}_n[x^{\#\Pi(t)}],
\end{equation}
where we denote by $\mathbb{E}_n$ the expectation conditionally given that $\#\Pi(0)=n$.
%establish the duality at the level of the semigroups \eqref{dualclassic}. 

The identity \eqref{dualclassic} has many important consequences. Firstly, since $X_t(x)$ is a bounded random variable, its law is entirely characterized by its moments and therefore the one-dimensional laws of $(\#\Pi(t),t\geq 0)$ are in one-to-one correspondence with those of the process $(X_t(x),t\geq 0)$. From a theoretical point of view one can see \eqref{dualclassic} as a representation of the semigroup of the process $(X_t(x),t\geq 0)$. Moreover, for any $n\in \mathbb{N}$, $\#\Pi(t)\leq \#\Pi(0)=n$ for all $t\geq 0$, $\mathbb{P}_n$-a.s. and letting $x$ approach $1$ provides the identity \[\underset{x\rightarrow 1-}{\lim}\mathbb{E}[X_t(x)^n]=\underset{x\rightarrow 1-}{\lim}\mathbb{E}_n[x^{\#\Pi(t)}]=\mathbb{P}_n(\#\Pi(t)<\infty)=1.\]
Since $x\mapsto X_t(x)$ admits left-limits, we see from the above convergence and Lebesgue's theorem that $X_t(1-)=1$ almost surely for all $t\geq 0$, so that the boundary $1$ is not an entrance of the $\Lambda$-Wright-Fisher process, but  is absorbing whenever it is reached. Letting $n$ go to infinity in \eqref{dualclassic} 
entails the identity
\begin{equation}\label{classicid}\mathbb{P}_x(\tau_1\leq t)=\mathbb{E}_{\infty}[x^{\#\Pi(t)}],\end{equation}
where $\tau_1:=\inf\{t>0: X_t(x)=1\}$ and $(\#\Pi(t),t\geq 0)$ starts from $\infty$ under $\mathbb{P}_\infty$. 
%
% $\mathbb{E}[x^{N_t}]=\mathbb{P}(X_t(x)=1)=\mathbb{P}_x(\tau_1\leq t)$ 
When Schweinsberg's condition \eqref{cdipsi} holds, $\#\Pi(t)<\infty$ a.s. for any $t>0$ and \eqref{classicid}  ensures that $\tau_1<\infty$ with positive probability. The identity \eqref{classicid} provides a representation of the cumulative distribution function of $\tau_1$ and in a dual way a representation of the entrance law at $\infty$ of the process $(\#\Pi(t),t\geq 0)$. Since $1$ is absorbing, the event $\{\tau_1<\infty\}$ coincides with the event of fixation at $1$: $\{\exists \ t_1\geq 0; \forall t\geq t_1; X_t(x)=1\}$. 
\subsubsection{Pure selection process} When there is no resampling in the population,  the frequency of the deleterious allele  $(X_t(x),t\geq 0)$ solves the ODE
\begin{equation}\label{ode}X_{t}(x)=x+\mu(\bar{\mathbb{N}})\int_{0}^{t}\big(f(X_s(x))-X_s(x)\big)\ddr s.
\end{equation}
%A study of the ODE \eqref{ode} shows that  has two distinct solutions started from $1$, $X_t:=1$ for all $t\geq 0$ and $X_t:=\underset{x\rightarrow 1-}{\lim} X_t(x)$ for all $t\geq 0$, if and only if 
%\begin{equation}\label{dynkin}
%\int^{1-} \frac{\ddr x}{x-f(x)}<\infty.
%\end{equation}
%
%Two  cases are possible for the boundary $1$ of $(X_t,t\geq 0)$. 
Equivalently, the map $(X_t(x),t\geq 0)$ satisfies for all $x\in (0,1)$ and $t\geq 0$, \begin{equation}\label{integralform} \int_{X_t(x)}^{x}\frac{\ddr u}{u-f(u)}=\mu(\bar{\mathbb{N}})t. 
\end{equation}
A study of \eqref{integralform} when $x$ tends to $1$, yields the following dichotomy. Either $\int^{1-}\frac{\ddr u}{f(u)-u}=\infty$, and in order for the integral in \eqref{integralform} to retain the value $\mu(\bar{\mathbb{N}})t$, we must have $X_t(1):=\underset{x\rightarrow 1-}{\lim}X_t(x)=1$, or else $\int^{1-}\frac{\ddr u}{f(u)-u}<\infty$ and we must have $X_t(1):=\underset{x\rightarrow 1-}{\lim}X_t(x)\in (0,1)$. In the latter case, the function $(X_t(1),t\geq 0)$ solves \eqref{ode} and starts from $1$. Hence, if $f(1)=1$ then \eqref{ode} has two distinct solutions started from $1$, $X_t:=1$ for all $t\geq 0$ and $X_t:=\underset{x\rightarrow 1-}{\lim} X_t(x)$ for all $t\geq 0$, if and only if 
\begin{equation}\label{dynkin}
\int^{1-} \frac{\ddr x}{x-f(x)}<\infty.
\end{equation}
The integral in \eqref{dynkin} converges for example when $x-f(x)\underset{x\rightarrow 1-}{\sim} \sigma (1-x)^{\alpha}$ with $\alpha\in (0,1)$ and $\sigma>0$. 
Note that if $f(1)<1$, i.e. $\mu(\infty)>0$, then \eqref{dynkin} clearly holds. 
%Note that generating functions $f$ such that $f(1)=1$ and \eqref{dynkin} holds, do exist; for instance if $x-f(x)\underset{x\rightarrow 1-}{\sim} \sigma (1-x)^{\alpha}$ with $\alpha\in (0,1)$ and $\sigma>0$.

Solutions to \eqref{ode} are well-known in the theory of branching processes.  Let $(N_t^{(n)},t\geq 0)$ be a continuous-time discrete-state branching process started from $n\in \mathbb{N}$ with offspring measure $\mu$. We refer for instance to Harris' book  \cite[Chapter V]{Harris}. Note that $\mu$ gives  no mass at $0$, namely there is no death in the process $(N_t^{(n)},t\geq 0)$ and its sample paths are almost surely non decreasing. The branching property of the process $(N_t^{(n)},t\geq 0)$ ensures that the boundary $\infty$ is absorbing whenever it is reached. One can identify the function $(X_t(x),t\geq 0)$ with the generating functions of $N^{(1)}_t$ at any time $t$, more precisely for any $x\in [0,1]$, any $n\in \mathbb{N}$ and any $t\geq 0$
\begin{equation} \label{dualbranching} X_t(x)^{n}=\mathbb{E}[x^{N_t^{(n)}}].\end{equation}
Letting $x$ go towards $1$ the identity above yields:
%\[
$\mathbb{P}_n(\zeta_\infty\leq t)=\underset{x\rightarrow 1-}{\lim} X_t(x) \text{ for any }t\geq 0,$
%\]
where $\zeta_\infty:=\inf\{t>0: N_{t-}^{(n)} \text{ or } N_t^{(n)}=\infty\}$. This is the standard method for studying the explosion of process $(N_t^{(n)},t\geq 0)$, see \cite[Chapter V, Theorem 9.1]{Harris}. In particular, the boundary $\infty$ is an exit for the branching process $(N_t^{(n)},t\geq 0)$ if and only Dynkin's condition \eqref{dynkin} holds.
%following the evolution of a proportion
%{\red Is $N^{(n)}_t$ absorbed at $\infty$ after explosion?}
%
%We also allow the function $f$ driving selection to be a defective generating function, namely such that $f(1)<1$. We interpret this as a possible jump to $\infty$ with rate $\mu(\infty):=1-\mu(\mathbb{N})$ and write $f(x):=\sum_{k=1}^{\infty}x^{k}\xi(k)$ where $\xi(\cdot):=\frac{\mu(\cdot)}{\mu(\bar{\mathbb{N}})}$ and $\mu$ is a finite measure on $\bar{\mathbb{N}}$. 
%As explained in the introduction, if letting $n$ to $\infty$ and $x$ to $1$ in the identity \eqref{dualmoment1} are legitimate limits, then one can deduce from \eqref{dualmoment1} some information on the process $(X_t(x),t\geq 0)$ at its boundary $1$. 
%We work with the convention \[x^{\infty}:=\underset{n\rightarrow \infty}{\lim} x^{n}=\mathbbm{1}_{\{x=1\}}, \text{ and } \underset{x\rightarrow 1-}{\lim}x^{n}=\mathrm{1}_{\{n<\infty\}}.\]

When resampling and selection are  both taken into account, the process solution to \eqref{SDEselec} may have more involved boundary conditions than what has just been seen for processes without selection or without resampling. In a similar fashion as for the pure selection process and for the pure resampling one,  we will see that the moments of $(X_t(x),t\geq 0)$ can be represented via certain continuous-time Markov chains, with values in $\bar{\mathbb{N}}$, whose jumps are mixture of those of the block counting process of a $\Lambda$-coalescent and those of an increasing branching processes. These processes appear when counting the number of blocks in certain exchangeable partition-valued processes. 
%\\
%
%We now give backgrounds on those processes.
%\newpage
\subsection{Backgrounds on exchangeable fragmentation coalescence (EFC) processes}\label{backgroundEFC}

\subsubsection{Definition and Poisson construction}
These processes, introduced by Berestycki in \cite{MR2110018}, are  Markov processes with state-space
%A simple EFC process is a partition-valued process $(\Pi(t),t\geq 0)$ which generalizes the class of $\Lambda$-coalescents by allowing each infinite block to fragmentate into sub-blocks of infinite size independently of each others at the same finite rate. Recall that we exclude from our study coalescent with a Kingman component and a star-shaped component. Coalescences are thus governed by a finite measure $\Lambda$ on $[0,1]$ with no mass at $0$ or $1$. . 
%We shall not introduce the whole framework of partition-valued processes here, for which we refer to Berestycki \cite{MR2110018} and to \cite{cdiEFC}, but focus on the  block counting process. 
$\mathcal{P}_\infty$, the space of partitions of $\mathbb{N}$. By convention, any partition $\pi\in \mathcal{P}_\infty$ is represented by the sequence (possibly finite) of its non-empty blocks $(\pi_i,i\geq 1)$ ordered by their least element. 
%For any $n\in \mathbb{N}$, $\pi_{|[n]}$ is the partition restricted to $[n]:=\{1,\cdots,n\}$: namely $\pi_{|[n]}=(\pi_1\cap [n],\pi_{2}\cap[n],\cdots)$. 
%  can be endowed with a compact metric $d(\pi,\pi'):=\max\{n\geq 1; \pi_{|[n]}=\pi'_{|[n]}\}^{-1}$.  
The number of non-empty blocks of a partition $\pi$ is denoted by $\#\pi$ and if $\#\pi<\infty$, then we set $\pi_{j}=\emptyset$ for any $j\geq \#\pi+1$. By definition an EFC process $(\Pi(t),t\geq 0)$  satisfies the following conditions
\begin{itemize}
\item[i)] for any time $t\geq 0$, the random partition $\Pi(t)$  is exchangeable, i.e. its law is invariant by the action of permutations with finite support,
\item[ii)] it evolves in time by merging of blocks  or fragmentation of an individual block into sub-blocks. 
\end{itemize}
Berestycki has established that any EFC process is characterized in law by two $\sigma$-finite exchangeable measures on $\mathcal{P}_\infty$, $\mu_{\mathrm{Coag}}$ the measure of coagulation and $\mu_{\mathrm{Frag}}$, that of fragmentation. We refer the reader to  \cite[Proposition 3]{MR2110018} for the form of those measures as mixtures of paint-box laws. We briefly sketch the Poissonian construction of EFC processes, see  \cite[Section 3.2]{MR2110018}. 

\noindent \textbf{Poisson construction}. Consider two independent Poisson Point Processes $\mathrm{PPP}_C$ and $\mathrm{PPP}_F$ respectively on $\mathbb{R}_+\times \mathcal{P} _\infty$ and $\mathbb{R}_+\times \mathcal{P} _\infty\times \mathbb{N}$ with intensity $\ddr t\otimes \mu_{\mathrm{Coag}}(\ddr \pi)$ and  $\ddr t\otimes \mu_{\mathrm{Frag}}(\ddr \pi)\otimes \#$ where $\#$ denotes the counting measure on $\mathbb{N}$. The process $(\Pi(t),t\geq 0)$ defined as follows is an EFC process: 
its initial value $\Pi(0)$ is an exchangeable random partition independent of $\mathrm{PPP}_F$ and $\mathrm{PPP}_C$ and
\begin{align*}
\Pi(t)&=\mathrm{Coag}\left(\Pi(t-),\pi^{c}\right) \text{ if } (t,\pi^{c}) \text{\,\, is an atom of } \mathrm{PPP}_C, \\
\Pi(t)&=\mathrm{Frag}\left(\Pi(t-),\pi^{f},j\right) \text{ if } (t,\pi^{f},j) \text{\,\, is an atom of } \mathrm{PPP}_F,
\end{align*}
where for any partitions $\pi$, $\pi^{c}$, $\pi^{f}$
\begin{align*}
\mathrm{Coag}(\pi,\pi^{c})&:=\{\cup_{j\in \pi^{c}_i}\pi_j, i\geq 1\} \text{ and } \mathrm{Frag}(\pi,\pi^{f},j):=\{\pi_j\cap\pi^{f}_i, i\geq 1;\ \pi_{\ell}, \ell\neq j\}^{\downarrow}
\end{align*}
where $\{\cdots\}^{\downarrow}$ means that we reorder the blocks by their least elements. 
%Among other results, Berestycki has shown that the $\mathcal{P}_\infty$-valued process $(\Pi(t),t\geq 0)$ is a c\`adl\`ag Feller process (for a certain metric $d$ on $\mathcal{P}_\infty$). 
%In particular, it is a strong Markov process.
%Many other interesting results are given in \cite{Berestycki04}.
%EFC processes form a broad class of partition-valued Markov processes. 

We focus on the subclass of simple EFC processes in which as in a $\Lambda$-coalescent,  there is no simultaneous multiple mergings, fragmentations occur at \textit{finite rate} and fragmentate any blocks into sub-blocks of infinite size (no formation of singletons).  In this setting,
\begin{itemize}
\item the coagulation measure $\mu_{\mathrm{Coag}}$ charges partitions with only one non-singleton block. If one associates to the atom  $(t,\pi^{c})$ of $\mathrm{PPP}_C$, a sequence of random variables $(X_i,i\geq 1)$ such that $X_i=1$ if $\{i\}\notin \pi^{c}$ and $X_i=0$ if $\{i\}\in \pi^{c}$, then 
%The pushforward random measure 
%$\mathrm{p}^{\mathrm{c}}$ on $\mathbb{R}_{+}\times [0,1]\times[0,1]^{\mathbb{N}}$ obtained by the map $(t,\pi^{c})\mapsto (t,z,X)$ with $z=\underset{n\rightarrow \infty}{\lim} \frac{1}{n}\sum_{i=1}^{n}X_i$ and $V=(V_i)_{i\geq 1)$ 
%, is a Poisson random measure with intensity $\ddr t\otimes \frac{\Lambda(\ddr z)}{z^{2}} \otimes \ddr \text{v}$, where $\ddr \text{v}$ corresponds to the law of an infinite sequence of i.i.d variables uniformly distributed over $[0,1]$. 
% with atoms (t,t
the random variables $(X_i,i\geq 1)$ are exchangeable and  mixture of i.i.d Bernoulli random variables whose parameter $\underset{n\rightarrow \infty}{\lim}\frac{1}{n}\sum_{i=1}^{n}X_i$ has for ``intensity" a measure of the form $\nu(\ddr z):=z^{-2}\Lambda(\ddr z)$ for some finite measure $\Lambda$ on $[0,1]$. We call $\Lambda$ the coalescence measure.
\item The fragmentation measure  $\mu_{\mathrm{Frag}}$ on $\mathcal{P}_\infty$ has finite mass, i.e. $\mu_{\mathrm{Frag}}(\mathcal{P}_\infty)<\infty$, and only charges partitions whose blocks are all of infinite size.  If one associates to each atom $(t,\pi^{f},j)$ of $\mathrm{PPP}_F$, the random variable $k:=\#\pi^{f}-1\in \bar{\mathbb{N}}$, one gets a Poisson point process on $\mathbb{R}_{+}\times \bar{\mathbb{N}}\times \mathbb{N}$ with intensity $dt\otimes \mu \otimes \#$, where $\mu$ is the image of $\mu_{\mathrm{Frag}}$ by the map $\pi\mapsto \#\pi-1$.   We call $\mu$ the splitting measure.
\end{itemize}
%In this setting, the block counting process has the same positive jumps as a pure discrete-state space branching process with no death and possibly a killing term $\mu(\infty)$, interpreted as a jump to $\infty$, and the same negative jumps as in a $\Lambda$-coalescent. More precisely
%
%Let $(\Pi(t),t\geq 0)$ be a simple EFC process without binary coalescences nor coalescences of all blocks at once. 
%a Kingman component and a star-shaped component. Coalescences are thus governed by a finite measure $\Lambda$ on $[0,1]$ with no mass at $0$ or $1$. 
For the sake of simplicity, we focus on initial partitions whose blocks are all infinite. In this case there is no singleton at any time in the system and the blocks of $\Pi(t)$ stay infinite at any time $t$. The infinitesimal time dynamics of $(\#\Pi(t),t\geq 0)$ are
%%is thus a mixture of those induced by the $\Lambda$-coalescent and those of a branching process. This
as follows. 
%
%For the sake of simplicity, we exclude from our study the initial exchangeable random partitions containing singletons. 
%Following Bertoin's terminology, see \cite[Chapter 2.3]{coursbertoin}, we call any exchangeable partition with no singleton blocks (namely with no \textit{dust}) \textit{a proper partition}. 
%Recall also that since $\Pi(t)$ is exchangeable for all $t\geq 0$, if $\#\Pi(t)<\infty$, then for all $j\leq \#\Pi(t)$, $\#\Pi_j(t)=\infty$ a.s. Namely $\Pi(t)$ is proper.}}%Consider a finite measure $\Lambda$ over $[0,1)$ with possibly a Dirac mass at $0$: $\Lambda(\{0\})=c_{\mathrm{k}}\geq 0$. The measure $\Lambda$ will govern the coalescences in the system. The coefficient $c_{\mathrm{k}}$ represents the Kingman part controlling binary coagulations. As fragmentation is assumed to occur at finite rate, the number of sub-blocks  created by a fragmentation is governed by the finite measure $\mu$ on $\bar{\mathbb{N}}$. We call $\mu$ the fragmentation measure. Once renormalized by its total mass $\mu(\bar{\mathbb{N}})$, the fragmentation measure corresponds simply to the law of the number of new fragments created after a fragmentation event. At this stage, we shall not make any assumption on the finite measure $\mu$ (no moment assumption).
%\vspace*{1mm}

\textit{Coalescence.}  Upon the arrival of an atom $(t,\pi^{c})$ of $\mathrm{PPP}_C$,  given $\#\Pi(t-)=n$, all blocks whose index $j\in [n]$ satisfies $X_{j}=1$ are merged. Given the parameter $z$ of the $X_i$'s, the number of blocks that merge at time $t$ has a binomial law with parameters $(n,z)$. Therefore, for any $k\in [|2,n|]$ the jump
%Let $k\in [|2,n|]$, any $k$ given blocks of $\Pi(t-)$ merge (to form one block) at rate $\int_{(0,1)}y^{k}(1-y)^{n-k}y^{-2}\Lambda(\ddr x)$. Therefore at time $t$
%\begin{equation}\label{multiplejump} 
$\#\Pi(t)=\#\Pi(t-)-(k-1)$
%\end{equation}
has rate $\binom{n}{k}\lambda_{n,k}$ where we recall $\lambda_{n,k}:=\int_{[0,1]}z^{k}(1-z)^{n-k}z^{-2}\Lambda(\ddr z).$
%Since there are $\binom{n}{k}$ possibilities for choosing $k$ blocks among $n$, the rate of the jump \eqref{multiplejump} is $\binom{n}{k}\int_{(0,1)}x^{k}(1-x)^{n-k}x^{-2}\Lambda(\ddr x)$.
%\end{itemize}
%Binary coalescences are hidden in the description above. They are governed by the Kingman parameter $\Lambda(\{0\})=:c_{\mathrm{k}}\geq 0$. We shall always assume that $\Lambda$ has no mass at $1$, so that it is impossible for all the blocks to coagulate simultaneously at once.
%we see that the process $(\#\Pi(t),t\geq 0)$ jumps from $n$ to $n-k+1$ for any $k\in [|2,n|]$ at rate $\binom{n}{k}\lambda_{n,k}$ where
%\[\lambda_{n,k}:=\int_{(0,1)}y^{k}(1-y)^{n-k}y^{-2}\Lambda(\ddr y)+\mathbbm{1}_{\{k=2\}}c_{\mathrm{k}}=\int_{[0,1)}y^{k}(1-y)^{n-k}y^{-2}\Lambda(\ddr y)\]
\vspace*{1mm}

\textit{Fragmentation}. Upon the arrival of an atom $(t,\pi^{f},j)$ of $\mathrm{PPP}_F$, given $\#\Pi(t-)=n$, and $j\leq n$, then the
$j^{\mathrm{th}}$-block is fragmentated into $k+1$ blocks where  $k=\#\pi^{f}\in \bar{\mathbb{N}}$. Therefore, at time $t$,
%\begin{equation}\label{fragmentation} 
$\#\Pi(t)=\#\Pi(t-)+k$.
%\end{equation}
Since there are $n$ blocks at time $t-$, the total rate at which such a jump occurs is  $n\mu(k)$ for any $k\in \bar{\mathbb{N}}$.

 Recall $\mathcal{L}^{c}$ defined in \eqref{coalpart}. 
\begin{lemma}[Proposition 2.11 in \cite{cdiEFC}]\label{genEFC} If $\Pi(0)$ has blocks of infinite size, then the process $(\#\Pi(t),t\geq 0)$ is a right-continuous process valued in $\bar{\mathbb{N}}$ and has the Markov property when lying on $\mathbb{N}$: i.e. setting $\zeta_{\infty}:=\inf\{t>0: \#\Pi(t-)=\infty \text{ or } \#\Pi(t)=\infty\}$, $(\#\Pi(t),t<\zeta_{\infty})$ is a continuous-time Markov chain whose generator  is $\mathcal{L}=\mathcal{L}^c+\mathcal{L}^f$ with $\mathcal{L}^{c}$ is given in \eqref{coalpart} and $\mathcal{L}^f$  is acting on any bounded function $g$ by
\begin{equation}\label{fragpart}\mathcal{L}^f g(n):=\sum_{k=1}^{\infty}n\mu(k)(g(n+k)-g(n))+n\mu(\infty)(g(\infty)-g(n)).
\end{equation}
\end{lemma} 
\subsubsection{A coupling in the number of initial blocks}

Several coupling procedures have been designed in \cite{cdiEFC} in order to study the process $(N_t,t\geq 0)$ started from $N_0=\infty$. As we will use it later, we recall \cite[Lemma 3.3 and Lemma 3.4]{cdiEFC} where it is explained how to define on the same probability space  as $(\Pi(t),t\geq 0)$, a monotone coupling of the block counting process $(\#\Pi(t),t\geq 0)$ in the initial values. For the sake of simplicity, we only consider EFC processes whose initial partitions have all blocks of infinite size. Let $n\in \bar{\mathbb{N}}$, set $(\Pi^{(n)}(t),t\geq 0)$ the process started from the first $n$ blocks of $\Pi(0)$, i.e. $\Pi^{(n)}(0)=\{\Pi_1(0),\ldots,\Pi_n(0)\}$, constructed in a Poisson way as follows:
\begin{align*}
\Pi^{(n)}(t)&=\mathrm{Coag}\left(\Pi^{(n)}(t-),\pi^{c}\right) \text{ if } (t,\pi^{c}) \text{ is an atom of } \mathrm{PPP}_C,\\
\Pi^{(n)}(t)&=\mathrm{Frag}\left(\Pi^{(n)}(t-),\pi^{f},j\right) \text{ if } (t,\pi^{f},j)\, \text{ is an atom of } \mathrm{PPP}_F.
% \text{ such that } \pi^{f}_{|[m]}\neq 1_{[m]},.
\end{align*}
The process $(\Pi^{(n)}(t),t\geq 0)$ follows the fragmentations and coagulations in the system restricted to the integers belonging to $\cup_{i=1}^{n}\Pi_i(0)$.

In the sequel, we write $(N_t^{(n)},t\geq 0):=(\#\Pi^{(n)}(t),t\geq 0)$ for the process counting the blocks of  $(\Pi^{(n)}(t),t\geq 0)$. Notice that by definition, $(\Pi^{(\infty)}(t),t\geq 0)$ coincides with $(\Pi(t),t\geq 0)$ and thus $(N_t^{(\infty)},t\geq 0)=(N_t,t\geq 0)$.The process $(N_t^{(n)},t\geq 0)$ is at the core of our study, and from now on we call it simply ``block counting process".
%They will allow us to take limits as $n$ goes to $\infty$ and $x$ goes to $1$ in duality relationships of the form \eqref{dualmoment1}. The following two lemmas taken from \cite{cdiEFC} will be therefore crucial in our study. To emphasize the fact that $\#\Pi(0)$ may be infinite, we set $(N_t^{(\infty)},t\geq 0):=(\#\Pi(t),t\geq 0)$.
\begin{lemma}[Monotonicity in the initial values, Lemma 3.4 in \cite{cdiEFC}]\label{monotonicity} For any $n\geq 1$,
\begin{center}
$N_t^{(n)}\leq N^{(n+1)}_t$ and $N_t^{(n)}\underset{n\rightarrow \infty}{\longrightarrow} N^{(\infty)}_t$ for all $t\geq 0$ a.s. \end{center}
Moreover, the process $(N^{(n)}_t,t<\zeta_{\infty})$ is Markovian and has the same law as $(N_t,t<\zeta_{\infty})$ when $N_0=n<\infty$.
%, that is to say it has law $\mathbb{P}_n$.
\end{lemma}
The next lemma ensures that one can approach from below the process $(N_t^{(n)},t\geq 0)$ by a non-decreasing sequence of non-explosive processes. For any $m\in \mathbb{N}$, set $\bar{\mu}(m):=\mu(\{m,\ldots, \infty\})$. 
\begin{lemma}[Lemma 3.8 in \cite{cdiEFC}] \label{Nnm} For any $n\in \bar{\mathbb{N}}$, there exists a non-decreasing sequence of processes $(N^{(n)}_m(t),t\geq 0)_{m\geq 1}$ started from $n\in \bar{\mathbb{N}}$ such that 
\begin{itemize}
\item[i)] for any $m\geq 1$, $(N^{(n)}_m(t),t\geq 0)$ has generator, the operator $\mathcal{L}^{m}$ acting on any function $g: \mathbb{N}\rightarrow \mathbb{R}_+$ as follows 
\begin{equation}\label{genNm}\mathcal{L}^{m}g(\ell):=\mathcal{L}^{c}g(\ell)+\ell\sum_{k=1}^{m}\mu_{m}(k)(g(\ell+k)-g(\ell)) \text{, for all } \ell \in \mathbb{N},
\end{equation}
where $\mu_{m}(k):=\mu(k)$ if $k\leq m-1$ and $\mu_m(m):=\bar{\mu}(m)$;
\item[ii)] for any $m\in \mathbb{N}$, the process $(N_m^{(n)}(t),t\geq 0)$ does not explode;
\item[iii)] almost surely for any $n\in \bar{\mathbb{N}}$, $m\in \mathbb{N}$ and all $t\geq 0$, $N^{(n)}_{m}(t)\leq N^{(n+1)}_{m}(t)$, 
%$N^{(\infty)}_m(t)\leq N^{(\infty)}_{m+1}(t)$ 
and  $$\underset{m\rightarrow \infty}{\lim} N^{(n)}_m(t)=N^{(n)}_t \text{a.s.}$$ 
\end{itemize}
\end{lemma}
\subsubsection{Summary of sufficient conditions for different boundary behaviors for $(N_t,t\geq 0)$}

Sufficient conditions for the process $(N_t,t\geq 0):=(\#\Pi(t),t\geq 0)$ to have boundary $\infty$ absorbing or not have been found in \cite{cdiEFC}. More precisely the process  can have the boundary $\infty$ exit, entrance or even regular for certain heavy-tailed splitting measures.
%, conditions for these three cases to occur have been designed in \cite{explosion}.   

%We do not repeat those results here, as they will appear in Section \ref{application} in a ``dual" way for the $\Lambda$-WF processes with selection. We shall however need the following result established in \cite{cdiEFC}. 
\begin{lemma}[Corollary 1.2-(2) in \cite{cdiEFC}]\label{key} Let $\lambda>0$. If the coalescence measure has no Kingman part, i.e. $\Lambda(\{0\})=0$, and $\mu(\infty)=\lambda$, then the process $(N_t,t\geq 0):=(\#\Pi(t),t\geq 0)$ has boundary $\infty$ as an exit, that is to say, for any $t\geq \zeta_{\infty}$, $N_t=\infty$ a.s.
where $\zeta_{\infty}:=\inf\{t>0: N_t=\infty\}$.
\end{lemma}
\begin{remark}\label{remkey} 
%Recall that we discard the case $\Lambda(\{0\})>0$. 
When $\Lambda(\{0\})=c_{\mathrm{k}}>0$, the process $(N_t^{\lambda},t\geq 0)$ comes down from infinity (i.e leaves $\infty$) if and only if $\frac{2\lambda}{c_{\mathrm{k}}}<1$, see Kyprianou et al. \cite[Theorem 1.1]{kyprianou2017} and \cite[Corollary 1.2-(1)]{cdiEFC}. Lemma \ref{key} is a key lemma in our construction of an extension getting out from $1$ of the minimal $\Lambda$-WF process with selection. 
\end{remark}
%We do not repeat those results here, as they will appear in Section \ref{application} in a ``dual" way for the $\Lambda$-WF processes with selection. We shall however need the following result established in \cite{cdiEFC}. 
The next lemmas provide sufficient condition for $\infty$ to be an entrance, an exit or a regular boundary. They will appear in Section \ref{application} in a ``dual" way for the $\Lambda$-WF processes with selection. 
\begin{lemma}[Theorem 3.4 in \cite{explosion}] \label{nonexplosioncriterion} Assume $\mu(\infty)=0$. If $\sum_{n=2}^{\infty}\frac{n}{\Phi(n)}\bar{\mu}(n)<\infty$, then $\infty$ is inaccessible for the process $(N_t,t\geq 0)$. Moreover, in this case,
\begin{itemize}
\item[i)] if $\sum_{n=2}^{\infty}\frac{1}{\Phi(n)}<\infty$ then $\infty$ is an entrance boundary,
\item[ii)] if $\sum_{n=2}^{\infty}\frac{1}{\Phi(n)}=\infty$ then $\infty$ is a natural boundary.
\end{itemize}
\end{lemma}
\begin{remark}\label{finitemeanremark} We recall that the sequence $(\Phi(n)/n, n\geq 2)$ is nondecreasing, see e.g. \cite[Lemma 2.1-(iv)]{limic2015}. Hence the sequence $(\frac{n}{\Phi(n)},n\geq 2)$ is bounded and when the splitting measure $\mu$ admits a first moment,  $\sum_{n=2}^{\infty}\frac{n}{\Phi(n)}\bar{\mu}(n)\leq \frac{2}{\Phi(2)}\sum_{n=2}^{\infty}\bar{\mu}(n)<\infty$, and the condition in Lemma \ref{nonexplosioncriterion} is fulfilled.
\end{remark}
%{\red Summarize sufficient conditions for different boundary behaviors for $N_t$ ?}
%
The next lemma provides a sufficient condition for $\infty$ to be an exit boundary.
\begin{lemma}[Theorem 3.1 in \cite{explosion}]\label{suffcondexitN}
%\begin{enumerate}
%\item 
Assume $\mu(\infty)=0$ and set $\ell :n\mapsto \sum_{k=1}^{n}\bar{\mu}(k)$.

If the map $n\mapsto \ell(n)$ satisfies the following condition $\mathbb{H}$:

\noindent $\mathbb{H}$: there exists a positive function $g$ on $\mathbb{R}_+$ eventually non-decreasing such that
\begin{center} $\ell(n)\geq g(\log n)\log n \text{ for large } n$, $\int^{\infty}\frac{1}{xg(x)}\ddr x<\infty$ \end{center} and \begin{equation}\label{suffcondentrancePhi} \underset{n\rightarrow \infty}{\lim}\  \frac{\Phi(n)}{n\ell(n)}=0,
\end{equation}
then the process $(N_t,t\geq 0)$ has $\infty$ as an exit boundary.
%\\
%\item  
%\\
%\end{enumerate}
\end{lemma} 
We now recall the classification of the boundary $\infty$ in some regularly varying cases found in \cite{explosion}. 
\begin{lemma}[Theorem 3.7 in \cite{explosion}]\label{stablefragtheorem}  Assume that $\Phi(n)\underset{n\rightarrow \infty}{\sim} dn^{1+\beta}$ with $d>0$ and $\beta\in (0,1)$ and $\mu(n)\underset{n\rightarrow \infty}{\sim} \frac{b}{n^{\alpha+1}}$ with $b>0$ and $\alpha\in (0,\infty)$. Then 
%$n\ell(n)\underset{n\rightarrow \infty}{\sim}\frac{b}{\alpha(1-\alpha)}n^{2-\alpha}$ and
\begin{itemize}
\item if $\alpha+\beta<1$, then 
%$(\#\Pi(t),t\geq 0)$ explodes and stays infinite almost surely:
%\vspace*{1mm}
%\begin{center}
$\infty$ is an exit boundary,
%\end{center}
\vspace*{1mm}
\item if $\alpha+\beta>1$, then 
%$(\#\Pi(t),t\geq 0)$ does not explode and 
%when started from a partition with infinitely many blocks of infinite size, 
%comes down from infinity instantaneously almost surely:
%\begin{center}
 $\infty$ is an entrance boundary,
% \end{center}
\vspace*{1mm}
\item if $\alpha+\beta=1$ and further,
%$\sigma:=\frac{b}{d}\frac{\pi}{\alpha \sin(\pi\alpha)}$ and $\theta:=\frac{b}{d}\frac{1}{\alpha(1-\alpha)}$ then, $\sigma>\theta$ and
\vspace*{2mm}
\begin{itemize}
\item if $b/d>\alpha(1-\alpha)$, then
% $(\#\Pi(t),t\geq 0)$ explodes and stays infinite almost surely:
%\vspace*{1mm}
%\begin{center}
$\infty$ is an exit boundary,
\vspace*{2mm}
%\end{center}
%\vspace*{2mm}
\item if $\frac{\alpha \sin(\pi \alpha)}{\pi}<b/d<\alpha(1-\alpha)$, then
% $(\#\Pi(t),t\geq 0)$ explodes and then comes down from infinity instantaneously almost surely:
%\vspace*{2mm}
%\begin{center}
$\infty$ is a regular boundary,
\vspace*{2mm}
%\end{center}
%\vspace*{2mm}
\item if $b/d<\frac{\alpha \sin(\pi \alpha)}{\pi}$, then 
%$(\#\Pi(t),t\geq 0)$ does not explode and
% when started from a partition with infinitely many blocks of infinite size,
%comes down from infinity instantaneously almost surely:
%\vspace*{2mm}
%\begin{center}
$\infty$ is an entrance boundary.
%\end{center}
\end{itemize}
\end{itemize}
\end{lemma}
%Recall the function $\Phi$, one has $\Phi(n)\sim \Phi(n)$ so that \eqref{cdipsi} is equivalent to \eqref{cdi}. 

%Similarly as in Section \ref{backgroundWF} where basic theoretical consequences of the duality relationship \eqref{dualclassic} linking entrance at $\infty $ for the block counting process $(N_t, t\geq 0)$ of a $\Lambda$-coalescent to the exit at $1$ for the dual neutral Wright-Fisher process, have been described, we shall now investigate which properties are entailed by duality for the block counting process of simple EFC processes and $\Lambda$-Wright-Fisher processes with selection.
\section{Theoretical study of the boundaries of $(N_t,t\geq 0)$ and $(X_t,t\geq 0)$}\label{sec}
%Before stating the main results, we give here a table that summarizes the behaviors found for the $\Lambda$-Wright-Fisher processes with multiple selection. 
% The two first cases are established in Section \ref{twofirstcases}. The two last cases are established in Section \ref{momentdualII}.

%We state now a fundamental  duality relationship between the process $(X_t(x),t\geq 0)$ and $(N_t^{(n)},t\geq 0)$ which was already observed by 
%
%mentioned in Lemma \ref{dualfinitemoment} that fulfills the duality relationship \eqref{dualmoment1} with the process $(X_t(x),t\geq 0)$ has the same infinitesimal dynamics as the block counting process of a simple EFC $(\Pi(t),t\geq 0)$ started from a partition with $n$ blocks with coalescence measure $\Lambda$ and splitting measure $\mu$.
Recall that by Lemma \ref{minimalprocess}, existence and uniqueness of the solution of \eqref{SDEselec} are only guaranteed until the process has reached one of its boundaries. It will be therefore necessary to construct the process $X$ whose boundary $1$ is an entrance or regular. Our study focuses on extensions of the minimal process at the barrier $1$. The boundary $0$ will be always absorbing. We are going to classify the boundaries by establishing some duality relationships between simple EFC processes and $\Lambda$-Wright-Fisher processes with selection. This will generalize the duality relationships \eqref{dualclassic} and \eqref{dualbranching} known for $\Lambda$-coalescents and branching processes. Imposing boundary conditions on the process $X$ will be necessary here. Recall the processes $(N^{(n)}_t,t\geq 0)$ described in Section \ref{backgroundEFC}. In order to establish the correspondences displayed in  Table \ref{correspondance}, we shall need two different moment-duality relationships established respectively in Section \ref{momentdual1} and Section \ref{momentdual2}, one for the non-stopped process $(N^{(n)}_t,t\geq 0)$ and the other for the process stopped at its first explosion time (whenever explosion occurs), $(N^{\mathrm{min},(n)}_t,t\geq 0):=(N_{t\wedge \zeta_\infty}^{(n)},t\geq 0)$.

In the same fashion as what we have explained in Section \ref{backgroundWF} for the classical $\Lambda$-Wright-Fisher process, we are going to classify the boundary $1$ of the process with selection according to the nature of the boundary $\infty$ of the simple EFC process.  

Recall $g_n(x)=g_x(n)=x^n$ for all $x\in [0,1]$ and $n\in \mathbb{N}$. We work with the convention \begin{equation}\label{convention}
%x^{\infty}:=
\underset{n\rightarrow \infty}{\lim} x^{n}=\mathbbm{1}_{\{x=1\}}, \text{ and } \underset{x\rightarrow 1-}{\lim}x^{n}=\mathrm{1}_{\{n<\infty\}}.
\end{equation}

\subsection{Moment-duality I}\label{momentdual1} Let $\Lambda$ be a coalescence measure such that $\Lambda(\{1\})=0$ and $\mu$ be a splitting measure with possibly a mass at $\infty$. For any $x\in [0,1]$, recall that $(X_t^{\mathrm{min}}(x),t\geq 0)$ is the minimal process which is absorbed at its boundary once it has reached it. 
% That is to say,  $X_t^{\mathrm{min}}(x)=X^{\mathrm{min}}_{t}(x)$ for any $t<\tau=\tau_0\wedge \tau_1$ and $X_t^{\mathrm{min}}(x)=1$ if $t\geq \tau=\tau_1$ and $X_t^{\mathrm{min}}(x)=0$ if $t\geq \tau=\tau_0$. 
%{\red Strictly speaking, $X^{\mathrm{min}}_\tau$ has not been defined. Do you mean $X^{\mathrm{min}}_\tau:=X^{\mathrm{min}}_{\tau-}$?  } 

%the strong solution to $\mathrm{MP}$, which is absorbed at its boundaries after having reached it. 

We first state a new duality relationship which holds for any $\Lambda$-Wright-Fisher process with frequency-dependent selection, subject to the condition $\Lambda(\{1\})=0$. No assumption on the generating function $f$ is made. 
%Recall $(N^{(n)}_t,t\geq 0)$ the process with infinite life-time, started from $n\in \bar{\mathbb{N}}$, with generator $\mathcal{L}$, described in Section \ref{backgroundEFC}. 

\begin{theorem}\label{thmmomentdual1} The Markov process $(X_t^{\mathrm{min}}(x),t\geq 0,x\in [0,1])$ satisfies the following property. For any $x\in [0,1]$ and any $n\in \mathbb{N}$
\begin{equation}\label{dualEFC}
\mathbb{E}[X^{\mathrm{min}}_t(x)^{n}]=\mathbb{E}[x^{N^{(n)}_t}].
\end{equation}
In particular, \[\mathbb{P}(X^{\mathrm{min}}_t(x)=1)=\underset{n\rightarrow \infty}{\lim}\mathbb{E}[x^{N^{(n)}_t}]= \mathbb{E}[x^{N^{(\infty)}_t}]\in [0,1]\] and the process $(X^{\mathrm{min}}_t(x),t\geq 0, x\in[0,1])$ gets absorbed at $1$ with positive probability if and only if the process $(N^{(\infty)}_t,t\geq 0)$ comes down from infinity. 
%\begin{itemize}
%\item[i)] If , then 
%%Putting this in other words, the selection is not strong enough for preventing fixation of the deleterious allele due to the random genetic drift governed by $\Lambda$. \\
%\item[ii)] If $(N^{(\infty)}_t,t\geq 0)$ stays infinite, then the process $(X^{\mathrm{min}}_t,t\geq 0)$ never hits $1$: the boundary $1$ is inaccessible. 
%%\item[iv)] If $(X^{\mathrm{min}}(x),t\geq 0)$ has its boundary $1$ regular absorbing or exit, then $\infty$ is regular reflecting or entrance for $(N_t^{(\infty)},t\geq 0)$.
%%. In this case the selection is so strong that the deleterious allele has no chance to be fixed. Note that $\mathbb{P}(X^{\mathrm{min}}_t=1)=0$ is equivalent to $\mathbb{P}(\tau_1\leq t)=0$ for all $t$, nameley $1$ is not accessible, since $X^{\mathrm{min}}$ is absorbed at $1$ whenever it reaches it.
%\end{itemize}
\end{theorem}
\begin{remark} Theorem \ref{thmmomentdual1} can be generalized to cases where $\Lambda$ has a mass at $0$,  see Remark \ref{remKing}. We shall however focus in the sequel on pure-jump $\Lambda$-Wright-Fisher processes, namely those with $\Lambda(\{0\})=0$, see Remark \ref{remwhynotkingman}.
\end{remark}
The following is the scheme of the proof of Theorem \ref{dualEFC}.
\begin{enumerate}
\item We first establish the duality \eqref{momentdual1} in the simpler case where $\mu$ has a second moment.
%$f$ is Lipschitz on $[0,1]$.
\item We then construct a sequence of processes $(X_t^{(m)},t\geq 0)$ solution to  \eqref{SDEselec} with $f$ replaced by a certain smooth generating functions $f_m$, that are approximating $f$. We then show that $(X_t^{(m)},t\geq 0,m\geq 1)$ converges pointwise almost surely towards a certain process $(X_t^{(\infty)},t\geq 0)$ which satisfies the targeted duality relationship  \eqref{dualEFC}. See Lemma \ref{approxXa}.
\item We finally identify the process $(X_t^{(\infty)},t\geq 0)$ with the process $(X_t^{\mathrm{min}},t\geq 0)$ which is absorbed at the boundaries after reaching it. See Lemma \ref{identificationXinfinity}.
\end{enumerate}
Our starting point is the following duality lemma which holds for $\Lambda$-WF process with selection whose function $f$ is smooth. 
\begin{lemma}\label{dualfinitemoment}  Assume $\mu(\infty)=0$ and $f''(1)=\sum_{n=2}^{\infty}n(n-1)\mu(n)<\infty$. Denote by $(X_t(x),t\geq 0)$ the solution to \eqref{SDEselec}. Then
for any $n\in \mathbb{N}$, and $x\in [0,1]$,
\begin{equation}\label{dualmoment1}\mathbb{E}[X_t(x)^n]=\mathbb{E}[x^{N^{(n)}_t}].
\end{equation}
\end{lemma}
\begin{remark}\label{remKing} This result has been observed by Gonz\'alez et al. \cite[Lemma 2]{Gonzalesetal} in their study of branching processes with interaction. The moment duality \eqref{dualmoment1} also holds true when $\Lambda$ gives mass at $0$. We check here that conditions for Ethier-Kurtz's theorem to hold are satisfied.  
\end{remark}
\begin{proof}
%We provide some explanations as \eqref{dualmoment1} will be crucial in the sequel. Identity \eqref{dualmoment1} can be proven by observing the duality relationship at the level of the generators. 
Recall $\mathcal{A}$ in \eqref{generator} and $\mathcal{L}^{c}$ in \eqref{coalpart} the generators respectively of the $\Lambda$-WF process with no selection and the block counting process of the $\Lambda$-coalescent. For any $x\in [0,1]$ and $n\in \mathbb{N}$, set $g(x,n)= g_x(n)=g_n(x):=x^{n}$. The following identity is well-known, see e.g. \cite[Lemma 2]{Gonzalesetal}, 
\begin{align}
\mathcal{A}g_n(x)&=\mathcal{L}^{c}g_x(n).
\end{align}
Recall $\mathcal{A}^{\mathrm{s}}$ in \eqref{generatorWFs}. By assumption since $\mu(\infty)=0$, one has $\text{for all } x\in [0,1]$, $\mathcal{A}^{\mathrm{s}}g(x)=\mathcal{A}g(x)+\mu(\mathbb{N})(f(x)-x)g'_n(x)$
%\begin{center} $(x,n)\mapsto g(x,n)=g_x(n)=g_n(x):=x^{n}$ for all $x\in [0,1]$ and all $n\in \mathbb{N}$\end{center}
and $\text{for all } x\in [0,1] \text{ and all } n\in \mathbb{N}$
\begin{center}
$\mu(\mathbb{N})(f(x)-x)g'_n(x)=n\sum_{k=1}^{\infty}(g_x(n+k)-g_x(n))\mu(k)$.
\end{center}
This provides the duality at the level of the generators \begin{equation}\label{expressionh}
	h(x,n):=\mathcal{A}^{\mathrm{s}}g_n(x)=\mathcal{L}g_x(n)= \sum_{k=2}^{n}(x^{n-k+1}-x^{n})\binom{n}{k}\lambda_{n,k}+\sum_{k=1}^{\infty}(x^{n+k}-x^{n})n\mu(k).
\end{equation} 

We now establish that the duality holds at the level of the semigroups. The process $(N_t^{(n)},t\geq 0)$ stays below a pure branching process $(Z_t,t\geq 0)$ with offspring measure $\mu$. By assumption $\mu$ admits a second moment, this entails in particular that for any time $T>0$, $\mathbb{E}(Z_T^{2})<\infty$, see e.g. \cite[Chapter 3, Corollary 1 page 111]{MR2047480}. In particular, $(Z_t,t\geq 0)$ does not explode which ensures that  $(N^{(n)}_t,t\geq 0)$ does not explode either. 
% and  the  drift term $x\mapsto f(x)-x$ is Lipschitz on $[0,1]$. 
%We now apply results of Ethier and Kurtz \cite[Corollary 4.15, page 196]{EthierKurtz}. In their notation $\alpha=\beta=0, \tau=\infty, \sigma=\sigma_{y}$. Let $y\in (0,1)$ and $m\in \mathbb{N}$. Set $\sigma_{y}:=\inf\{t\geq 0; X_t(x)\geq y\}$ and $\tau_m^{+}:=\inf\{t\geq 0; N_t^{(n)}\geq m\}$. By Dynkin's formula for continuous-time Markov chains, since the process does not explode, we see that
%\[\left(g(x,N_{t}^{(n)})-\int_{0}^{t}h(x,N_s^{(n)})\ddr s, t\geq 0\right)\]
%is a martingale.  In a similar way, It\^o's formula applied to the process $(X_t(x),t\geq 0)$ stopped at $\sigma_y$, entails that
%\[\left(g(X_{t\wedge \sigma_y}(x),n)-\int_{0}^{t\wedge \sigma_y}h(X_s(x),n)\ddr s, t\geq 0\right)\]
%is a martingale. Since $\mathcal{A}^{\mathrm{s}}g_n(x)=\mathcal{L}g_x(n)$,  and for any $T>0$, clearly $\sup_{s,t\leq T}|g(X_{s\wedge \sigma_{y}},N_{tb\wedge \tau_m^{+}})|\leq 1$ and
%\begin{align*}
%\sup_{s,t\leq T}|h(N_{t\wedge \tau_m^{+}},X_{s\wedge \sigma_{y}})|&\leq \sum_{k=2}^{m}\binom{m}{k}\lambda_{m,k}+m\mu(\mathbb{N}),
%\end{align*}
%the integrability assumption (4.50) of \cite[Theorem 4.11]{EthierKurtz} is verified, and Ethier and Kurtz's corollary states that
%\[\mathbb{E}[x^{N_{t\wedge \tau_m^{+}}^{(n)}}]-\mathbb{E}[X_{t\wedge \sigma_{y}}(x)^{n}]=\int_{0}^{t}\mathbb{E}\left[\big(\mathbbm{1}_{\{s\leq \tau_m^{+}\}}-\mathbbm{1}_{\{t-s\leq \sigma_y\}}\big)h\big(N_{s\wedge \tau_m^{+}}^{(n)},X_{(t-s)\wedge \sigma_y}(x)\big)\right]\ddr s.\]
By Dynkin's formula for continuous-time Markov chains, since $g_x$ is bounded and  the process does not explode, we see that
\[\left(g(x,N_{t}^{(n)})-\int_{0}^{t}h(x,N_s^{(n)})\ddr s, t\geq 0\right)\]
is a martingale.  Since $\mu$ admits in particular a  finite first moment, then the  drift term $x\mapsto f(x)-x$ is Lipschitz on $[0,1]$ and as noticed in Remark \ref{Lipschitzrem}, it ensures that there is only one solution to the equation \eqref{SDEselec} and that $1$ is absorbing for $X$ whenever it is reached. By applying It\^o's formula  to the process $(X_t(x),t\geq 0)$, we get that
\[\left(g(X_{t}(x),n)-\int_{0}^{t}h(X_s(x),n)\ddr s, t\geq 0\right)\]
is a local martingale. Since $g_n$ is bounded and $s\mapsto h(X_s(x),n)$ is bounded over finite time interval, the latter is a true martingale.
 
We now apply results of Ethier and Kurtz \cite[Theorem  4.4.11, page 192]{EthierKurtz}. Assume that $(X_t(x),t\geq 0)$ and $(N_t(x),t\geq 0)$ are independent. Provided that the integrability assumption (4.50) of \cite[ Theorem 4.4.11]{EthierKurtz} is verified, Ethier and Kurtz's theorem (with in their notation $\alpha=\beta=0$) states that for all $x\in [0,1], n\in \mathbb{N}$
%\begin{align*}
$\mathbb{E}[x^{N_{t}^{(n)}}]=\mathbb{E}[X_{t}(x)^{n}].$
%\end{align*}
We check now assumption (4.50). Let $T>0$, clearly $\sup_{s,t\leq T}|g(X_{s},N_{t})|\leq 1$ and it remains to see that the random variable $\sup_{s,t\leq T}|h(X_{s},N_{t})|$ is integrable. 
%Clearly, the process $(N_t,t\leq T)$ stays below a pure branching process $(Z_t,t\leq T)$ with no death and offspring measure $\mu$. 
Recall the expression of $h(x,n)$ in \eqref{expressionh}, one has for any $x\in[0,1]$ and $n\in \mathbb{N}$
\begin{align*}
|h(x,n)|&\leq \sum_{k=2}^{n}\binom{n}{k}\lambda_{n,k}+n\mu(\mathbb{N}). \label{integralform}
\end{align*}
Recall the form of the $\lambda_{n,k}$'s in \eqref{lambdank}. Simple binomial calculations provide that
\[\sum_{k=2}^{n}\binom{n}{k}\lambda_{n,k}=\int_{0}^{1}\big(1-(1-z)^n-nz(1-z)^{n-1}\big)z^{-2}\Lambda(\ddr z).\]
Setting $h(z)=1-(1-z)^n-nz(1-z)^{n-1}$, one checks $h'(u)=n(n-1)u(1-u)^{n-2}$ for all $u\in [0,1]$ and thus
\begin{align*}\sum_{k=2}^{n}\binom{n}{k}\lambda_{n,k}&=\int_{0}^{1}z^{-2}\Lambda(\ddr z)\int_{0}^{z}n(n-1)u(1-u)^{n-2}\ddr u\\
&\leq n(n-1)\int_{0}^{1}z^{-2}\Lambda(\ddr z)\int_{0}^{z}u\ddr u\\
&=\frac{\Lambda([0,1])}{2}n(n-1).
\end{align*}
%\[|\mathcal{A}g_n(x)|\leq n(n-1)(x^2+(1-x)^2)\Lambda([0,1])\leq C_1n(n-1),\] with $C_1$ a positive constant.  On the other hand, \[\mu(\mathbb{N})|f(x)-x|nx^{n-1}\leq \mu(\mathbb{N})f'(1)n\leq C_2n\] with $C_2$ a positive constant. 
Therefore, for all $x\in [0,1]$ and all $n\in \mathbb{N}$
%\[
$|h(x,n)|\leq \frac{\Lambda([0,1])}{2}n(n-1)+\mu(\mathbb{N})n$
%\]
and since
$N_t\leq Z_t\leq Z_T$ for any $t\leq T$, one has almost surely
\begin{align*}
\sup_{s,t\leq T}|h(X_{s},N_{t})|&\leq \frac{\Lambda([0,1])}{2}Z_T^2+\mu(\mathbb{N})Z_T:=\Gamma_T 
%\sum_{k=2}^{m}\binom{m}{k}\lambda_{m,k}+m\mu(\mathbb{N}).
\end{align*}
The random variable $\Gamma_T$ is integrable since $\mu$ admits a second moment.
% entail that if $(X_s(x),s\geq 0)$ and $(N_{s}^{(n)},s\geq 0)$ are taken independent then by combining the facts that $g$ is bounded, $(N_t^{(n)},t\geq 0)$ is non-explosive and $1$ is absorbing for $X$, one can check that for any fixed $t$, $\frac{\partial}{\partial s} \mathbb{E}[g(X_s(x),N_{t-s}^{(n)})]=0$ and hence $\mathbb{E}[g(x,N_t^{(n)})]=\mathbb{E}[g(X_{t}(x),n)].$ {\red{Clement:I will add some details here}}
\qed
\end{proof}
%We will see in the next theorem how to generalize Lemma \ref{dualfinitemoment} when the  splitting measure $\mu$ is more general (namely, without the first moment).
We now go to step (2). Let $\mu$ be a finite measure on $\bar{\mathbb{N}}$ and $f$ be its generating function (possibly defective). For any $m\in \mathbb{N}$, recall $\bar{\mu}(m)=\mu(\{m,\ldots, \infty\})$ and $\mu_m$ defined in Lemma \ref{Nnm}. For any $m\in \mathbb{N}$, set for any $x\in [0,1]$, 
\begin{equation}\label{fm} f_m(x):=\frac{1}{\mu_m(\mathbb{N})}\sum_{k=1}^{\infty}x^{k}\mu_{m}(k)=\frac{1}{\mu(\bar{\mathbb{N}})}\left(\sum_{k=1}^{m-1}x^{k}\mu(k)+x^{m}\bar{\mu}(m)\right). 
\end{equation}

%set $f_m(x):=\frac{1}{\mu_m(\bar{\mathbb{N}})}\sum_{k=1}^{m}x^{k}\mu_m(k)=\frac{1}{\mu(\bar{\mathbb{N}})}\left(\sum_{k=1}^{m-1}x^{k}\mu(k)+x^{m}\bar{\mu}(m)\right)$.
\begin{lemma}\label{approxXa}  Let $(X_t^{(m)}(x),t\geq 0)$ be the unique strong solution to the SDE
\begin{align}\label{sdem}X^{(m)}_t(x)=x+\int_{0}^{t}\int_{0}^{1}\int_{0}^{1}&z\left(\mathbbm{1}_{\{v\leq X^{(m)}_{s-}(x)\}}-X^{(m)}_{s-}(x)\right)\bar{\mathcal{M}}(\ddr s,\ddr v, \ddr z) \nonumber\\
&+\mu_m(\mathbb{N})\int_{0}^{t}(f_m(X^{(m)}_{s-}(x))-X^{(m)}_{s-}(x))\ddr s.\end{align}
Then, for any $m\geq 1$, almost surely for all $t\geq 0$,
\begin{equation}
\label{order}
X_t^{(m+1)}\leq X_t^{(m)} \text{ for all } t\geq 0 \text{ almost surely}
\end{equation}  
and the limiting process $(X^{(\infty)}_t,t\geq 0)$ defined by
%following pointwise limit exists almost surely:
$X^{(\infty)}_t:=\underset{m\rightarrow \infty}{\lim}\!\! \downarrow X_t^{(m)} \text{ for all }t\geq 0,$
satisfies the duality relationship: for all $x\in [0,1)$, $n\in \bar{\mathbb{N}}$,
\begin{equation}\label{dualinfiny}\mathbb{E}[X^{(\infty)}_t(x)^{n}]=\mathbb{E}[x^{N^{(n)}_t}].
\end{equation} 
\end{lemma}
%Note that $f'_m(1-)<\infty$, therefore applying \eqref{dualmoment1}
%we have 
%\begin{equation}\label{dualm}\mathbb{E}[x^{N_m^{(n)}(t)}]=\mathbb{E}[(X_t^{(m)}(x))^{n}]
%\end{equation}
\begin{proof} 
Note that for any $m\geq 1$, $f'_m(1-)<\infty$,  therefore the SDE \eqref{sdem} admits a unique strong solution. Moreover, $f_m(x)-f_{m+1}(x)=\bar{\mu}(m+1)(x^{m}-x^{m+1})\geq 0$ for any $x\in [0,1]$. By  the comparison theorem for SDEs with jumps, we see that \eqref{order} holds true and thus $(X^{(\infty}_t,t\geq 0) $ is well defined.

Recall Lemma \ref{Nnm} and consider a monotone sequence of processes $(N^{(n)}_{m}(t),t\geq 0)$ with generator $\mathcal{L}^m$ defined in \eqref{genNm}, with splitting measure $\mu_m$
such that $\mu_{m}(k)=\mu(k)$ if $k\leq m-1$ and $\mu_m(m)=\bar{\mu}(m)$. Since $\mu_m$ admits a second moment, the duality relationship \eqref{dualmoment1} holds and we have 
\begin{equation}\label{dualm}\mathbb{E}[x^{N_m^{(n)}(t)}]=\mathbb{E}[(X_t^{(m)}(x))^{n}].
\end{equation}
By Lemma \ref{Nnm}, $N_m^{(n)}(t)$ converges almost surely towards $N^{(n)}_t$ as $m$ goes to $\infty$. Hence, the identity \eqref{dualinfiny} follows readily by taking limit on $m$.
\qed
\end{proof}
%The duality relationship \eqref{dualEFC} is shown by letting $m$ to $\infty$ in \eqref{dualm} and $n$ to $\infty$. The other two statements are immediate by the fact that $(X^{\mathrm{min}}_t,t\geq 0)$ is absorbed at $1$.

\begin{lemma}\label{identificationXinfinity} The limit process $(X^{(\infty)}_t,t\geq 0)$ has the same law as   $(X^{\mathrm{min}}_t,t\geq 0)$ the minimal solution of \eqref{SDEselec}. 
\end{lemma}
\begin{proof}
We first precise the behavior of the process $(X^{(\infty)}_t(x),t\geq 0)$  when it reaches one of its boundaries $1$ or $0$. For the boundary $1$, since $\mu_m$ has a finite mean, Corollary 1.4 and Remark 1.5 in \cite{cdiEFC} apply and entail that the non-explosive processes $(N^{(\infty)}_m(t),t\geq 0)$ comes down from infinity. As observed in Section \ref{backgroundWF}, this is equivalent to the fact that the dual process $(X_t^{(m)},t\geq 0)$ is getting absorbed at $1$ in finite time with positive probability. Note that by \eqref{order}, $\tau_1^{(m)}:=\inf\{t\geq 0: X_t^{(m)}=1\}$  verifies $\tau_1^{(m+1)}\geq \tau_1^{(m)}$ and therefore $\tau_1^{(\infty)}:=\inf\{t\geq 0: X^{(\infty)}_t=1\}\geq \underset{m\rightarrow \infty}{\lim}\uparrow \tau_1^{(m)}$. Hence, since $1$ is absorbing for $(X_t^{(m)},t\geq 0)$ and $\tau_1^{(\infty)}\geq \tau_1^{(m)}$ a.s. If $\tau_1^{(\infty)}<\infty$ then 
$X^{(\infty)}_{t+\tau_1^{(\infty)}}=\underset{m\rightarrow \infty}{\lim} X^{(m)}_{t+\tau_1^{(\infty)}}=1 \text{ for any } t\geq 0$ a.s. On the event $\{\tau_1^{\infty}<\infty\}$, 
%(and in particular if this has positive probability) 
$(X^{(\infty)}_t,t\geq 0)$ is absorbed at its boundary $1$ in finite time, hence for any $t\geq 0$, $\mathbb{P}(X^{(\infty)}_t=1)=\mathbb{P}(\tau_1^{(\infty)}\leq t)$ and 
\[\mathbb{P}(\tau_1^{(m)}\leq t))=\mathbb{P}(X^{(m)}_t=1)\underset{m\rightarrow \infty}{\longrightarrow} \mathbb{P}(X^{(\infty)}_t=1)=\mathbb{P}(\tau_1^{(\infty)}\leq t),\]
and thus $\tau_1^{(\infty)}=\underset{m\rightarrow \infty}{\lim}\uparrow \tau_1^{(m)}$ a.s.

For the boundary $0$, by taking $x=0$ in the duality relationship \eqref{dualinfiny}, we see that $X_t^{(\infty)}(0)=0$ a.s. Hence, $0$ is  necessarily absorbing.

We establish now that the process $(X^{(\infty)}_t,t>0)$ stopped at its first hitting time of the boundaries has the same law as the minimal process. 
%solution to \eqref{SDEselec}.
This follows from uniform convergence of the generators. Recall $\mathcal{A}^{\mathrm{s}}$ the generator of the minimal solution $(X^{\mathrm{min}}_t,t\geq 0)$ to \eqref{SDEselec}. For any $g\in C^2_c((0,1))$, $\mathcal{A}^{\mathrm{s}}g(x)=\mathcal{A}g(x)+\mu(\bar{\mathbb{N}})(f(x)-x)g'(x)$. Let $\mathcal{A}^{\mathrm{s},(m)}$ be the generator of $(X_t^{(m)},t\geq 0)$. Since the jump parts of $\mathcal{A}^{\mathrm{s}}$ and $\mathcal{A}^{\mathrm{s},(m)}$ are the same, and $\mu(\bar{\mathbb{N}})=\mu_m(\bar{\mathbb{N}})$ for any $m$, we have that
\begin{align}\label{diffgen}
||\mathcal{A}^{\mathrm{s},(m)}g-\mathcal{A}^{\mathrm{s}}g||_{\infty}&=\mu(\bar{\mathbb{N}})\underset{x\in (0,1)}{\sup}|(f_m(x)-f(x))g'(x)|.
\end{align}
For any $x\in [0,1]$,
\begin{align*}
0\leq f_m(x)-f(x)&=\frac{1}{\mu(\bar{\mathbb{N}})}\sum_{k\geq m}(x^{m}-x^{k})\mu(k)\\
&\leq  \frac{1}{\mu(\bar{\mathbb{N}})}x^{m}\sum_{j\geq 0}(1-x^{j})\mu(j+m)\leq \frac{\bar{\mu}(m)}{\mu(\bar{\mathbb{N}})}x^m\leq x^m.
\end{align*}
Hence we see from \eqref{diffgen} that $||\mathcal{A}^{\mathrm{s},(m)}g-\mathcal{A}^{\mathrm{s}}g||_{\infty}\leq \sup_{x\in (0,1)}|x^{m}g'(x)|$. Since by assumption $g'$ has a compact support on $(0,1)$ and for any $x\in (0,1)$, $x^{m}\underset{m\rightarrow \infty}{\longrightarrow}0$,  one has
\begin{equation}\label{convunifAm} ||\mathcal{A}^{\mathrm{s},(m)}g-\mathcal{A}^{\mathrm{s}}g||_{\infty}\underset{m\rightarrow \infty}{\longrightarrow} 0.
\end{equation}
Moreover for large enough $m\geq 1$, $||\mathcal{A}^{\mathrm{s},(m)}g||_\infty\leq 1+||\mathcal{A}^{\mathrm{s}}g||_\infty$ and since $X_s^{(m)}\underset{k\rightarrow \infty}{\longrightarrow} X_s^{(\infty)}$ a.s. for any $s\geq 0$, $\mathcal{A}^{\mathrm{s},(m)}g(X_s^{(m)})\underset{m\rightarrow \infty}{\longrightarrow} \mathcal{A}^{\mathrm{s}}g(X_s^{(\infty)})$ a.s. for any $s\geq 0$. Let $0\leq t_1\leq t_2\leq \ldots \leq t_n\leq s<t$ and $f_1,\ldots, f_n$ some bounded and continuous functions. By Lemma \ref{minimalprocess} applied to the process $(X_t^{(m)},t\geq 0)$, for any $g\in C_c^{2}([0,1])$, the process
\[\left(g(X_{t\wedge \tau^{(m)}}^{(m)})-\int_{0}^{t}\mathcal{A}^{\mathrm{s},(m)}g(X_{s\wedge \tau^{(m)}}^{(m)})\ddr s, t\geq 0\right)\]
is a martingale.  By applying Lebesgue's theorem, we have that 
\begin{align*}
&\mathbb{E}_x\left[\left(g(X_{t\wedge \tau^{(\infty)}}^{(\infty)})-g(X_{s\wedge \tau^{(\infty)}}^{(\infty)})-\int_{s}^{t}\mathcal{A}^{\mathrm{s}}g(X_{r\wedge \tau^{(\infty)}}^{(\infty)})\ddr r\right)\prod_{i=1}^{n}f_i(X_{t_i\wedge \tau^{(\infty)}}^{(\infty)})\right]\\
&=\underset{m\rightarrow \infty}{\lim} \mathbb{E}_x\left[\left(g(X_{t\wedge \tau^{(m)}}^{(m)})-g(X_{s\wedge \tau^{(m)}}^{(m)})-\int_{s}^{t}\mathcal{A}^{\mathrm{s},(m)}g(X_{r\wedge \tau^{(m)}}^{(m)})\ddr r\right)\prod_{i=1}^{n}f_i(X_{t_i\wedge \tau^{(m)}}^{(m)})\right]\\
&=0.
\end{align*}
This shows that %(4) We establish here that for any $\lambda<\lambda'$, $X_t^{\lambda}(x)\geq X_t^{\lambda'}(x)$.  What about $X_t^{\lambda}(1)$?
%Recall that 
the limiting process $(X^{(\infty)}_t,t\geq 0)$ stopped at time $\tau^{(\infty)}$ solves the martingale problem $\mathrm{(MP)}$. Lemma \ref{minimalprocess} ensures that there is a unique solution to 
$(\mathrm{MP})$, therefore $(X^{(\infty)}_t,t<\tau^{(\infty)})$ has the same law as the minimal process.
%{\red Is $X^{(\infty)}$ a solution to SDE (2.6)? I think it is a strong solution to (2.6), we should be able to show that $X^{(\infty)}_t=X^{\mathrm{min}}_t$  for all $t<\tau$ a.s, maybe it is enough to say that $X^{(\infty)}_t$ is adapted to the filtration of $\mathcal{M}$ and is a weak solution? I dont think it is that important. }
\qed
\end{proof}

%To emphasize the fact that whenever the boundary $1$ is reached the process $(X^{\mathrm{min}},t\geq 0)$ is absorbed at $1$, we rebaptize the process as $(X_t^{\mathrm{min}},t\geq 0)$. 
%In view of Lemma \ref{identificationXinfinity}, we rebaptize the limit process $(X^{(\infty)}_t,t\geq 0)$, and call it $(X^{\mathrm{min}}_t,t\geq 0)$ to emphasize that it is absorbed at $1$ whenever it reaches $1$.\\

\noindent \textbf{Proof of Theorem \ref{thmmomentdual1}.} This is the combination of Lemma \ref{approxXa} and Lemma \ref{identificationXinfinity}.\qed

We deduce now from Theorem \ref{thmmomentdual1} two important results for the block counting process of a simple EFC process. Those results were left unaddressed in \cite{cdiEFC} and \cite{explosion}. Recall Lemma \ref{monotonicity} and that when $n<\infty$, $(N^{(n)}_t,t<\zeta_{\infty})$ has the same law as $(\#\Pi(t),t<\zeta_{\infty})$ where $(\Pi(t),t\geq 0)$ is a simple EFC process started from a partition with $n$ blocks, and whose coalescence measure is $\Lambda$ and splitting measure is $\mu$.
\begin{theorem}[Markov property of $(\#\Pi(t),t\geq 0)$]\label{Markovblockcounting} Let $(\Pi(t),t\geq 0)$ be a simple EFC process whose coalescence measure is $\Lambda$ and splitting measure is $\mu$. The block counting process $(N_t,t\geq 0):=(\#\Pi(t),t\geq 0)$ with state-space  $\bar{\mathbb{N}}$ is a Markov process satisfying the Feller property.
\end{theorem}
\begin{proof}
%We borrow here the notation of \cite{cdiEFC}. 
Let $(\Pi^{(n)}(t),t\geq 0)$ be the process following the coalescences and fragmentations involving only the first $n$ initial blocks of $(\Pi(t),t\geq 0)$, see Section \ref{backgroundEFC} and \cite[Lemma 3.3 and Lemma 3.4]{cdiEFC}. In particular, $\Pi^{(n)}(0):=(\Pi_1(0),\ldots,\Pi_n(0))$. When $n=\#\Pi(0)$, both processes $(N_t^{(n)},t\geq 0):=(\#\Pi^{(n)}(t),t\geq 0)$ and $(\Pi(t),t\geq 0)$ coincide. Hence, by Theorem \ref{thmmomentdual1} and the duality relationship \eqref{dualEFC}, for any $x\in [0,1]$ and any $t\geq 0$, $\mathbb{E}[x^{\#\Pi(t)}]=\mathbb{E}[X_t^{\mathrm{min}}(x)^{\#\Pi(0)}]$, where $(X_t^{\mathrm{min}}(x),t\geq 0)$ is the unique solution to \eqref{SDEselec} absorbed at the boundaries. Without loss of generality, we enlarge the probability space on which $(\Pi(t),t\geq 0)$ is defined by assuming that $(X_t^{\mathrm{min}}(x),t\geq 0)$ is also defined on it and is independent of $(\Pi(t),t\geq 0)$. By the Markov property of $(X_t^{\mathrm{min}}(x),t\geq 0)$, for any $s,t\geq 0$, conditionally given $X^{\mathrm{min}}_s(x)$, the random variable $X_{s+t}^{\mathrm{min}}(x)$ has the same law as $X_t^{\mathrm{min}}(X_{s}^{\mathrm{min}}(x))$. Hence, setting $n=\#\Pi(0)$, 
\begin{align*}\mathbb{E}[x^{\#\Pi(s+t)}]&=\mathbb{E}[x^{\#\Pi^{(n)}(s+t)}]=\mathbb{E}[X_{s+t}^{\mathrm{min}}(x)^{n}]=\mathbb{E}[\mathbb{E}[X_{t}^{\mathrm{min}}(X_{s}^{\mathrm{min}}(x))^{n}|X_{s}^{\mathrm{min}}(x)]]\\
&=\mathbb{E}[X_s^{\mathrm{min}}(x)^{\#\Pi^{(n)}(t)}]=\mathbb{E}[x^{\#\Pi^{(\#\Pi^{(n)}(t))}(s)}].
\end{align*}
We see finally that $\#\Pi(t+s)$  has the same distribution as $\#\Pi^{(\#\Pi(t))}(s)$. The process $(\#\Pi(t),t\geq 0)$ is therefore  Markovian. We now establish the Feller property. Recall $\bar{\mathbb{N}}$ the one-point compactification of $\mathbb{N}$ and $C(\bar{\mathbb{N}})$ the space of continuous functions defined on $\bar{\mathbb{N}}$.  By the duality relationship \eqref{dualEFC}, we see that 
\begin{equation}\label{entrancelawN}
\underset{n\rightarrow \infty}{\lim} \mathbb{E}_n(x^{\#\Pi(t)})=\mathbb{P}(X_t^{\mathrm{min}}(x)=1)=\mathbb{E}[x^{N_t^{(\infty)}}].
\end{equation} Recall $g_x(n)=x^{n}$ and set $x^{\infty}:=\underset{n\rightarrow \infty}{\lim} x^{n}=\mathbbm{1}_{\{x=1\}}$. We have just established in \eqref{entrancelawN} that $n\mapsto \mathbb{E}_n(g_x(\#\Pi(t)))$ is continuous at $\infty$. The subalgebra $A$ of $C(\bar{\mathbb{N}})$ generated by the linear combinations of the maps $\{n\mapsto g_x(n), x\in [0,1]\}$,  is separating $C(\bar{\mathbb{N}})$. Moreover, for any $n\in \bar{\mathbb{N}}$, there is a function $g$ in $A$ such that $g(n)\neq 0$. By the Stone-Weierstrass theorem, $A$ is dense in $C(\bar{\mathbb{N}})$ for the uniform norm, since $n\mapsto \mathbb{E}_n(g(\#\Pi(t))) $ is continuous for any $g\in A$, this holds true for any $g\in C(\bar{\mathbb{N}})$ and the semigroup of  $(\#\Pi(t),t\geq 0)$ maps $C(\bar{\mathbb{N}})$ to $C(\bar{\mathbb{N}})$. It remains to verify the continuity of the semigroup at $0$. This is a direct application of the duality relationship \eqref{momentdual1}, since $X^{\mathrm{min}}$ is right-continuous. Finally the process $(\#\Pi(t),t\geq 0)$ is Feller. \qed
\end{proof}
Recall the first hitting times of the boundary $\tau_i:=\inf\{t\geq 0; X_t^{\mathrm{min}}=i\}$ for $i\in \{0,1\}$.
\begin{theorem}[Recurrence of $(N_t,t\geq 0)$] \label{rec} If the process $(N_t^{(\infty)},t\geq 0)$ comes down from infinity (i.e. infinity is non-absorbing), then it is positive recurrent and has a stationary distribution whose generating function is $\varphi:x\in [0,1] \mapsto \mathbb{P}_x(\tau_1<\tau_0)$.
\end{theorem}
\begin{remark} A sufficient condition for coming down from infinity of simple EFC processes is given in \cite[Theorem 1.1]{cdiEFC}. In the notation of \cite{cdiEFC}, if $\theta^{\star}<1$, then the block counting process comes down from infinity and  by Theorem \ref{rec}, $(N_t,t\geq 0)$ is positive recurrent.
\end{remark}
\begin{remark} The stationary distribution of $(N_t,t\geq 0)$ is  carried over $\mathbb{N}$ if and only if the generating function $\varphi$ is non-defective i.e. $\varphi(x)=\mathbb{P}_x(\tau_1<\tau_0)\underset{x\rightarrow 1-}{\longrightarrow} 1$. 
%The latter is equivalent to the condition that boundary $1$ is regular for itself for $X$.
\end{remark}
\begin{proof}
Denote by $(Y_t(x),t\geq 0)$ the neutral $\Lambda$-Wright-Fisher process, that is to say, the unique solution to \eqref{SDEselecintro}. As recalled in the introduction, when the $\Lambda$-coalescent process comes down, the process $(Y_t(x),t\geq 0)$ has a positive probability to be absorbed at $0$. Moreover, for any $m\geq 1$, $f_m(x)-x\leq 0$, and by applying the comparison theorem between $(X_t^{(m)}(x),t\geq 0)$ and $(Y_t(x),t\geq 0)$, we get $X_t^{(m)}(x)\leq Y_t(x)$ for all $t\geq 0$ almost surely. Hence, letting $m$  converge to $\infty$, we see that $X^{\mathrm{min}}_t(x)\leq Y_t(x)$ for all $t\geq 0$ almost surely, which ensures that the process $(X^{\mathrm{min}}_t(x),t\geq 0)$ hits $0$ with positive probability. Moreover, according to Theorem \ref{thmmomentdual1}-i), since $(N_t^{(\infty)},t\geq 0)$ comes down from infinity, the process $(X^{\mathrm{min}}_t(x),t\geq 0)$ hits $1$ with positive probability. Finally, since process $(X_t^{\mathrm{min}},t\geq 0)$ is a positive supermartingale, it converges almost surely as $t$ goes to $\infty$ towards one of its absorbing  boundaries $0$ or $1$. On the event $\{\tau_1<\tau_0\}$, $\underset{t\rightarrow \infty}{\lim} X_t^{\mathrm{min}}=1$ a.s and similarly on the event $\{\tau_0<\tau_1\}$, $\underset{t\rightarrow \infty}{\lim} X_t^{\mathrm{min}}=0$ a.s.  Thus, by the duality relationship \eqref{dualabso},
$\underset{t\rightarrow \infty}{\lim} \mathbb{E}[x^{N_t^{(n)}}]=\mathbb{P}_x(\tau_1<\tau_0).$  
%Since $\underset{x\rightarrow 1-}{\lim}\mathbb{P}_x(\tau_1<\tau_0)=1$, the limiting distribution of $(N_t,t\geq 0)$ is carried over $\mathbb{N}$.
\qed
\end{proof}
\subsection{Moment duality II}\label{momentdual2}
The duality relation \eqref{dualEFC} somehow maps the entrance laws of the process  $(N^{(n)}_t,t\geq 0)$ to the exit laws of $(X^{\mathrm{min}}_t(x),t\geq 0)$. We now wish to get a more complete correspondence between the boundaries, and look for the process $(X_t(x),t\geq 0)$ that is dual to $(N_t,t\geq 0)$ when the boundary $\infty$ of the latter is either an exit or regular \textit{absorbing}. Recall that by  \textit{regular absorbing},  we mean that the boundary $\infty$ is reached and sample paths of the process are stopped at it. The assumption $\Lambda(\{0\})=0$ will play an important role in this section. 
We consider from now on such a coalescence measure $\Lambda$ (also with no mass at $1$), a splitting measure $\mu$ without mass at $\infty$ and its associated probability generating function $f$.
%We now go in the other way round and establish  properties of $(N_t,t\geq 0)$ from the process $(X^{\mathrm{min}},t\geq 0)$.
\begin{theorem}\label{thmmomentduality2} There exists a Feller process $(X_t^{\mathrm{r}}(x),t\geq 0, x\in [0,1])$ extending the minimal process such that for any $t\geq 0$ and any $n\in \mathbb{N}$, 
\begin{equation}\label{dualityEFCref}
\mathbb{E}[X^{\mathrm{r}}_t(x)^n]=\mathbb{E}[x^{N^{\mathrm{min},(n)}_t}] \text{ for any } x\in [0,1) \text{ and } \mathbb{E}[X^{\mathrm{r}}_t(1)^n]=\mathbb{P}(N^{\mathrm{min},(n)}_t<\infty),\end{equation} where $(N^{\mathrm{min},(n)}_t,t\geq 0):=(N_{t\wedge \zeta_{\infty}}^{(n)},t\geq 0)$ for $\zeta_{\infty}:=\inf\{t>0: N_t^{(n)}=\infty\}\in [0,\infty]$.
%, is the process $(N^{(n)}_t,t\geq 0)$ absorbed after having reached $\infty$ (whenever it reaches it). 
Moreover,
\begin{itemize}
\item[i)] the boundary $1$ is an entrance for $(X^{\mathrm{r}}_t(x),t\geq 0)$ if and only if $\infty$ is an exit for $(N^{(n)}_t,t\geq 0)$;
\item[ii)] the boundary $1$ is regular non-absorbing for  $(X^{\mathrm{r}}_t(x),t\geq 0)$ if and only if $\infty$ is regular absorbing for $(N^{\mathrm{min},(n)}_t,t\geq 0)$,  i.e $\infty$ is regular non-absorbing for the process $(N^{(n)}_t,t\geq 0)$;
\item[iii)] the boundary $1$ is an exit for $(X^{r}_t(x),t\geq 0)$ if and only if $\infty$ is an entrance of $(N^{(n)}_t,t\geq 0)$; 
\item[iv)] the boundary $1$ is natural for $(X^{r}_t(x),t\geq 0)$ if and only if $\infty$ is natural for $(N^{(n)}_t,t\geq 0)$. 
%$X^{\mathrm{r}}_t(1-)=1$ for any $t\geq 0$ a.s. 
%\item {\blue Is there a case missing for $\infty$ regular non-absorbing for $N_t$????}
\end{itemize}
\end{theorem}
%\begin{remark} In Theorem \ref{thmmomentduality2}, the boundary $\infty$ is assumed to be absorbing but can either be an exit or a regular boundary of $(N_t^{(n)},t\geq 0)$. In the latter case, the extra boundary condition of being absorbing is imposed.
%\end{remark}
The proof is more involved than for Theorem \ref{thmmomentdual1}, since one has to construct an extension of the minimal solution to \eqref{SDEselec} possibly getting out from $1$. We will construct $(X_t^{\mathrm{r}}, t\geq 0)$ as limit of processes with boundary $1$ entrance and the assumption $\Lambda(\{0\})=0$ will play here a crucial role. The proof of Theorem \ref{thmmomentduality2} is deferred. We establish first several lemmas. 
The strategy  of the proof is the following. The coalescence and splitting measures $\Lambda$ and $\mu$ are fixed and are assumed to verify $\Lambda(\{0\})=0$ and $\mu(\infty)=0$. Recall that $f$ is the generating function of $\mu$.
\begin{itemize}
\item[-] Step 1: Let $\mu$ be a splitting measure with no mass at $\infty$ and $f$ its associated  probability generating function.  We consider a family of $\Lambda$-Wright-Fisher processes with selection, $(X^{\lambda}_t,t\geq 0)$ indexed by $\lambda>0$. Each has its selection mechanism driven by the defective function $f^{\lambda}$ associated to $\mu^{\lambda}=\mu+\lambda\delta_{\infty}$. We then establish a duality relationship between this process $(X^{\lambda}_t,t\geq 0)$ and the block counting process $(N_t^{\lambda},t\geq 0)$ of a simple EFC process whose splitting measure is $\mu^{\lambda}$. We show  that under the assumption $\Lambda(\{0\})=0$, $(X^{\lambda}_t,t\geq 0)$ has boundary $1$ \textit{entrance}.  This is the aim of Lemma \ref{recextenslambda}.
\item[-] Step 2: We study the dual processes $(N_t^{\lambda},t\geq 0)$  as $\lambda$ goes to $0$ and establish that they converge as $\lambda$ goes to $0$ towards $(N_t^{\mathrm{min}},t\geq 0)$, the block counting process with splitting measure $\mu$ and coalescence measure $\Lambda$ that is stopped after it has reached the boundary $\infty$. This is the aim of Lemma \ref{Nlambda}. 
\item[-] Step 3: The convergence of the processes  $(N_t^{\lambda},t\geq 0)$ as $\lambda$ goes to $0$ shown in Step 2 entails the convergence of processes $(X^{\lambda}_t,t\geq 0)$ considered in Step 1. We study the limit process called $(X^{\mathrm{r}}_t,t\geq 0)$, establish the duality relationship \eqref{dualityEFCref} and verify that this is an extension of the minimal $\Lambda$-WF process with selection driven by $f$. This is the aim of Lemma \ref{recextensX}.
\item[-] Step 4: We study the possible behaviors at the boundary $1$ of the process $(X_t^{\mathrm{r}},t\geq 0)$ from the duality \eqref{dualityEFCref}. When the process started from $1$ is degenerate at  $1$, and further $1$ is inaccessible,  we say that the boundary $1$ is natural. In the case that $1$ is accessible, the boundary is an exit. When the process started from $1$ leaves $1$ and never reaches it again, $1$ is an entrance. Finally, when  the process started from $1$, leaves and returns to $1$, we say that the boundary is regular. The correspondences i) to iv) stated in Theorem \ref{thmmomentduality2} are then established. This is done in Lemma \ref{correspond}. 
\end{itemize}
%Recall that $\Lambda$ and $\mu$ are respectively the coalescence and the splitting measures. 

%We define the process $(N^{\mathrm{min}}_t,t\geq 0)$ as a pointwise limit as follows $N_t^{\mathrm{min}}:=\underset{\lambda \rightarrow 0+}{\lim}N_t^{\lambda}$ for all $t\geq 0$ a.s. 

\textbf{Step 1}. Let $\Lambda$ be a coalescence measure with no atom at $0$ nor $1$. Let $\mu$ be a finite measure on $\mathbb{N}$ and denote by $f$ its probability generating function. Fix $\lambda>0$. We denote by $(N_t^{\lambda,(n)},t\geq 0)$ the block counting process started from $n\in \bar{\mathbb{N}}$ with coalescence measure $\Lambda$ and splitting measure $\mu^{\lambda}$ defined such that $\mu^{\lambda}(k)=\mu(k)$ for any $k\in \mathbb{N}$ and $\mu^{\lambda}(\infty)=\lambda$. Let $f^{\lambda}$ be the defective probability generating function associated to $\mu^{\lambda}$. For any $x\in [0,1]$, $f^{\lambda}(x)=\sum_{k=1}^{\infty}x^{k}\frac{\mu(k)}{\lambda+\mu(\mathbb{N})}=\frac{\mu(\mathbb{N})}{\lambda+\mu(\mathbb{N})}f(x)$ and $f^{\lambda}(1)-1=-\frac{\lambda}{\mu(\mathbb{N})+\lambda}<0$. 

Recall the SDE \eqref{SDEselec}. Let $(X^{\lambda}_t,t\geq 0)$ be the minimal process solution to the equation  
\begin{align}\label{sdelambda}X^{\lambda}_t(x)=x+\int_{0}^{t}\int_{0}^{1}\int_{0}^{1}z\left(\mathbbm{1}_{\{v\leq X^{\lambda}_{s-}(x)\}}-X^{\lambda}_{s-}(x)\right)&\bar{\mathcal{M}}(\ddr s,\ddr v, \ddr z)\\ \nonumber 
&+\mu^{\lambda}(\bar{\mathbb{N}})\int_{0}^{t}\big(f^{\lambda}(X^{\lambda}_{s}(x))-X^{\lambda}_{s}(x)\big)\ddr s.
\end{align}

%\begin{align}\label{sdelambda}X_t(x)=x+\int_{0}^{t}\int_{0}^{1}\int_{0}^{1}z\left(\mathbbm{1}_{\{v\leq X_{s-}(x)\}}-X_{s-}\right)&\bar{\mathcal{M}}(\ddr s,\ddr v, \ddr z) \nonumber\\
%&+(\mu(\mathbb{N})+\lambda)\int_{0}^{t}(f^{\lambda}(X_{s-}(x))-X_{s-}(x))\ddr s\end{align}
%We see in the next lemma that the process can not reach boundary $1$ and can be started from $1$. 
%%and secondly that the process solution to \eqref{SDEselec} (loosely speaking the case with $\lambda=0$) can also be started from $1$. We now explain the steps of the proof.
\begin{lemma}
%[Extension of $(X_t^{\lambda},t\geq 0)$]
\label{recextenslambda} The process $(X_t^{\lambda}(x),t\geq 0)$ verifies the duality relationship 
\begin{equation}\label{duallambda} \mathbb{E}[X^{\lambda}_t(x)^n]=\mathbb{E}[x^{N_t^{\lambda,(n)}}],
\end{equation}
for any $n\in \bar{\mathbb{N}}$, any $x\in [0,1)$ and $t\geq 0$. Moreover, $(X_t^{\lambda},t\geq 0)$ has $1$ as an entrance boundary, and the entrance law at $1$  is characterized via its moments by \begin{equation}\label{entrancelawlambda}\mathbb{E}[X^{\lambda}_t(1)^n]:=\underset{x\rightarrow 1-}{\lim} \mathbb{E}[x^{N_t^{\lambda,(n)}}]=\mathbb{P}(N_t^{\lambda,(n)}<\infty),\text{ for any } t\geq 0 \text{ and any } n\in \mathbb{N}.
\end{equation} The semigroup of $(X_t^{\lambda}(x),t\geq 0, x\in [0,1])$ satisfying \eqref{duallambda} and \eqref{entrancelawlambda} is Feller.
\end{lemma}
\begin{remark} When the measure $\mu$ is a Dirac mass at $\infty$, $\mu=\lambda\delta_{\infty}$, the drift term in the stochastic equation \eqref{sdelambda} reduces to $-\lambda\int_{0}^{t}X_s^{\lambda}(x)\ddr s$. The process $(X_t^{\lambda}(x),t\geq 0)$ corresponds to a $\Lambda$-Wright-Fisher process with no selection but unilateral mutation  from type $a$ to $A$ at rate $\lambda$. A side consequence of Lemma \ref{recextenslambda} is that the boundary $1$ is an entrance (thus inaccessible) for the $\Lambda$-Wright-Fisher process with unilateral mutation when $\Lambda(\{0\})=\Lambda(\{1\})=0$.  \end{remark}
\begin{proof} Let $(X_t^{\lambda,\mathrm{min}},t\geq 0)$ be the process solution to \eqref{sdelambda} that is stopped when reaching boundary $1$. Namely $(X_t^{\lambda,\mathrm{min}},t\geq 0):=(X_{t\wedge \tau_1}^{\lambda},t\geq 0)$. By Theorem \ref{thmmomentdual1}, the process satisfies the identity $\mathbb{E}[X_t^{\lambda,\mathrm{min}}(x)^n]=\mathbb{E}[x^{N_t^{\lambda,(n)}}]$ for any $x\in [0,1)$ and $n\in \mathbb{N}$. Recall that the key Lemma \ref{key} ensures that for any $\lambda>0$, the process $(N^{\lambda}_t,t\geq 0)$ with splitting measure $\mu^{\lambda}$ has $\infty$ as an exit boundary (i.e. $\infty$ is accessible absorbing). Moreover, by Lemma \ref{monotonicity}, $N_t^{\lambda,(\infty)}=\underset{n\rightarrow \infty}{\lim} N_t^{(n)}$ a.s. Therefore $\mathbb{E}[x^{N_t^{\lambda,(\infty)}}]=\mathbb{P}(X_t^{\lambda,\mathrm{min}}(x)=1)=0$ and $1$ is inaccessible for $(X_t^{\lambda,\mathrm{min}},t\geq 0)$ and for $(X_t^{\lambda},t\geq 0)$. Thus, one has $\mathbb{E}[X^{\lambda}_t(x)^n]=\mathbb{E}[x^{N_t^{\lambda,(n)}}]$ and letting $x$ go towards $1$, we see that $$\underset{x\rightarrow 1-}{\lim} \ \mathbb{E}[X^{\lambda}_t(x)^n]=\mathbb{P}(N_t^{\lambda,(n)}<\infty)\in (0,1).$$
This characterizes an entrance law at boundary $1$ for the process $(X_t^{\lambda},t\geq 0)$. We establish the Feller property of the extended semigroup of $(X^{\lambda}_t,t\geq 0)$ on $[0,1]$. Recall $g_n(x)=g_x(n)=x^{n}$.
Plainly by the duality relationship \eqref{duallambda}, if one denotes by $(P^{\lambda}_t)$ the semigroup of $(X_t^{\lambda},t\geq 0)$, we see that $P^{\lambda}_tg_n:x\mapsto \mathbb{E}[X_t^{\lambda}(x)^n]=\mathbb{E}[x^{N_t^{\lambda,(n)}}]$ is continuous on $[0,1)$. Thus, for any polynomial function $h$ on $[0,1]$, $x\mapsto \mathbb{E}[h(X_t^{\lambda}(x))]$ is continuous on $[0,1]$. By  the Weierstrass theorem, if $g\in C([0,1])$, one can find a sequence of polynomial functions $(h_n)$ such that $h_n\underset{n\rightarrow \infty}{\longrightarrow} g $ uniformly. The following routine calculation establishes the continuity of $P^{\lambda}_t g$: for any $x,y\in [0,1]$,
\begin{align*}
&\left\lvert P^{\lambda}_tg(x)-P_t^{\lambda}g(y)\right\lvert\\
&=\left\lvert P_t^{\lambda}g(x)-P_t h_n(x)+P_t^{\lambda}h_n(x)-P_t^{\lambda}h_n(y)+P_t^{\lambda}h_n(y)-P_t^{\lambda}g(y)\right\lvert\\
&\leq 2||g-h_n||_\infty+|P_t^{\lambda}h_n(x)-P_t^{\lambda}h_n(y)|.
\end{align*}
Since $P_t^{\lambda} h_n$ is continuous on $[0,1]$, if one let $x$ tend to $y\in [0,1]$ and then $n$ to $\infty$, we get 
\[\underset{x\rightarrow y}{\limsup}\left\lvert P_t^{\lambda}g(x)-P_t^{\lambda}g(y)\right\lvert\leq ||g-h_n||_\infty\underset{n\rightarrow \infty}{\longrightarrow}0,\]
which allows us to conclude that $P_t^{\lambda}$ maps $C([0,1])$ into $C([0,1])$.   We now check the strong continuity at $0$ of the semigroup $P_t^{\lambda}$. Since it is Feller, it is sufficient to check the pointwise continuity, and by the Weierstrass theorem, we can focus on the functions $g_n:x\mapsto x^{n}$, namely we need to show $\mathbb{E}[X^{r}_t(x)^{n}]\underset{t\rightarrow 0+}{\longrightarrow} x^{n}$. The latter follows readily from the duality \eqref{duallambda} and the right-continuity of $(N_t^{\lambda},t\geq 0)$. \qed
\end{proof}
\textbf{Step 2}: We now study the dual process $(N_t^{\lambda}, t\geq 0) $. Recall that by assumption $\Lambda(\{0\})=0$. 
%Let $\mu$ be a finite measure on $\mathbb{N}$. 
The key Lemma \ref{key} plays again a central role in the proof of the following lemma.
%, which ensures that for any $\lambda>0$, the process $(N^{\lambda}_t,t\geq 0)$ with splitting measure $\mu^{\lambda}$ has $\infty$ as an exit boundary (i.e. $\infty$ is accessible absorbing). 
Recall that $(N^{\mathrm{min},(n)}_t,t\geq 0)$ denotes the block counting process started from $n\in \mathbb{N}$ and absorbed at $\infty$ whenever it reaches it.
\begin{lemma}\label{Nlambda} Let $(N_t^{(n)},t\geq 0)$ be a block counting process with coalescence measure $\Lambda$ (with no mass at $0$) and splitting measure $\mu$ (with no mass at $\infty)$. There exists on the same probability space, block counting processes $(N_t^{\lambda,(n)},t\geq 0,\lambda>0)$ started from $n\in \mathbb{N}$ with splitting measure $\mu^{\lambda}:=\mu+\lambda\delta_\infty$ and same coalescence measure $\Lambda$, such that if $\lambda'\leq \lambda$ then
\begin{equation}\label{orderlambda}N_t^{\lambda',(n)}\leq N_t^{\lambda,(n)} \text{ for all }t\geq 0 \text{ a.s.}
\end{equation}
%and for any $\lambda>0$, $(N_t^{\lambda,(n)},t\geq 0)$ has generator
%\[\mathcal{L}^{\lambda}g(\ell):=\mathcal{L}^c g(\ell)+\mathcal{L}^{f}g(\ell)+\lambda \ell(g(\infty)-g(\ell)), \text{ with } \ell\in \mathbb{N},\]
%for any bounded $g$ defined on $\bar{\mathbb{N}}$.
Almost surely for any $n\in \mathbb{N}$, $\underset{\lambda\rightarrow 0+}\lim \!\!\! \downarrow N_t^{\lambda,(n)}=N_t^{\mathrm{min},(n)} \text{ for all } t\geq 0$
%has the same law as the stopped process $(N^{(n)}_{t\wedge \zeta_{\infty}},t\geq 0)$, where $\zeta_{\infty}$ is the explosion time. 
and $\zeta_{\infty}^{\lambda}\underset{\lambda \rightarrow 0}{\longrightarrow} \zeta_{\infty} \text{ a.s. }$
where $\zeta_{\infty}^{\lambda}:=\inf\{t>0: N_t^{\lambda,(n)}=\infty\}$ and $\zeta_{\infty}:=\inf\{t>0: N_{t-}^{(n)}=\infty\}$.
%with \[\mathcal{L}^{f,\lambda}f(n)=\mathcal{L}^{f}f(n)+\lambda n(f(\infty)-f(n)),\]
%the same law as the block counting process of a simple EFC process with coalescence measure $\Lambda$, and splitting measure $\mu^{\lambda}$ such that $\mu^{\lambda}(n):=\mu(n)$ for any $n\in \mathbb{N}$ and $\mu^{\lambda}(\infty):=\lambda$, 
\end{lemma}
\begin{proof}
We work at the level of the partition-valued processes. 
%Poisson point processes driven fragmentations. 
%The coupling between the block counting processes $(N_t^{\lambda},t\geq 0, \lambda>0)$ given in the proof of Lemma \ref{Nlambda} can be made at the level of the partition-valued processes.
% and one may wonder why indeed the processes $(N_t^{\lambda},t\geq 0)$ are the block counting processes of simple EFCs with splitting measures $\mu^{\lambda}$.
%Since the present work concerns more the block counting processes than partition-valued processes behind, we have not insisted in the proof of Lemma \ref{Nlambda}, on the reasons why $(N^{\lambda}_t,t\geq 0)$ has indeed the same law as the block counting process of a simple EFC process. Here is an argument for the interested reader. 
Recall the Poisson construction of the process $(\Pi(t),t\geq 0)$ explained in Section \ref{backgroundEFC}.  
%The coupling in $\lambda$'s can be made in a simple way at the level of the Poisson point processes driven fragmentations.  Let $(\Pi(t),t\geq 0)$ be a simple EFC process whose (finite) fragmentation measure is denoted by $\mu_{\mathrm{Frag}}$ . 
Let $\mathrm{PPP}_F$ and $\mathrm{PPP}_C$ be the Poisson point processes governing respectively fragmentations and coalescences. Recall the fragmentation measure $\mu_{\mathrm{Frag}}$, and that the latter is finite by assumption. Let $\mathrm{m}:=(m_1,m_2,\ldots)$ be a positive sequence such that $\sum_{i=1}^{\infty}m_i=1$ and $m_i>m_{i+1}>0$ for all $i\geq 1$. Denote by $\rho_{\mathrm{m}}$ the law of the paint-box partition whose ranked asymptotic frequencies are given by $\mathrm{m}$ and note that $\#\pi=\infty$, for $\rho_{\mathrm{m}}$-almost every partition $\pi$. Consider now $\mathrm{PPP}^{1}:=\sum_{i\geq 1}\delta_{(t^{1}_i,\pi_i,j)}$  an independent Poisson point process, with intensity $\ddr t\otimes \rho_{\mathrm{m}}\otimes \#$, where we recall that $\#$ is the counting measure. Let $\big(\mathrm{PPP}^{\lambda}, \lambda>0\big)$ be the images of $\mathrm{PPP}^{1}$ by the map $t\mapsto \lambda t$. They are Poisson point processes with intensity $\lambda \ddr t\otimes \rho_{\mathrm{m}}\otimes \#$, such that if $\lambda\geq \lambda'$, then $t_1^{\lambda}\leq t_1^{\lambda'}$ almost surely where $t_1^{\lambda}$ denotes the first atom of time of $\mathrm{PPP}^{\lambda}$. Let $(\Pi^{\lambda}(t),t\geq 0)$ for $\lambda>0$ be the simple EFC processes  built from the Poisson point process $\mathrm{PPP}^{\lambda}_{F}=\mathrm{PPP}_F+\mathrm{PPP}^{\lambda}$ and $\mathrm{PPP}_C$. For any $\lambda>0$, the fragmentation measure is $\mu^{\lambda}_{\mathrm{Frag}}:=\mu_{\mathrm{Frag}}+\lambda \rho_{\mathrm{m}}$, and thus the splitting measure is $\mu^{\lambda}$ such that $\mu^{\lambda}(k))=\mu(k)$ for all $k\in \mathbb{N}$ and $\mu^{\lambda}(\infty)=\lambda \rho_{\mathrm{m}}(\pi; \#\pi=\infty)=\lambda$. 
Recall $(\Pi^{\lambda,(n)}(t),t\geq 0)$ defined in Section \ref{backgroundEFC}. Set $(N_t^{\lambda,(n)},t\geq 0):=(\#\Pi^{\lambda,(n)}(t),t\geq 0)$ and $\zeta_{\infty}^{\lambda}:=\inf\{t>0: N_{t-}^{\lambda,(n)}\text{ or } N_{t}^{\lambda,(n)}=\infty\}$.  Recall Lemma \ref{key}. All processes $N^{\lambda}$ have boundary $\infty$ exit and we see by construction that for any $\lambda>\lambda'>0$  the processes $(N_t^{\lambda,(n)},t\geq 0)$ satisfy $N^{\lambda,(n)}_t\geq N_t^{\lambda',(n)}$ for any $t$ almost surely and $\zeta_{\infty}^{\lambda}\leq \zeta_{\infty}^{\lambda'}$ a.s.
For any $t\geq 0$, set $N_t^{0}:=\underset{\lambda \rightarrow 0}{\lim}\downarrow N_t^{\lambda}$ a.s. Plainly for any $\lambda>0$, $\zeta_{\infty}:=\inf\{t>0: N_{t-}^{(n)}=\infty\}\geq \zeta^{\lambda}$ a.s. Hence, on the event $\{\zeta_{\infty}^{0}<\infty\}$, for any $t>0$, $t+\zeta_{\infty}^{0}\geq t+\zeta_{\infty}^{\lambda}$ and  Lemma \ref{key} yields that $N^{\lambda,(n)}_{t+\zeta_{\infty}^{0}}=\infty$ for all $t\geq 0$ a.s. Thus
\[N^{0,(n)}_{t+\zeta_{\infty}^{0}}=\underset{\lambda \rightarrow 0+}{\lim} N^{\lambda,(n)}_{t+\zeta_{\infty}^{0}}=\infty \text{ a.s.}\]
The process $(N_t^{0,(n)},t\geq 0)$ is therefore absorbed at infinity whenever it reaches it. We show that $\zeta^0_{\infty}:=\underset{\lambda \rightarrow 0}{\lim} \uparrow \zeta_{\infty}^{\lambda}=\zeta_{\infty}$ a.s. As noticed before, $\zeta_{\infty}^{\lambda}\leq \zeta_{\infty}$ for any $\lambda>0$. Hence $\zeta_{\infty}^0\leq  \zeta_{\infty}$. We work now on the event $\{\zeta_{\infty}^0<\infty\}$. Assume that $N_{\zeta_{\infty}^0}^{0,(n)}<\infty$. Since $N_t^{\lambda,(n)} \underset{\lambda \rightarrow 0}{\longrightarrow} N_t^{0,(n)}$ for any $t\geq 0$ a.s, if $N_{\zeta_{\infty}^0}^{0,(n)}<\infty$, then there exists $\lambda>0$ small enough such that $N^{\lambda}_{\zeta_{\infty}^{0},(n)}=N_{\zeta_{\infty}^0}^{0,(n)}<\infty$. This entails the contradiction $\zeta_{\infty}^{\lambda}>\zeta_{\infty}^{0}$ a.s. Hence, $N_{\zeta_{\infty}^0}^{0,(n)}=\infty$ a.s and since $\zeta_{\infty}^0\leq  \zeta_{\infty}$, we have that $\zeta_{\infty}^0=\zeta_{\infty}$ a.s. 
Observe finally that by construction, $(N_t^{0,(n)},t<\zeta^{0}_{\infty})$ has the same dynamics as the block counting process $(N_t^{(n)},t<\zeta_{\infty})$ whose splitting measure is $\mu$ and coalescence measure is $\Lambda$. By the uniqueness of the minimal continuous-time Markov chain with generator $\mathcal{L}$, $(N_t^{0,(n)},t\geq 0)$ and the stopped process $(N_t^{\mathrm{min},(n)},t\geq 0):=(N_{t\wedge \zeta_{\infty}}^{(n)},t\geq 0)$ coincide.\qed
%\zeta^{\lambda}_m
%\end{remark}
\end{proof}
We study now some extensions of the minimal process, solution to the SDE \eqref{SDEselec}, whose drift term satisfies $f(1)=1$. We define extensions  by looking at the limit arising in the processes $(X^{\lambda},\lambda>0)$ when the mutation rate $\lambda$ gets very low.
\begin{lemma}[Extension of $(X^{\mathrm{min}}_t,t\geq 0)$ after reaching $1$]\label{recextensX} The Markov processes $(X^{\lambda}_t(x),t\geq 0, x\in [0,1])$ converge as $\lambda$ goes to $0$, in the Skorokhod space towards a Feller process $(X_t^{\mathrm{r}}(x),t\geq 0,x\in [0,1])$ valued in $[0,1]$, which extends the minimal solution of \eqref{SDEselec}, and whose semigroup satisfies : for any $n\in \mathbb{N}$ 
\begin{equation}\label{dualabso} \mathbb{E}[X^{\mathrm{r}}_t(x)^n]=\mathbb{E}[x^{N^{\mathrm{min},(n)}_t}] 
\text{ for any } x\in [0,1),
\end{equation}
\text{ and } 
\begin{equation}\label{entrancelawXr}
\mathbb{E}[X^{\mathrm{r}}_t(1)^n]:=\underset{x\rightarrow 1-}{\lim}\mathbb{E}[x^{N^{\mathrm{min},(n)}_t}]=\mathbb{P}_n(\zeta_\infty>t).
\end{equation} 
%where 
%Moreover, for any $x\in [0,1]$, $(X_t^{\lambda}(x),t\geq 0)$  towards  $(X_t^{\mathrm{r}}(x),t\geq 0)$.
\end{lemma}
\begin{remark}\label{remwhynotkingman} 
In order to establish Lemma \ref{recextensX}, we require the assumption that the measure $\Lambda$ gives no mass to $0$. We shall indeed use the fact that when there is no Kingman component, the processes $N^{\lambda}$ have all their boundary $\infty$ as exit. This is not the case when $\Lambda(\{0\})=c_{\mathrm{k}}>0$ for which processes $N^{\lambda}$ may have boundary $\infty$ regular, see Remark \ref{remkey}.
%and our method \red{is not successful}
%This towards the stopped process $N^{\mathrm{min}}$, when $\lambda$  tends to $0$, is not true,  see Remark \ref{remkey2}. 
%%\begin{itemize}
%%\item[1)] The process $(N^{\mathrm{min},(n)}_t,t\geq 0)$ corresponds to the minimal process with generator $\mathcal{L}$. 
%%\item[2)] 
%The dual processes of the processes $(N_t^{\lambda},t\geq 0, \lambda>0)$ defined in Lemma \ref{Nlambda} are not clearly ordered as the functions $(f_{\lambda},\lambda>0)$ are not ordered. 
%%\end{itemize}
\end{remark}
%\begin{remark}\label{remwhynotkingman} 
%{\red{Our approach in order to establish Lemma \ref{recextensX} will require the assumption that the measure $\Lambda$ gives no mass to $0$. We shall indeed use the fact that when there is no Kingman component, the processes $N^{\lambda}$ have all their boundary $\infty$ as exit. This is not the case when $\Lambda(\{0\})=c_{\mathrm{k}}>0$ for which processes $N^{\lambda}$ may have boundary $\infty$ regular, see Remark \ref{remkey}. }}
%%and our method \red{is not successful}
%%This towards the stopped process $N^{\mathrm{min}}$, when $\lambda$  tends to $0$, is not true,  see Remark \ref{remkey2}. 
%%%\begin{itemize}
%%%\item[1)] The process $(N^{\mathrm{min},(n)}_t,t\geq 0)$ corresponds to the minimal process with generator $\mathcal{L}$. 
%%%\item[2)] 
%%The dual processes of the processes $(N_t^{\lambda},t\geq 0, \lambda>0)$ defined in Lemma \ref{Nlambda} are not clearly ordered as the functions $(f_{\lambda},\lambda>0)$ are not ordered. 
%%%\end{itemize}
%\end{remark}

\begin{proof}
In order to ease the reading we outline the scheme of the proof. The strategy is similar to that in the proof of Lemma \ref{recextenslambda}
\begin{enumerate}
%\item We show that the sequence of processes $(N_t^{\lambda,(n)},t\geq 0)$ converges pointwise almost surely towards a process $(N^{\mathrm{min},(n)}_t,t\geq 0)$ with the same law as $(N^{(n)}_{t\wedge \zeta},t\geq 0)$ where we recall $\zeta:=\inf\{t>0: N_t=\infty\}$.
\item We  first show that for any fixed $t\geq 0$, the random variables $(X_t^{\lambda}(x),x \in[0,1])$  converge in law as $\lambda$ goes to $0$ through convergence of their moments.
\item We establish that the semigroup of $(X^{\lambda}_t(x),t\geq 0, x\in [0,1])$ converges as $\lambda$ goes to $0$ uniformly. We denote by $P_t^{\mathrm{r}}$ the limiting operator and verifies that it is a Feller semigroup. The associated process is denoted by $(X^{\mathrm{r}}_t(x),t\geq 0,x\in [0,1])$. 
%
%The convergence in the Skorokhod space is ensured by a Theorem of Ethier-Kurtz \cite{EthierKurtz}.
\item We then show that $(X^{\mathrm{r}}_t(x),t\geq 0,x\in [0,1])$ extends the minimal solution to \eqref{SDEselec} after explosion at $1$.
%\item We finally check that for any $x\in [0,1]$, if $\lambda<\lambda'$, then $X^{\lambda}_t(x)\leq X_t^{\lambda'}(x)$. This entails the pointwise convergence.
%for any $t>0$, $p_t^{\lambda}(1,\ddr y)$ converges weakly as $\lambda$ goes to $0$ towards a probability measure $\nu_t$ on $[0,1]$, characterized by its moments as follow for all $n\geq 0$, $\int_{[0,1]}x^{n}\nu_t(\ddr x):=\mathbb{P}_n(\zeta>t)$ and check that $\nu_t$ converges weakly towards $\delta_1$ as $t$ goes to $0$.
%\item We verify that $(\nu_{t},t>0)$ forms an entrance law for the process $(X^{\lambda}_t,t\geq 0)$ i.e $\nu^{\lambda}_{t+s}=\nu_t^{\lambda}P^{\lambda}_s$ for any $s\geq 0$ and $t>0$.
\end{enumerate}
(1) 
By letting $\lambda$ go towards $0$ in the duality relationship \eqref{duallambda}, and recalling the almost sure convergence of $N_t^{\lambda,(n)}$ towards $N_t^{\mathrm{min},(n)}$, we see that  $\underset{\lambda\rightarrow 0}{\lim}\ \mathbb{E}[X_t^{\lambda}(x)^{n}]=\mathbb{E}[x^{N_t^{\mathrm{min},(n)}}]$. Recall that the convergence in law of random variables valued in $[0,1]$ is characterized by the convergence of the entire moments. Therefore, the $X_t^{\lambda}(x)$'s are converging in law as $\lambda$ goes to $0$, and the limit law is characterized by the sequence of its entire moments, $\left(\mathbb{E}[x^{N_t^{\mathrm{min},(n)}}],n\geq 0\right)$. For all $x\in [0,1]$ we denote by $X_t^{\mathrm{r}}(x)$ the random variable valued in $[0,1]$ such that
%\[
$\underset{\lambda\rightarrow 0}{\lim}\ \mathbb{E}[(X_t^{\lambda}(x))^{n}]=\mathbb{E}[(X^{\mathrm{r}}_t(x))^{n}]=\mathbb{E}[x^{N_t^{\mathrm{min},(n)}}].$
%\]
%Let $m\in \mathbb{N}$. One 
%By letting $\lambda$ towards $0$ in the duality relationship \eqref{duallambda}, and recalling the almost sure convergence of $N_t^{\lambda,(n)}$ towards $N_t^{\mathrm{min},(n)}$, we see that 
%\[\underset{\lambda\rightarrow 0}{\lim} \mathbb{E}[X_t^{\lambda}(x)^{n}]=\mathbb{E}[x^{N_t^{\mathrm{min},(n)}}].\]
%Since the $X_t(x)^{\lambda}$'s are random variables taking values in $[0,1]$, the convergence in law is characterized by the convergence of the entire moments, and moreover $\left(\mathbb{E}[x^{N_t^{\mathrm{min},(n)}}],n\geq 0\right)$ characterizes a probability measure on $[0,1]$, we see that for all $x\in [0,1]$ and all $n\geq 0$, there exists a random variable $X_t^{\mathrm{r}}(x)$ valued in $[0,1]$ such that
%\[\underset{\lambda\rightarrow 0}{\lim} \mathbb{E}[(X_t^{\lambda}(x))^{n}]=\mathbb{E}[(X^{\mathrm{r}}_t(x))^{n}]=\mathbb{E}[x^{N_t^{\mathrm{min},(n)}}].\]

(2) Recall that by Lemma \ref{recextenslambda}, $(X^{\lambda}_t(x),t\geq 0)$ is a Feller process and that we denote its semigroup by $(P_t^{\lambda},t\geq 0)$.
% and its transition kernel by $(p_t^{\lambda}(x,\ddr y),t\geq 0, x\in [0,1])$. 
Let $g_n(x)=x^{n}$ for any $x\in [0,1]$ and $n\in \mathbb{N}$.
By the duality relationship \eqref{duallambda}, 
%\[
$P_t^{\lambda}g_n(x)=\mathbb{E}[x^{N^{(n),\lambda}_{t}}]$.
We check that 
\begin{equation}\label{unifconvsemigroup}||P_t^{\lambda}g_n-P_t^{\mathrm{r}}g_n||_\infty=\sup_{x\in [0,1]}\mathbb{E}[x^{N_t^{\mathrm{min},(n)}}-x^{N^{(n),\lambda}_{t}}]\underset{\lambda \rightarrow 0}{\longrightarrow} 0.
\end{equation}
Arguments are adapted from those in \cite[Section 7]{zbMATH07055671}.  
%For any $m\in \mathbb{N}$,
%\begin{align*}
%\sup_{x\in [0,1]}\mathbb{E}[x^{N_t^{\mathrm{min},(n)}}-x^{N^{(n),\lambda}_{t}}]\\
%&\leq \mathbb{E}\left[\sup_{x\in [0,1]}x^{N_t^{\mathrm{min},(n)}}-x^{N^{(n),\lambda}_{t}}\right]\\
%%&\leq \mathbb{E}\left[\sup_{x\in [0,1]}\left(x^{N^{\mathrm{min},(n)}_{t}}-x^{N_t^{\lambda,(n)}}\right)\right].
%\end{align*}
For any $x\in [0,1]$,
\begin{align*}
\mathbb{E}[x^{N_t^{\mathrm{min},(n)}}-x^{N^{(n),\lambda}_t}]&=\mathbb{E}\left[\left(x^{N^{\mathrm{min},(n)}_{t}}-x^{N_t^{\lambda,(n)}}\right)\mathrm{1}_{\{N_t^{\mathrm{min},(n)}\leq N_t^{\lambda,(n)}<\infty\}}\right]\\
&\qquad +\mathbb{E}\left[\left(x^{N^{\mathrm{min},(n)}_{t}}-x^{N_t^{\lambda,(n)}}\right)\mathrm{1}_{\{N_t^{\mathrm{min},(n)}<\infty, N_t^{\lambda,(n)}=\infty\}}\right]\\
&\leq \mathbb{E}\left[\left(x^{N^{\mathrm{min},(n)}_{t}}-x^{N_t^{\lambda,(n)}}\right)\mathrm{1}_{\{N_t^{\mathrm{min},(n)}\leq N_t^{\lambda,(n)}<\infty\}}\right]+2\mathbb{P}_n(\zeta_{\infty}>t\geq \zeta_{\infty}^{\lambda}).
\end{align*}
Recall that by Lemma \ref{Nlambda}, $\zeta_{\infty}^{\lambda}\underset{\lambda \rightarrow 0+}{\longrightarrow} \zeta_{\infty}$ a.s, thus $\mathbb{P}_n(\zeta_{\infty}>t\geq \zeta_{\infty}^{\lambda})\underset{\lambda \rightarrow 0}{\longrightarrow} 0$. It remains to study the uniform convergence on the event $\{N_t^{\mathrm{min},(n)}\leq N_t^{\lambda,(n)}<\infty\}$.
Plainly, 
\begin{align}\label{upperboundunif}
\underset{x\in [0,1]}{\sup}& \mathbb{E}\left[\left(x^{N^{\mathrm{min},(n)}_{t}}-x^{N_t^{\lambda,(n)}}\right)\mathrm{1}_{\{N_t^{\mathrm{min},(n)}\leq N_t^{\lambda,(n)}<\infty\}}\right]\nonumber\\
&\leq \mathbb{E}\left[\underset{x\in [0,1]}{\sup}\left(x^{N^{\mathrm{min},(n)}_{t}}-x^{N_t^{\lambda,(n)}}\right)\mathrm{1}_{\{N_t^{\mathrm{min},(n)}\leq N_t^{\lambda,(n)}<\infty\}}\right].
\end{align}
Recall that $N_t^{\lambda,(n)} \underset{\lambda \rightarrow 0}{\longrightarrow}
N^{\mathrm{min},(n)}_t$ a.s. On the event $\{N_t^{\mathrm{min},(n)}\leq N_t^{\lambda,(n)}<\infty\}$, since both random variables $N_t^{\mathrm{min},(n)}$ and $N_t^{\lambda,(n)}$ are integer-valued, there exists almost surely a small enough $\lambda_0>0$ such that for all $\lambda<\lambda_0$, $N_t^{\lambda,(n)}=N_t^{\mathrm{min},(n)}$.
% setting $H:=N^{(n),\lambda}_t-N_t^{\mathrm{min},(n)}$, we have $H\underset{\lambda \rightarrow 0}{\longrightarrow} 0$ a.s. and
%\begin{equation}\label{Phit} \Phi_t(\lambda):=\frac{\log \left(N_t^{\mathrm{min},(n)}+H\right)-\log N_t^{\mathrm{min},(n)}}{H}\underset{\lambda \rightarrow 0}{\longrightarrow} \frac{1}{N_t^{\mathrm{min},(n)}} \text{ a.s}
%\end{equation}
Thus the upper bound \eqref{upperboundunif} vanishes for $\lambda\leq \lambda_0$. Finally
\[\underset{\lambda \rightarrow 0}{\limsup} \underset{x\in [0,1]}{\sup}\mathbb{E}\left[\left(x^{N^{\mathrm{min},(n)}_{t}}-x^{N_t^{\lambda,(n)}}\right)\mathrm{1}_{\{N_t^{\mathrm{min},(n)}\leq N_t^{\lambda,(n)}<\infty\}}\right]=0,\]
and the convergence in \eqref{unifconvsemigroup} is established. To see that the uniform convergence holds for any function $f\in C([0,1])$, one argues by the Stone-Weierstrass theorem as follows. Let $f\in C([0,1])$ and $(f_n)$ be a sequence of polynomial functions such that $||f_n-f||_\infty\underset{n\rightarrow \infty}{\longrightarrow} 0$. Then, 
\begin{align*}||P_t^{\lambda}f-P_t^{\mathrm{r}}f||_\infty &\leq ||P_t^{\lambda}f-P_t^{\lambda}f_n||_\infty+||P_t^{\lambda}f_n-P_t^{\mathrm{r}}f_n||_\infty+||P_t^{\mathrm{r}}f_n-P_t^{\mathrm{r}}f_n||_\infty\\
&\leq 2||f-f_n||_\infty+||P_t^{\lambda}f_n-P_t^{\mathrm{r}}f_n||_\infty.
\end{align*}
By letting $\lambda$ go towards $0$, we see that
\[\underset{\lambda \rightarrow 0+}{\limsup} ||P_t^{\lambda}f-P_t^{\mathrm{r}}f||_\infty\leq  2||f-f_n||_\infty\]
and one concludes by letting $n$ go to $\infty$. We now deduce that $(P_t^{\mathrm{r}})$ is a Feller semigroup. As previously, the Stone-Weierstrass theorem asserts that it suffices to establish the semigroup property for the functions $g_n$. Let $f:=P_s^{\mathrm{r}}g_n$. For any $n\geq 0$,
\begin{align}\label{semigroupproperty}
&||P_{t+s}^{\mathrm{r}}g_n-P_{t}^{\mathrm{r}}P_{s}^{\mathrm{r}}g_n||_\infty \nonumber\\
&\leq ||P_{t+s}^{\mathrm{r}}g_n-P_{t+s}^{\lambda}g_n||_\infty +||P_{t}^{\lambda}P_{s}^{\lambda}g_n-P_{t}^{\lambda}P_{s}^{\mathrm{r}}g_n||_\infty+||P_{t}^{\lambda}P_s^{\mathrm{r}}f-P_{t}^{\mathrm{r}}P_s^{\mathrm{r}}||_\infty \nonumber \\
&\leq ||P_{t+s}^{\mathrm{r}}g_n-P_{t+s}^{\lambda}g_n||_\infty+||P_{s}^{\lambda}g_n-P_{s}^{\mathrm{r}}g_n||_\infty+||P_{t}^{\lambda}f-P_{t}^{\mathrm{r}}f||_\infty,
\end{align}
where we have used  the fact that $P_t^{\lambda} $ is a contraction. The upper bound in \eqref{semigroupproperty} goes towards $0$ as $\lambda$ goes to $0$, and the semigroup property is established. The Feller property follows from the same argument as in the proof of Lemma \ref{recextenslambda}.
%We now justify that the semigroup is strongly continuous at $0$. Since it is Feller, it is sufficient to check the pointwise continuity at the functions $g_n$, namely $\mathbb{E}[X^{r}_t(x)^{n}]\underset{t\rightarrow 0+}{\longrightarrow} x^{n}$. The latter follows readily from the duality \eqref{dualabso} and the contruction of $(N_t^{\mathrm{min},(n)},t\geq 0)$. Indeed, for any $0<t<\zeta_{\infty}^\lambda$, $N_t^{\mathrm{min},(n)}=N_t^{\lambda,(n)}$ and since $(N_t^{\lambda,(n)},t\geq 0)$ is c\`adl\`ag, $N_t^{\mathrm{min},(n)}\underset{t\rightarrow0+}{\longrightarrow} n$ a.s. and $\mathbb{E}[x^{N_t^{\mathrm{min},(n)}}]\underset{t\rightarrow 0+}{\longrightarrow} x^{n}$. 
The fact that the convergence of the sequence $(X_t^{\lambda},t\geq 0,\lambda>0)$ as $\lambda$ goes to $0$, holds in the Skorokhod sense is a direct application of \cite[Theorem 2.5 page 167]{EthierKurtz}.\\
(3) Denote by $\mathcal{A}^{\lambda}$ the generator of $(X_t^{\lambda},0\leq t\leq \tau)$. Recall $\mathcal{A}^{\mathrm{s}}$ the generator of the minimal solution to \eqref{SDEselec}. By Lemma \ref{recextenslambda}, for any $g\in C_c([0,1])$, the process $(M_t^{\lambda},t\geq 0)$ defined by
\[M_t^{\lambda}=g(X_t^{\lambda})-\int_{0}^{t}\mathcal{A}^{\lambda, \mathrm{s}}g(X_s^{\lambda})\ddr s\]
is a martingale. We establish now that $\mathcal{A}^{\lambda, \mathrm{s}}$ converges uniformly towards $\mathcal{A}^{\mathrm{s}}$ on functions with a compact support $[a,b]$ contained in $(0,1)$. Recall the drift term of $\mathcal{A}^{\mathrm{s}}g(x)$ in \eqref{generatorWFs}: $\mu(\mathbb{N})(f(x)-x)g'(x)$. Since $\mathcal{A}^{\lambda, \mathrm{s}}$ and $\mathcal{A}^{\mathrm{s}}$ have the same jump parts, for any $g\in C_c([0,1])$; 
\begin{align}\label{unifconvgen}
||\mathcal{A}^{\lambda, \mathrm{s}}g-\mathcal{A}^{\mathrm{s}}g||_\infty&=\underset{x\in (0,1)}{\sup}\left\lvert (\mu(\mathbb{N})+\lambda)f^{\lambda}(x)-x)g'(x)-\mu(\mathbb{N})(f(x)-x)g'(x)\right \lvert \nonumber\\
&=\underset{x\in (0,1)}{\sup}|g'(x)|\left( \mu(\mathbb{N})\lvert f^{\lambda}(x)-f(x)\lvert)+\lambda f^{\lambda}(x)\right)\nonumber\\
&\leq \mu(\mathbb{N})||g'||_\infty \underset{x\in [a,b]}{\sup}|f(x)-f^{\lambda}(x)|+\lambda ||g'||_\infty.
\end{align}
One has
%\[
$\underset{x\in [a,b]}{\sup}|f(x)-f^{\lambda}(x)|=\underset{x\in [a,b]}{\sup}\sum_{k=1}^{\infty}x^{k}\mu(k)\frac{\lambda}{\mu(\mathbb{N})+\lambda}\leq \lambda \sum_{k=1}^{\infty}\frac{\mu(k)}{\mu(\mathbb{N})+\lambda}\leq \lambda.$
%\]
Therefore the right-hand side of \eqref{unifconvgen} goes to $0$ when $\lambda$ goes to $0$ and the generators $\mathcal{A}^{\lambda, \mathrm{s}}$ uniformly converge to $\mathcal{A}^{\mathrm{s}}$. The same arguments as in the proof of Lemma \ref{identificationXinfinity} show that
%Applying \cite[Lemma 5.1 page 196]{EthierKurtz}, we get that 
%\[M_t^{\mathrm{r}}=g(X_t^{\mathrm{r}})-\int_{0}^{t}\mathcal{A}g(X_s^{\mathrm{r}})ds,\] 
%is a martingale. Denote by $\tau^{\mathrm{r}}$ the first explosion time of $(X^{\mathrm{r}}_t,t\geq 0)$, namely 
%$\tau^{\mathrm{r}}:=\tau^{\mathrm{r}}_0\wedge \tau^{\mathrm{r}}_1$, with 
%$\tau^{\mathrm{r}}_i=\inf\{t>0: X^{\mathrm{r}}_t=i\}$ for $i\in \{0,1\}$.
%This is a stopping time and the process $(M^{\mathrm{r}}_{t\wedge \tau^{\mathrm{r}}},t\geq 0)$ is a martingale.  As $g$ and $\mathcal{A}g$ have compact supports, we see that
\[\left(g(X^{\mathrm{r}}_{t\wedge \tau^{\mathrm{r}}})-\int_{0}^{t}\mathcal{A}^{\mathrm{s}}g(X_{s\wedge \tau^{\mathrm{r}}}^{\mathrm{r}})\ddr s, t\geq 0\right)\]
is a martingale where $\tau^{\mathrm{r}}:=\inf\{t>0: X^{\mathrm{r}}_t\notin (0,1)\}$. Finally,  the process $(X^{\mathrm{r}}_{t},t\geq 0)$, stopped at time $\tau^{\mathrm{r}}$ solves the martingale problem  $(\mathrm{MP})$. Since the latter has a unique solution, we see that $(X^{\mathrm{r}}_t,t\geq 0)$ extends the minimal process. \qed%\\
%(4) We establish here that for any $\lambda<\lambda'$, $X_t^{\lambda}(x)\geq X_t^{\lambda'}(x)$.  What about $X_t^{\lambda}(1)$?
%Recall that 
\end{proof}
The next lemma completes the convergence result  in Lemma \ref{recextensX} for processes $(X^{\lambda}_t(x),t\geq 0)$ as $\lambda$ goes to $0$.
\begin{lemma}[Almost sure pointwise convergence and monotonicity]\label{almostsureinx} For any $x\in [0,1]$ and $t\geq 0$, the  limit $X^{\mathrm{r}}_t(x):=\underset{\lambda \rightarrow 0+}{\lim}\! \uparrow X_t^{\lambda}(x)$ exists almost surely. Moreover, if $x\leq y$ then $X_t^{\mathrm{r}}(x)\leq X_t^{\mathrm{r}}(y)$ a.s. In particular, the limit $X_t^{\mathrm{r}}(1):=\underset{x\rightarrow 1-}{\lim}\! \uparrow X_t^{\mathrm{r}}(x)\in [0,1]$ exists almost surely.
\end{lemma}
\begin{proof}
For any $m\geq 2$. Let $f^{\lambda}_m$ be the generating function associated to the measure $\mu^{\lambda}_m$ defined by $\mu^{\lambda}_m(k):=\mu(k)$ for $k\leq m-1$ and $\mu^{\lambda}_m(k):=\mu(\mathbb{N})+\lambda$ for $k\geq m$. Denote by $(X_t^{\lambda,(m)},t\geq 0)$ the $\Lambda$-WF process with selection driven by $f^{\lambda}_m$.  By Lemma \ref{identificationXinfinity}, for any $x\in [0,1)$, $X_t^{\lambda}(x):=\underset{m\rightarrow \infty}{\lim} X_t^{\lambda,(m)}(x)$ and since $x^{m}-x\leq 0$, we easily check that if $\lambda'>\lambda$,
\begin{align*}
(\mu(\mathbb{N})+\lambda)\left(f_{m}^{\lambda}(x)-x\right)&=\sum_{k=1}^{m-1}(x^k-x)\mu(k)+(\bar{\mu}(m)+\lambda)(x^{m}-x)\\
&\geq  \sum_{k=1}^{m-1}(x^k-x)\mu(k)+(\bar{\mu}(m)+\lambda')(x^{m}-x)\\
&=(\mu(\mathbb{N})+\lambda')\left(f_{m}^{\lambda'}(x)-x\right).
\end{align*}
For any $m\geq 2$ and any $\lambda>0$, the function $f_{m}^{\lambda}$ is Lipschitz over $[0,1]$, the comparison theorem therefore applies and for any $x\in [0,1)$, $X_t^{\lambda,(m)}(x)\leq X_t^{\lambda',(m)}(x)$ a.s. By letting $m$ go to $\infty$, we get $X_t^{\lambda}(x)\leq X_t^{\lambda'}(x)$ a.s. Recall that $(X^{\lambda}_t,t\geq 0)$ can be started from $1$. By letting $x$ go to $1$, we also get $X_t^{\lambda}(1)\leq X_t^{\lambda'}(1)$ a.s. Finally the limit $\underset{\lambda \rightarrow 0+}{\lim} X_t^{\lambda}(x)=:X_t^{\mathrm{r}}(x)$ exists almost surely for all $x\in [0,1]$. The monotonicity in the initial values can be checked similarly. 
%Let $x\leq y$, for any $m$ and any $\lambda$,  we have $X_t^{\lambda,(m)}(x)\leq X_t^{\lambda,(m)}(y)$ a.s. Hence, by letting $m$ go to $\infty$, we have $X_t^{\lambda}(x)\leq X_t^{\lambda}(y)$ a.s. Letting $\lambda$ go to $0$, we finally get $X_t^{\mathrm{r}}(x)\leq X_t^{\mathrm{r}}(y)$. 
\qed
%
%Recall $f^{\lambda}$ the defective generating function associated to the renormalized splitting measure $\mu^{\lambda}$. We easily check that if $\lambda'>\lambda$,
%\begin{align*}
%(\mu(\mathbb{N})+\lambda)\left(f^{\lambda}(x)-x\right)&=\sum_{k\in \mathbb{N}}(x^k-x)\mu(k)-\lambda x \\
%&\geq \sum_{k\in \mathbb{N}}(x^k-x)\mu(k)-\lambda'x=(\mu(\mathbb{N})+\lambda')\left(f^{\lambda'}(x)-x\right).
%\end{align*}
%Therefore by the comparison theorem, if $x\in [0,1)$, $X_t^{\lambda}(x)\leq X_t^{\lambda'}(x)$ a.s. Finally the limit $\underset{\lambda \rightarrow 0+}{\lim} X_t^{\lambda}(x)=:X_t^{\mathrm{r}}(x)$ exists almost surely for all $x\in [0,1)$. The monotonicity in the initial values is checked similarly since $X_t^{\lambda}(x)\leq X_t^{\lambda}(y)$ a.s. \qed
\end{proof}
%We need now to show that $(X_t^{\lambda},t\geq 0)$ converges towards a certain process $(X_t^{\mathrm{r}},t\geq 0)$ and to identify the latter limit process. 
%\\
%Set $g_n(x)=x^{n}$ for all $x\in [0,1]$ and all $n\in \bar{\mathbb{N}}$. Denote by $P_t^{\lambda}$ the semigroup of the process $(X^{\lambda}(t),t\geq 0)$ {\red Markov?}
%By duality \eqref{duallambda}, we see that for any fixed $t$, $P^{\lambda}_tg_n(x):=\mathbb{E}[X^{\lambda}_t(x)^n]$ converges towards $P_t^{\mathrm{min}}g_n(x):=\mathbb{E}_n[x^{N^{\mathrm{min}}_t}]$ as $\lambda$ goes to $0$. {\red is it clear that the latter defines a measure? By the Stone-Weierstrass theorem, any continuous function on $[0,1]$ is a uniform limit of polynomial functions, hence the knowledge of $P_t^{\mathrm{min}}g_n(x)$ for all $n\in \mathbb{N}$, is enough to characterize a probability measure on $[0,1]$.}
%\textbf{Details to add!} The operator $P_t^{\mathrm{min}}g_n(x)$ takes the following form \[P_t^{\mathrm{min}}g_n(x)=\int_{[0,1]} y^{n}p^{\mathrm{r}}_t(x,\ddr y)\]
%where $p^{\mathrm{r}}_t(x,\ddr y)$ is the transition kernel of some process $(X_t^{\mathrm{r}},t\geq 0$ limit in finite-dimensional sense of $(X^{\lambda}_t,t\geq 0)$ as $\lambda$ goes to $0$. \textbf{The limit should also be in a weak convergence sense, by using a theorem of Ethier Kurtz.}
%% all moments of $X_t^{\lambda}(x)$ converges as $\lambda$ goes to $0$.. 
%We finally have for any $t\geq 0$, any $n\in \mathbb{N}$ and any $x\in [0,1)$
%\begin{equation}\label{dualabso2} \mathbb{E}_n[x^{N^{\mathrm{min}}_t}]=\mathbb{E}_x[(X^{\mathrm{r}}_t)^n].
%\end{equation}
We can now use both duality relationships \eqref{dualEFC} and \eqref{dualityEFCref} in order to classify the boundaries as in Table \ref{correspondance} and establish
% the statements (i), (ii), (iii) and (iv) of 
Theorem \ref{thmmomentduality2}.
%In the following lemma, $(X_t,t\geq 0)$ denotes the 
\begin{lemma}\label{correspond} 
%Let $(N_t,t\geq 0)$ be the block counting process of a simple EFC process with coalescence measure $\Lambda$ and splitting measure $\mu$, satisfying $\Lambda(\{0\})=\Lambda(\{1\})=0$ and $\mu(\infty)=0$. 
The boundary $1$ is non-absorbing (respectively, inaccessible) for $(X_t^{\mathrm{r}},t\geq 0)$ if and only if the boundary $\infty$ is accessible (respectively, absorbing) for $(N_t,t\geq 0)$.
%\begin{enumerate}
%\item $(N_t,t\geq 0)$ has $\infty$ as exit boundary\footnote{in this case $(N_t,t\geq 0)$ and $(N^{\mathrm{min}}_t,t\geq 0)$ have the same law.} if and only if  $(X_t^{\mathrm{r}},t\geq 0)$ has $1$ as entrance boundary.
%\item $(N_t,t\geq 0)$ has $\infty$ as entrance boundary if and only if $(X_t^{\mathrm{r}},t\geq 0)$ has $1$ as exit boundary\footnote{in this case $(X_t^{\mathrm{r}},t\geq 0)$ and $(X_t^{\mathrm{min}},t\geq 0)$ have the same law.}.
%\item $(N_t,t\geq 0)$ has $\infty$ as regular non-absorbing boundary if and only if  $(X_t^{\mathrm{min}},t\geq 0)$ has $1$ as regular absorbing boundary.
%\item $(N_t,t\geq 0)$ has $\infty$ as regular absorbing boundary if and only if  $(X_t^{\mathrm{r}},t\geq 0)$ has $1$ as regular non-absorbing boundary. 
%\item $(N_t,t\geq 0)$ has $\infty$ as natural boundary if and only if $(X_t^{\mathrm{r}},t\geq 0)$ has $1$ as natural boundary. 
%\end{enumerate}
\end{lemma}
\begin{proof} Recall that by Lemma \ref{recextensX}, $(X_t^{\mathrm{r}},t<\tau_1)$ has the same law as $(X^{\mathrm{min}}_t,t<\tau_1)$. As we shall use it repeatedly, we recall the duality relationships \eqref{dualEFC} and \eqref{dualityEFCref}: for any $n\in \mathbb{N}$, $x\in [0,1]$ and $t\geq 0$
\[\mathbb{E}[x^{N^{(n)}_t}]\overset{\eqref{dualEFC}}{=}\mathbb{E}[X^{\mathrm{min}}_t(x)^n]  \text{ and }\mathbb{E}[x^{N^{\mathrm{min},(n)}_t}]\overset{\eqref{dualityEFCref}}{=}\mathbb{E}[X^{\mathrm{r}}_t(x)^n].\]
%(1) and (2): 
By Lemma \ref{almostsureinx}, the limit $X^{\mathrm{r}}_t(1):=\underset{x\rightarrow 1-}{\lim} X^{\mathrm{r}}_t(x)$ exists almost surely. By letting $x$ go towards $1$ in the identity \eqref{dualityEFCref} above with $n=1$, we have 
%\[
$\mathbb{E}[X^{\mathrm{r}}_t(1)]=\mathbb{P}_n(\zeta_{\infty}>t)$.
%\]

For the first implication, we look at the contraposition and verify that if $\infty$ is inaccessible  for $(N^{(n)}_t,t\geq 0)$, then the boundary $1$ is absorbing for $(X_t^{\mathrm{r}},t\geq 0)$.  If $\infty$ is inaccessible  for $(N^{(n)}_t,t\geq 0)$ then it is inaccessible for $(N^{\mathrm{min},(n)}_t,t\geq 0)$ and $\mathbb{P}_1(\zeta_{\infty}>t)=\mathbb{E}\big(X^{\mathrm{r}}_t(1)\big)=1$. Therefore, $\mathbb{E}\big(1-X^{\mathrm{r}}_t(1)\big)=0$ and since $X^{\mathrm{r}}_t(1)\leq 1$ a.s, we get $X^{\mathrm{r}}_t(1)=1$ a.s. Thus, the boundary $1$ is absorbing for $(X_t^{\mathrm{r}},t\geq 0)$.   
%
%$\mathbb{P}(X_t^{\mathrm{r}}(1)<1)>0$ and the boundary $1$ is non-absorbing for $(X_t^{\mathrm{r}},t\geq 0)$. 

We then show the second implication. If $\infty$ is accessible for $(N^{(n)}_t,t\geq 0)$, then it is accessible and absorbing for $(N^{\mathrm{min},(n)}_t,t\geq 0)$ and there exists $t>0$ such that $\mathbb{P}_1(\zeta_{\infty}>t)=\mathbb{E}[X^{\mathrm{r}}_t(1)]<1$. Therefore, $\mathbb{P}(X_t^{\mathrm{r}}(1)<1)>0$ and the boundary $1$ is non-absorbing for $(X_t^{\mathrm{r}},t\geq 0)$.  We thus have established that $(X_t^{\mathrm{r}},t\geq 0)$ has boundary $1$ non-absorbing if and only if $\infty$ is accessible for $(N_t^{(n)},t\geq 0)$. 

The second equivalence is shown along similar arguments. Letting $n$ go to $\infty$ in the identity \eqref{dualEFC}, we get for any $x\in [0,1)$, 
%\[
$\mathbb{E}[x^{N^{(\infty)}_t}]=\mathbb{P}(X^{\mathrm{min}}_t(x)=1)$.
%\] 
We see that the boundary $1$ is inaccessible for the process $(X^{\mathrm{min}}_t,t\geq 0)$, which is equivalent to be inaccessible for $(X_t^{\mathrm{r}},t\geq 0)$, if and only if $\mathbb{E}[x^{N^{(\infty)}_t}]=0$ for any $x\in [0,1)$, which is equivalent to  $N^{(\infty)}_t=\infty$ almost surely, that is to say the boundary $\infty$ is absorbing for the process $(N_t^{(n)},t\geq 0)$.    \qed 
\end{proof}
\noindent \textbf{Proof of Theorem \ref{thmmomentduality2}}. The moment duality relationship \eqref{dualityEFCref} is provided by Lemma \ref{recextensX}. Statements (i) to (iv) are deduced by applying Lemma \ref{correspond} and combining the necessary and sufficient conditions for boundaries to be respectively absorbing, non-absorbing and  inaccessible or accessible. We establish statements  (ii) and (iv), the others are obtained via similar arguments. For statement (ii), if $(N^{\mathrm{min}}_t,t\geq 0)$ has $\infty$ as regular absorbing boundary, then the boundary $\infty$ is regular non-absorbing for the non-stopped process $(N_t^{(n)},t\geq 0)$.  Therefore $\mathbb{P}(X_t^{\mathrm{min}}(x)=1)=\mathbb{E}[x^{N_t^{(\infty)}}]>0$ and the boundary $1$ of $(X_t^{\mathrm{r}},t\geq 0)$ is accessible. On the other hand, since $\infty$ is accessible, $\mathbb{E}[X_t(1)^{n}]=\mathbb{P}(\zeta^{(n)}_\infty>t)<1$, and boundary $1$ is also non-absorbing. Hence, $1$ is regular non-absorbing. 

For statement (iv), if $(N_t,t\geq 0)$ has $\infty$ as a natural boundary, then $(N_t,t\geq 0)$ and $(N^{\mathrm{min}}_t,t\geq 0)$ have the same law and $\infty$ being inaccessible. The boundary $1$ of $(X_t^{\mathrm{r}},t\geq 0)$ is thus absorbing. The boundary $\infty$ being absorbing, the boundary $1$ is also inaccessible. Hence, $1$ is a natural boundary.
%\begin{enumerate}
%\item[i)] $(N_t,t\geq 0)$ has $\infty$ as exit boundary\footnote{in this case $(N_t,t\geq 0)$ and $(N^{\mathrm{min}}_t,t\geq 0)$ have the same law.} if and only if  $(X_t^{\mathrm{r}},t\geq 0)$ has $1$ as entrance boundary;
%\item[ii)] $(N_t,t\geq 0)$ has $\infty$ as entrance boundary if and only if $(X_t^{\mathrm{r}},t\geq 0)$ has $1$ as exit boundary\footnote{in this case $(X_t^{\mathrm{r}},t\geq 0)$ and $(X_t^{\mathrm{min}},t\geq 0)$ have the same law.};
%\item[iii)] $(N_t,t\geq 0)$ has $\infty$ as regular non-absorbing boundary if and only if  $(X_t^{\mathrm{min}},t\geq 0)$ has $1$ as regular absorbing boundary;
%\item[iv)]  if and only if  $(X_t^{\mathrm{r}},t\geq 0)$ has $1$ as regular non-absorbing boundary;
%\item[iv)] $(N_t,t\geq 0)$ has $\infty$ as natural boundary if and only if $(X_t^{\mathrm{r}},t\geq 0)$ has $1$ as natural boundary. 
%\end{enumerate}
\qed

The next theorem precises the possible behaviors of the dual processes at their boundaries when they are regular non-absorbing.  Recall that we say that a boundary $b$ for a process $(Z_t,t\geq 0)$ is {\it regular reflecting } if it is regular non absorbing and the random level-set $\overline{\{t>0: Z_t=b\}}$ has Lebesgue measure zero. Recall also that the boundary $b$ is {\it regular for itself} if the process started from $b$ returns immediately to $b$, i.e $\sigma_b:=\inf\{t>0: Z_t=b\}=0$ $\mathbb{P}_b$-almost surely.

We  now establish  the correspondences given in Table \ref{correspondancereg}.
%\begin{table}[htpb]
%\begin{center}
%\begin{tabular}{c|c}
%Boundary of $N$ &  Boundary  of $X$\\
%\hline
%$\infty$  regular for itself &  $1$ regular reflecting \\
%\hline
%$\infty$  regular reflecting & $1$ regular for itself \\ 
%\end{tabular}
%\vspace*{3mm}
%\caption{regular for itself/regular reflecting}
%\label{correspondancereg}
%\end{center}
%\end{table}
\begin{theorem}[regular reflecting/regular for itself]\label{regularforitself} The boundary $1$ of the process $(X_t^{\mathrm{r}},t\geq 0)$ is regular for itself (respectively, regular reflecting) if and only if the boundary $\infty$ of the process $(N_t,t\geq 0) $ is regular reflecting (respectively, regular for itself). 
\end{theorem}
\begin{proof}
By definition, $1$ is regular for itself if for any $t>0$, $\mathbb{P}_1(\sigma_1>t)=0$. We first observe that this is equivalent to the condition $\underset{x\rightarrow 1-}{\lim}\ \mathbb{P}_x(\sigma_1>t)=0$. By applying the Markov property at a time $s>0$, we obtain that for any time $t>0$, $ \mathbb{P}_1 (\sigma_1>t+s)=\mathbb{E}_1\big[\mathbb{P}_{X^{\mathrm{r}}_s(1)}(\sigma_1>t)\mathbbm{1}_{\{\sigma_1>s\}}\big]$. By the right-continuity at $0$ of the sample paths, $X^{\mathrm{r}}_s(1)\underset{s\rightarrow 0+}{\longrightarrow} 1$ almost surely. Thus, $1$ is regular for itself if and only if $\underset{x\rightarrow 1-}{\lim}\ \mathbb{P}_x(\sigma_1>t)=0$ for all $t>0$. Note that for any $x\in [0,1)$ under $\mathbb{P}_x$, $\sigma_1$ has the same law as $\tau_1$. Moreover, $(X^{\mathrm{r}}_{t\wedge \tau_1},t\geq 0)$ has the same law as $(X_t^{\mathrm{min}},t\geq 0)$. By  the duality relationship \eqref{dualEFC}, for any $t>0$,
%\[
$\mathbb{P}_x(\tau_1\leq t)=\mathbb{P}(X_t^{\mathrm{min}}(x)=1)=\mathbb{E}[x^{N_t^{(\infty)}}].$
%\]
Hence, 
%As we have seen before,  $\infty$ is regular reflecting and thus
%\[
$\underset{x\rightarrow 1-}{\lim} \ \mathbb{P}_x(\tau_1\leq t)=\mathbb{P}(N_t^{(\infty)}<\infty).$
%\]
If the boundary $1$ of the process $(X_t^{\mathrm{r}},t\geq 0)$ is regular for itself then  $\underset{x\rightarrow 1-}{\lim}\ \mathbb{P}_x(\tau_1\leq t)=1$ and $\mathbb{P}(N_t^{(\infty)}<\infty)=1$. By Fubini's theorem, the set $\overline{\{t\geq 0: N^{(\infty)}_t=\infty\}}$ has zero Lebesgue measure almost surely, namely, $\infty$ is regular reflecting. If $\infty$ is regular reflecting, then $\mathbb{P}(N_t^{(\infty)}<\infty)=1$ for all $t>0$, then $\underset{x\rightarrow 1-}{\lim}\ \mathbb{P}_x(\tau_1\leq t)=1$, and as noticed before this entails that $\mathbb{P}_1(\sigma_1\leq t)=1$ for all $t>0$. Therefore, $\sigma_1=0$, $\mathbb{P}_1$-almost surely and $1$ is regular for itself.  

We now show that $1$ is regular reflecting if and only if $\infty$ is regular for itself. Set $\zeta_\infty^{(n)}:=\inf\{t>0: N_{t-}^{(n)}=\infty\}$. For any $t>0$ and $n\in \mathbb{N}$, by the duality relationship \eqref{dualabso}, 
%\[
$\mathbb{P}(\zeta_\infty^{(n)}>t)=\underset{x\rightarrow 1-}{\lim}\ \mathbb{E}[x^{N_t^{\mathrm{min},(n)}}]=\mathbb{E}[X_t^{\mathrm{r}}(1)^n].$
%\]
Letting $n$  go to  $\infty$ yields
$$\underset{n\rightarrow \infty}{\lim} \mathbb{P}(\zeta_\infty^{(n)}>t )=\mathbb{P}(X_t^{\mathrm{r}}(1)=1).$$
%\]
Provided that $1$ is regular reflecting for the process $(X_t^{\mathrm{r}},t\geq 0)$, we get $\underset{n\rightarrow \infty}{\lim} \mathbb{P}(\zeta_\infty> t)=0$ for all $t>0$, hence $\zeta_\infty=0$, $\mathbb{P}_{\infty}$-a.s. Therefore,  $\infty$ is regular for itself for $(N_t^{(\infty)},t\geq 0)$. Similarly, if $\infty$ is regular for itself, we see that $1$ is regular reflecting. \qed
%The proof that $1$ is regular for itself \qed
%Finally, one has $\mathbb{P}_1(\sigma_1\leq t) \geq \underset{x\rightarrow 1-}{\lim} \mathbb{P}_x(\sigma_1\leq t)=1$. Since $t$ is arbitrary, $\sigma_1=0$ $\mathbb{P}_1$-a.s. 
\end{proof}
Notice that when boundary $1$ is regular reflecting, then  $1$ is necessarily an  instantaneous point, in the sense that $\tau^{1}:=\inf\{t>0: X_t^{\mathrm{r}}(1)<1\}=0$ a.s. The next proposition shows how the  property for boundary $1$  of being instantaneous is associated to some condition on the boundary $\infty$ of the dual process $(N_t,t\geq 0)$. We recall the notation $\zeta_\infty^{(n)}:=\inf\{t>0: N_t^{(n)}=\infty\}$ for all $n\geq 1$ and that $\infty$ is an instantaneous exit if  for all $t>0$, $\underset{n\rightarrow \infty}{\lim} \mathbb{P}(\zeta_\infty^{(n)}>t)=0$.
\begin{proposition}[Instantaneous entrance]\label{instantanprop} 
Assume that the boundary $1$ is an entrance for $(X_t^{\mathrm{r}},t\geq 0)$.  The boundary $1$ is an instantaneous entrance if and only if $\infty$ is an instantaneous exit. Similarly, the boundary $1$ is an instantaneous exit if and only if $\infty$ is an instantaneous entrance.
%the boundary $\infty$ of process $(N^{(n)}_t,t\geq  0)$ satisfies the following 
%$\zeta_\infty^{(n)}\underset{n\rightarrow \infty}{\longrightarrow} 0$ a.s.
\end{proposition}
\begin{proof} Recall $\tau^1$ the first entrance time in $[0,1)$ of process $(X_t^{\mathrm{r}},t\geq 0)$. 
%By Lemma \ref{monotonicity}, the condition $\zeta_\infty^{(n)}\underset{n\rightarrow \infty}{\longrightarrow} 0$ is equivalent in our setting to the condition: for all $t>0$,
%\begin{equation}\label{tregular}
%$\underset{n\rightarrow \infty}{\lim} \mathbb{P}(\zeta_\infty^{(n)}>t)=0.$
%\end{equation}
The argument is similar to that in the proof of Theorem \ref{regularforitself}. By Theorem \ref{thmmomentduality2}, for any $t\geq 0$,  $\mathbb{E}[X^{\mathrm{r}}_t(1)^{n}]=\mathbb{P}(N_t^{\mathrm{min},(n)}<\infty)=\mathbb{P}(\zeta^{(n)}_\infty>t)$. Since $1$ is not accessible, for any $t>0$,  
\begin{align*}
\mathbb{P}_1(\tau^1>t)&=\mathbb{P}(X_t^{\mathrm{r}}(1)=1)=\underset{n\rightarrow \infty}{\lim}\mathbb{E}[X^{\mathrm{r}}_t(1)^{n}]=\underset{n\rightarrow \infty}{\lim} \mathbb{P}(\zeta^{(n)}_\infty>t)=0
\end{align*}
which allows us to conclude the first equivalence. The second is established similarly from the first duality relationship \eqref{dualEFC}. \qed
\end{proof}
\begin{remark} 
The condition $\underset{n\rightarrow \infty}{\lim} \mathbb{P}(\zeta_\infty^{(n)}>t)=0$ is sometimes called $t$-regularity of the boundary $\infty$, see Kolokoltsov's book \cite[Section 6.1, page 273]{MarkovProcessesSemigroupsandGenerators}. 
\end{remark}
\begin{remark} 
We mention that  when the block counting process $(N_t^{(n)},t\geq 0)$ comes down from infinity (in particular, when $\infty$ is an entrance boundary), then provided that $\Lambda(\{1\})=0$, the process leaves the boundary $\infty$ instantaneously, see \cite[Lemma 2.5]{cdiEFC}.
\end{remark}
We clarify now the longterm behavior of the process $(X_t^{\mathrm{r}},t\geq 0)$ when the boundary $1$ is not an exit. In particular, we establish that in the regular non-absorbing case, the boundary $1$ is transient, in the sense that the level set $\overline{\{t>0; X_t^{\mathrm{r}}=1\}}$ is a.s. bounded.
\begin{theorem}\label{fixationat0} Assume that $\Lambda$ satisfies \eqref{cdipsi}. If $(X_t^{\mathrm{r}}(x),t\geq 0)$
has boundary $1$ either regular non-absorbing or an entrance, then \[\exists \ t_0>0; X_t^{\mathrm{r}}(x)=0 \text{ for all } t\geq t_0, \text{ a.s.}\]
\end{theorem}
\begin{proof}
By the comparison theorem, for all $x\in [0,1]$ and $t\geq 0$, $X_t^{\mathrm{min}}(x)\leq Y_t(x)$ a.s. where $(Y_t(x),t\geq 0)$ is a $\Lambda$-Wright-Fisher process with no selection. Under the condition \eqref{cdipsi}, the latter reaches $0$ with positive probability, and so does the process $(X_t^{\mathrm{min}},t\geq 0)$. 

Assume $1$ is regular non-absorbing for $(X_t^{\mathrm{r}},t\geq 0)$.  Consider the successive excursions out from $1$ of the process $(X_t^{\mathrm{r}},t\geq 0)$ which  crosses a given level $x<1$. Namely, set $\tau^{(0)}_1:=0$ and $ \tau^{(n)}_x:=\inf\{t>\tau^{(n-1)}_1: X_t^{\mathrm{r}}\leq x\}$ and $\tau^{(n)}_1:=\inf\{t>\tau^{(n)}_x: X_t^{\mathrm{r}}=1\}$. Then the processes $(X^{\mathrm{r}}_{(t+\tau_x^{(n)})\wedge \tau_1^{(n)}},t\geq 0)$ are independent and with the same law as $(X_t^{\mathrm{min}},t\geq 0)$ started from $X^{\mathrm{r}}_{\tau_x^{(n)}}\leq x$ a.s. By the comparison property, each process $(X^{\mathrm{r}}_{(t+\tau_x^{(n)})\wedge \tau_1^{(n)}},t\geq 0)$ is below a process
$(X_t^{\mathrm{min}}(x),t\geq 0)$ and since $\mathbb{P}_x(\tau_0<\tau_1)>0$, each excursion has a positive probability to hit $0$. Therefore, there exists almost surely an excursion among the latters which attains the boundary $0$. 

Assume $1$ is an entrance for $(X_t^{\mathrm{r}},t\geq 0)$. The boundary $\infty$ is therefore an exit for the process $(N^{\mathrm{min}}_t,t\geq 0)$ and by the duality relationship \eqref{dualabso}, we get for all $x\in [0,1]$,
\[\underset{t\rightarrow \infty}{\lim}\mathbb{E}[X_t^{\mathrm{r}}(x)]=\underset{t\rightarrow \infty}{\lim} \mathbb{E}[x^{N_t^{\mathrm{min}}}]=0.\]
Hence $\underset{t\rightarrow \infty}{\liminf}\ X_t^{\mathrm{r}}(x)=0$ a.s. Set $\tau_{1/n}:=\inf\{t>0: X_t^{\mathrm{r}}\leq 1/n\}$. For all $n\geq 2$, $\tau_{1/n}<\infty$ a.s. Since $1$ is an entrance boundary, 
for any $x\in [0,1)$, $(X_t^{\mathrm{r}}(x),t\geq 0)$ has the same law as $(X_t^{\mathrm{min}}(x),t\geq 0)$. By the Markov property at time $\tau_{1/n}$, $(X_{t+\tau_{1/n}}^{\mathrm{r}}(x),t\geq 0)$ has the same law $(X^{\mathrm{min}}(X^{\mathrm{r}}_{\tau_{1/n}}(x)),t\geq 0)$. Since $X^{\mathrm{r}}_{\tau_{1/n}}(x))\leq 1/n$ a.s, by the comparison theorem, $(X_{t+\tau_{1/n}}^{\mathrm{r}}(x),t\geq 0)$ is stochastically smaller than $(Y_t(1/n),t\geq 0)$ where $(Y_t(1/n),t\geq 0)$ is a $\Lambda$-Wright-Fisher process with no selection. 

%We now establish that the event $\bigcup_{n=2}^{\infty}\{X_{t+\tau_{1/n}}^{\mathrm{r}}(x)=0 \text{ for some } t\geq 0\}$ has probability one.

Set $E_n:=\{X^{\mathrm{min}}_{t+\tau_{1/n}}>0, \forall t\geq 0\}$ for any $n\geq 2$. One has 
\[
\mathbb{P}(E_n)\leq \mathbb{P}(Y_t(1/n)>0 \text{ for all } t\geq 0)=\mathbb{P}_{1/n}(\tau^Y_0>\tau^Y_1),\]
where $\tau^Y_i:=\inf\{t\geq 0; Y_t=i\}$ for $i=0,1$. By the duality relationship for the pure $\Lambda$-coalescent: 
%\[
$\mathbb{P}_x(\tau^Y_1<\tau^Y_0)=\mathbb{E}_\infty[x^{N^{Y}_t}]\underset{x\rightarrow 0}{\longrightarrow} 0$
%\]
where we have denoted by $(N_t^{Y},t\geq 0)$ the moment dual of $(Y_t,t\geq 0)$. Thus $\mathbb{P}(E_n)\underset{n\rightarrow \infty}{\longrightarrow} 0$. Since $0$ is an absorbing boundary, $E_{n+1}\subset E_n$ for all $n\geq 2$. Hence 
%\[
$\mathbb{P}\left(\cap_{n=2}^{\infty}E_n\right)=\underset{n\rightarrow \infty}{\lim} \mathbb{P}(E_n)=0.$
%\]
This allows us to conclude since $\cup_{n=2}^{\infty} E_n^{c}$ has probability $1$ and
%\[
$\cup_{n=2}^{\infty} E_n^{c}\subset \{\exists t_0\geq 0; X_t^{\mathrm{r}}(x)=0, \ \forall t\geq t_0\}.$
%\]
%\bigcup_{n=2}^{\infty}\{X_{t+\tau_{1/n}}^{\mathrm{r}}(x)=0 \text{ for some } t\geq 0\}$ has probability one   
\qed
% then the process $(N_t^{(\infty),a},t\geq 0)$ has either $\infty$ as an exit or as a regular absorbing boundary. Recall that $(X^{\mathrm{min}}_t,t\geq 0)$ is a (bounded positive) submartingale. In the case where $\infty$ is an exit ($1$ is an entrance), by the duality relationship \eqref{dualabso}
%\[\mathbb{P}(\underset{t\rightarrow \infty}{\lim} X_t^{\mathrm{min}}(x)=1) =\underset{t\rightarrow \infty}{\lim}\mathbb{E}[x^{N_t^{(n)}}]=0.\]
%Hence $X_t^{\mathrm{min}}(x)\underset{t\rightarrow \infty}{\longrightarrow} 0$ a.s. Since $1$ is an entrance, $(X_t^{\mathrm{min}}(x),t\geq 0)$ has the same law as $(X_t^{\mathrm{r}}(x),t\geq 0)$ and the statement is established.\\
%
%Assume now that $\infty$ is regular reflecting, then by Proposition \ref{recpositive} 
%
%\[\underset{t\rightarrow \infty}{\lim} \mathbb{E}[X_t^{\mathrm{min}}(x)]=\underset{t\rightarrow \infty}{\lim} \mathbb{E}[x^{N^{(1)}_t}]=\mathbb{E}[x^{N_\infty}]<1\]
%where $N_\infty$ is a random variable whose law is the stationary distribution of $(N_t,t\geq 0)$.
%Hence we see that $\mathbb{P}(X_t^{\mathrm{min}}(x)\underset{t\rightarrow \infty}{\longrightarrow} 0)>0$. Since the boundary  $1$ is regular, the process $(X_t^{r},t\geq 0)$ makes excursion out of $1$, each successive excursion crossing level $0<x<1$ has the same law as $(X_t^{\mathrm{min}}(x),t\geq 0)$, therefore there exists almost surely one such excursion that will go to $0$, and 
%\[X_t^{\mathrm{r}}(x)\underset{t\rightarrow \infty}{\longrightarrow} 0 \text{ a.s.}\]
\end{proof}

Until now we  have only shown theoretical results on possible extensions of the minimal process and their duality relationships with the process $(N_t,t\geq 0)$ and the stopped process $(N^{\mathrm{min}}_{t},t\geq 0)$. One may wonder whether there exist mechanisms of resampling $\Lambda$ and selection $f$ that result in the regular boundary, see the last line in Table \ref{correspondance}. It is not clear whether the easiest route to study a given process is to look at its dual or not. The aim of the next section is to gather all results known about block counting processes for simple EFCs and to translate them to results for the $\Lambda$-Wright-Fisher process with frequency-dependent selection.
\section{Sufficient conditions and explicit cases}\label{application}

We transfer the results known for the block counting process $(N_t,t\geq 0)$ to the $\Lambda$-Wright-Fisher process with  frequency-dependent selection by applying our two duality relationships \eqref{momentdual1} and \eqref{momentdual2}. Recall the correspondences stated in Table \ref{correspondance} and Table \ref{correspondancereg}. In the sequel, we work with the extension $(X_t^{\mathrm{r}},t\geq 0)$ of the minimal solution to \eqref{SDEselec}, $(X_t^{\mathrm{min}},t< \tau)$, which is constructed in Lemma \ref{recextensX}.

\subsection{Sufficient conditions for $1$ to be an exit or an entrance}
%\\
%The simplest condition is the following 
%\begin{proposition} If $\int^{1}\frac{\ddr x}{x-f(x)}=\infty$ then the boundary $1$ is an exit.
%\end{proposition}
%\begin{proof}
%The condition $\int^{1}\frac{\ddr x}{x-f(x)}=\infty$ entails that pure branching process above $(N_t,t\geq 0)$ does not explode.
%the deterministic part of the SDE \eqref{SDEselecintro} cannot start from $1$, so that the deterministic flow $(x_t(x),t\geq 0)$ solution to
%\[\ddr x_t(x)=\mu(\mathbb{N})(f(x_t(x))-x_t(x))\ddr t, x_0(x)=x\]
%satisfies $x_t(1-)=1$ for all $t\geq 0$.
%\end{proof}
%For instance any measure $\mu$ such that $\ell(n)\geq (\log n)^{2+\epsilon}$ with $\epsilon>0$ satisfies $\mathbb{H}$. 
Recall the assumptions $\Lambda(\{0\})=\Lambda(\{1\})=0$. According to Lemma \ref{recextenslambda}, when the function $f$ is defective, namely  $f(1)<1$,  the $\Lambda$-WF process with selection  $(X_t,t\geq 0)$, minimal solution to \eqref{SDEselec}, has boundary $1$ entrance. 
We are interested in this section on the non-defective selection functions for which $f(1)=1$. The next theorem provides a sufficient condition on the resampling measure $\Lambda$ and the selection function $f$ entailing that the process solution to \eqref{SDEselec}, has boundary $1$ as an entrance. Recall the map $\Phi$ defined in \eqref{phi2}. Set $\Phi(x):=\Phi(\lfloor x \rfloor)$ for any $x\geq 2$. 

Introduce the following condition $\mathbb{\textbf{A}}$ over the drift function $x\mapsto x-f(x)$.

\noindent $\mathbb{\textbf{A}}$: there exists a positive function $L$ defined on $(0,1)$ such that  
\begin{center} $x-f(x)\geq L(x)$ for $x$ close enough to $1$, \end{center}
the map $h:x\mapsto \frac{L(x)}{(1-x)\log\left(1/(1-x)\right)}$ is eventually non-decreasing in the neighbourhood of $1$, $\int^{1-} \frac{1}{L(x)}\ddr x<\infty$ and
%$\Phi(x)\geq g(\log x)\log x$ with $g$ eventually non-decreasing such that $\int^{\infty}\frac{\ddr x}{x g(x)}$
\begin{equation}\label{entrancecondu}\frac{(1-x)^2\Phi\big(1/\log(1/x)\big)}{x-f(x)}\underset{x\rightarrow 1-}{\longrightarrow} 0.
\end{equation}

\begin{remark} The first condition in  $\mathbb{\textbf{A}}$ stipulates 
that $(x-f(x))/(1-x)\log(1/(1-x))$ stays above a non-decreasing function when $x$ is close to $1$. This requires some regularity in the difference quotient of the function $f$ near $1$. 
\end{remark}
\begin{remark} Since $\Phi$ is non-decreasing, $\Phi\big(1/\log(1/x)\big)\leq \Phi\big(1/(1-x)\big)$ for any $x\in [1/2,1)$. 
\end{remark}
\begin{theorem}\label{theorementrancef}  If the function $f$ satisfies  condition $\mathrm{\textbf{A}}$, then  the boundary $1$ is an instantaneous entrance.
\end{theorem}

\begin{example} 
\begin{enumerate}
\item Let $c>0$ and $\alpha \in (0,1)$. If $x-f(x)\geq c(1-x)^{\alpha}$ for $x$ close enough to $1$ and 
%\[
if $\Phi\big(1/(1-x)\big)(1-x)^{2-\alpha}\underset{x\rightarrow 1-}{\longrightarrow} 0$,
%\] 
then condition $\mathrm{\textbf{A}}$ is satisfied. This holds for instance when $\Phi(n)\underset{n\rightarrow \infty}{\sim} dn^{\beta}$ with $0<\beta<1-\alpha$ and $d>0$.
\item Let $\alpha>0$. If $x-f(x)\geq \log\left(1/(1-x)\right)^{1+\alpha}(1-x)^3$ for $x$ close enough to $1$  and \[\frac{\Phi\big(1/(1-x)\big)}{1/(1-x) \log\big(1/(1-x)\big)^{1+\alpha}}\underset{x\rightarrow 1-}{\longrightarrow} 0,\]
% \frac{1}{1-x}
then  condition $\mathrm{\textbf{A}}$ is satisfied. This holds for instance when $\Phi(n)\underset{n\rightarrow \infty}{\sim} dn(\log n)^{\beta}$ with $0<\beta<1+\alpha$ and $d>0$.
\end{enumerate}
\end{example}
%The proof of Theorem \ref{theorementrancef} is based on Lemma \ref{exitcriterion}  which provides a condition for $\infty$ to be an exit boundary for the process $(N_t,t\geq 0)$ involving the splitting measure $\mu$.  
%Recall the process $(X^{\mathrm{r}}_t,t\geq 0)$ constructed in Theorem \ref{thmmomentduality2} 
%\begin{lemma}[Theorem 3.1 in \cite{explosion}]\label{suffcondentrance}
%%\begin{enumerate}
%%\item 
%If $n\mapsto \ell(n)$ satisfies the following condition $\mathbb{H}$:
%
%\noindent $\mathbb{H}$: there exists a positive function $g$ on $\mathbb{R}_+$ eventually non-decreasing such that
%\begin{center} $\ell(n)\geq g(\log n)\log n \text{ for large } n$, $\int^{\infty}\frac{1}{xg(x)}\ddr x<\infty$ \end{center} and \begin{equation}\label{suffcondentrancePhi} \underset{n\rightarrow \infty}{\lim}\  \frac{\Phi(n)}{n\ell(n)}=0,
%\end{equation}
%then the process $(X^{\mathrm{r}}_t,t\geq 0)$ has $1$ as an entrance boundary.
%%\\
%%\item  
%%\\
%%\end{enumerate}
%\end{lemma} 

\noindent \textbf{Proof of Theorem \ref{theorementrancef}}. 
If the boundary $\infty$ of the process $(N_t,t\geq 0)$ is an exit, then Theorem \ref{thmmomentduality2} implies that the process $(X_t^{\mathrm{r}},t\geq 0)$ has boundary $1$ entrance. 
The fact that $1$ is an entrance boundary will therefore be a simple consequence of Lemma \ref{suffcondexitN}. We will use some Tauberian theorems to show that condition $\mathrm{\textbf{A}}$  for the drift function $x\mapsto x-f(x)$ entails the condition $\mathbb{H}$  in Lemma \ref{suffcondexitN}.  Recall the map $\ell :n\mapsto \sum_{k=1}^{n}\bar{\mu}(k)$.

Condition $\mathrm{\textbf{A}}$ ensures that the integral $ \int^{1-} \frac{1}{x-f(x)}\ddr x$ is finite, which implies $f'(1-)=\infty$. Set $u(x):=\mu(\mathbb{N})(x-f(x))$ for all $x\in [0,1]$ and recall $\ell(n)=\sum_{k=1}^{n}\bar{\mu}(k)$. For all $\lambda \geq 0$, set
$\kappa(\lambda):=\int_{0}^{\infty}\left(1-e^{-\lambda x}\right)\mu(\ddr x)$ with $\mu(\ddr x)=\sum_{k=1}^{\infty}\mu(k)\delta_k$. Since $\frac{1-f(x)}{1-x}\underset{x\rightarrow 1-}{\longrightarrow} f'(1-)=\infty$, \[  u(x)=\mu(\mathbb{N})(x-f(x))=\mu(\mathbb{N})\left(1-f(x)-(1-x)\right)\underset{x\rightarrow 1-}{\sim} \mu(\mathbb{N})(1-f(x)).\]
For $\lambda\geq 0$, \begin{equation}\label{equiv}
u(e^{-\lambda})\underset{\lambda \rightarrow 0}{\sim} \mu(\mathbb{N})\left(1-f(e^{-\lambda})\right)=\sum_{k=1}^{\infty}(1-e^{-\lambda k})\mu(k)=\kappa(\lambda).
\end{equation}
By \cite[Proposition 1, Chapter III]{Bertoin96}, there exists a universal constant $c>1$ such that
\[\frac{1}{c} \int_{1}^{1/\lambda}\bar{\mu}(x)\ddr x \leq \frac{\kappa(\lambda)}{\lambda}\leq c \int_{1}^{1/\lambda}\bar{\mu}(x)\ddr x,\]
where $\bar{\mu}(x)=\bar{\mu}(k)$ for any $x\in [k,k+1[$. One can check 
%\[
$\int_{1}^{1/\lambda}\bar{\mu}(x)\ddr x \underset{\lambda \rightarrow 0}{\sim} \ell(\lfloor 1/\lambda \rfloor ).$
%\]
By change of variable $\lambda=\log 1/x$ and using the equivalence \eqref{equiv}, we see that for $x$ close enough to $1$:
\begin{equation}\label{comparisonlu} \frac{1}{c}\ell\left(\left \lfloor 1/\log 1/x \right \rfloor \right)\leq \frac{u(x)}{\log 1/x}\leq c\ell\left(\left \lfloor 1/\log 1/x  \right \rfloor \right).\end{equation}
%We can now rewrite the conditions of Theorem \ref{suffcondentrance} over $\ell$ in terms of the function $u$. 
%Setting $h(x)=g(\log 1/\log 1/x)$, Condition $\mathbb{H}$ is equivalent to 
We now show that condition $\mathrm{\textbf{A}}$  entails  condition  $\mathbb{H}$. By condition $\mathrm{\textbf{A}}$, if $x$ is close enough to $1$, then $u(x)\geq \mu(\mathbb{N})L(x)$ for some function $L$ such that the map $h$ satisfying  $\mu(\mathbb{N})L(x)=(1-x)\log\left(1/(1-x)\right) h(x)$, is non-decreasing. Moreover, since $1-x \underset{x\rightarrow 1-}{\sim} \log 1/x$, there is $C>0$ such that if $x$ is close enough to $1$ then
\[u(x)\geq \mu(\mathbb{N})L(x)=(1-x)\log\left(1/(1-x)\right) h(x) \geq C (\log 1/x) (\log 1/\log 1/x) h(x).\]
By applying the upper bound in \eqref{comparisonlu} in the inequality above, we see that $$\ell\left(\lfloor 1/\log 1/x \rfloor \right)\geq C\log( 1/\log 1/x) h(x),$$ for some constant $C>0$. Thus, for large enough $n$,
$\ell(n)\geq C(\log n)g(\log n)$ with
% if $x$ is near $1$, 
%%\[
%$\frac{\ell\left(\lfloor 1/\log 1/x \rfloor \right)}{\lfloor 1/\log 1/x \rfloor }\geq C\log( 1/\log 1/x) h(x).$
%\]
$g$ the map such that $g(\log 1/\log 1/x):=Ch(x)$. Since by assumption $h$ is non-decreasing in a neighbourhood of $1$, the map $g$ is eventually non-decreasing. One also easily checks that $\int^{1}\frac{1}{L(x)}\ddr x<\infty$ entails $\int^{\infty}\frac{1}{xg(x)}\ddr x<\infty$. Finally, by condition \eqref{entrancecondu} and the bounds  \eqref{comparisonlu}, we see that $\underset{n\rightarrow \infty}{\lim}\  \frac{\Phi(n)}{n\ell(n)}=0$, hence Condition $\mathbb{H}$ is fulfilled.

It remains to show that the entrance at boundary $1$ is instantaneous. It has been established in \cite[Lemma 5.3]{explosion}, that under condition $\mathbb{H}$, the boundary $\infty$ of process $(N_t^{(n)},t\geq 0)$ satisfies the condition $\zeta_\infty^{(n)}\underset{n\rightarrow \infty}{\longrightarrow} 0$ a.s., where we recall $\zeta_\infty^{(n)}$ stands for the first explosion time of the process $(N_t^{(n)},t\geq 0)$. We can therefore apply Proposition \ref{instantanprop} which ensures that boundary $1$ of the dual process $(X_t^{\mathrm{r}},t\geq 0)$ is an instantaneous entrance. \qed

\begin{remark}\label{tauberianremark} Note that \eqref{comparisonlu} entails that 
%the equivalence : 
%\begin{center} 
$\int^{1-}\frac{\ddr x}{x-f(x)}<\infty$ if and only if $\sum_{n\geq 1}\frac{1}{n\ell(n)}<\infty$. 
%\end{center}
We recover analytically here the equivalence between  Dynkin's condition and Doney's condition for explosion of a pure discrete branching process in continuous time whose offspring law is $\mu$ and generating function of $\mu$ is $f$, see Doney \cite{Doney1984} and e.g. \cite[Section 2.4]{explosion}.
\end{remark}

By considering resampling measures $\Lambda$ and selection functions $f$ with some properties of regular variation, we can use Tauberian theorems to relate the asymptotics of the selection function $f$ at boundary $1$, with asymptotics of the splitting measure $\mu$ at $\infty$. 

We gather here these results. Recall $u(x):=\mu(\mathbb{N})\big(x-f(x)\big)$ for any $x\in (0,1)$ and the equivalence \eqref{equiv}, $\kappa(\lambda):=\int_{0}^{\infty}(1-e^{-\lambda x})\mu(\ddr x)\underset{\lambda \rightarrow 0}{\sim} u(e^{-\lambda})$. 
%The condition $\mathbb{H}'$ is equivalent to $\kappa(\lambda)\geq \lambda^{2}h(e^{-\lambda})$ for some $h$ such that $\int^{1}\frac{1}{h(x)\log (1/\log 1/x)}\ddr x<\infty$. 
Let $s$ be a slowly varying function at $\infty$. 
By the Tauberian theorem, see e.g. \cite[Chapter 0.7 page 10]{Bertoin96}, the following are equivalent: 
\begin{itemize}
\item[(i)] $\mu(n)\underset{n\rightarrow \infty}{\sim} \frac{b }{n^{\alpha+1}}s(n)$ for some $\alpha\in (0,1)$,
\item[(ii)] $\kappa(\lambda)\underset{\lambda \rightarrow 0+}{\sim} \lambda^{\alpha}\frac{b\Gamma(2-\alpha)}{\alpha(1-\alpha)}s(1/\lambda)$, 
\item[(iii)] $\mu(\mathbb{N})(x-f(x))\underset{x\rightarrow 1-}{\sim}\kappa(\log 1/x)\underset{x\rightarrow 1-}{\sim} \frac{b\Gamma(1-\alpha)}{\alpha}(1-x)^{\alpha}s\big(\frac{1}{1-x}\big).$
\end{itemize}
Similarly, if $s$ is slowly varying at $\infty$, then we have the equivalence:
\begin{itemize}
\item[(1)] $\mu(\mathbb{N})\big(x-f(x)\big)\underset{x\rightarrow 1-}{\sim}\kappa(\log 1/x)\underset{x\rightarrow 1-}{\sim} s\big(1/(1-x)\big)$,
\item[(2)] $\bar{\mu}(n)\underset{n\rightarrow \infty}{\sim} s(n)$.
\end{itemize}
%
%The condition \eqref{entrancecondu} is equivalent to $\frac{\Phi(x)}{\kappa(x)}\underset{x\rightarrow \infty}{\longrightarrow} 0$. By \cite[Theorem 8.1.6, page 333]{regularvariation}, the map $k\mapsto \bar{\mu}(k)$ is slowly varying at $\infty$ if and only if $\lambda\mapsto \kappa(\lambda)$ is slowly varying at $0$ and moreover in this case, $\kappa(\lambda)\underset{\lambda \rightarrow 0}{\sim} \bar{\mu}(1/\lambda)$. Hence, if $\bar{\mu}$ is slowly varying at $\infty$, we have that \[u(x)\underset{x\rightarrow 1}{\sim} \kappa(\log 1/x) \underset{x\rightarrow 1}{\sim} \bar{\mu}(1/\log 1/x)\underset{x\rightarrow 1}{\sim} \bar{\mu}\big(1/(1-x)\big).\]
%\begin{itemize}
%\item[i)] Notice that for all $n\geq 2$, $\Phi(n)\leq Cn^{2}$ for some constant $C>0$ and that the following equivalences holds 
%\begin{center}
%$\sum_{n=2}^{\infty}\frac{1}{n}\bar{\mu}(n)<\infty \Longleftrightarrow \sum_{n=1}^{\infty}\mu(n)\log n <\infty \Longleftrightarrow \int^1\frac{x-f(x)}{1-x} \ddr x<\infty$. \end{center}
%Hence, if $\int^1\frac{x-f(x)}{1-x} \ddr x<\infty$ then boundary $1$ is absorbing whenever it is reached.
%\item[ii)] 
The next theorem provides a sufficient condition over the selection function $f$ and the resampling measure $\Lambda$ for $1$ to be an absorbing boundary for the (non-stopped) process $(X_t^{\mathrm{r}},t\geq 0)$, so that $1$ is either an exit or a natural boundary. 
\begin{theorem}\label{suffcondexitf}
Assume $x\mapsto x-f(x)$ regularly varying at $1$ with index $\alpha\in [0,1)$. If \begin{equation}\label{integralform2}\int^{1-}\frac{x-f(x)}{(1-x)^3\Phi\big(1/(1-x)\big)} \ddr x<\infty,
\end{equation} then the boundary $1$ of $(X^{\mathrm{r}}_t,t\geq 0)$ is absorbing. Moreover in this case,
\begin{itemize}
\item[i)] if $\int^{\infty}\frac{\ddr x}{\Phi(x)}<\infty$, then $1$ is  an exit boundary,
\item[ii)] if $\int^{\infty}\frac{\ddr x}{\Phi(x)}=\infty$, then $1$ is a natural boundary.
\end{itemize}
%\end{itemize}
\end{theorem}

\begin{proof}
According to Lemma \ref{nonexplosioncriterion}, if $\sum_{n=2}^{\infty}\frac{n}{\Phi(n)}\bar{\mu}(n)<\infty$, then the block counting process $(N_t,t\geq 0)$ does not explode, i.e. $\infty$ is inaccessible. By Lemma \ref{correspond}, the process $(X_t^{\mathrm{r}}(x),t\geq 0)$ has in this case its boundary $1$ absorbing. If moreover, $\sum_{n=2}^{\infty}\frac{1}{\Phi(n)}<\infty$, then the process $(N_t,t\geq 0)$ has $\infty$ as an entrance. Theorem \ref{thmmomentduality2} ensures that $1$ is an exit for $(X_t^{\mathrm{r}}(x),t\geq 0)$. If  $\sum_{n=2}^{\infty}\frac{1}{\Phi(n)}=\infty$, then $\infty$ is natural and by Theorem \ref{thmmomentduality2}, $1$ is natural.  It remains to show that the condition \eqref{integralform2} entails the convergence of the series $\sum_{n=2}^{\infty}\frac{n}{\Phi(n)}\bar{\mu}(n)$. Recall the Tauberian equivalence $\mathrm{(iii)} \Longleftrightarrow \mathrm{(i)}$.  When $x\mapsto x-f(x)$ is regularly varying at $1$ with index $\alpha\in [0,1)$, $x-f(x)\underset{x\rightarrow 1-}{\sim} c(1-x)^{\alpha}s\big(1/(1-x)\big)$ for some slowly varying function $s$ at $\infty$ and $\bar{\mu}(n)\underset{n\rightarrow \infty}{\sim} \frac{c'}{n^{\alpha}}s(n)$ for some constant $c'>0$. Hence, $\bar{\mu}\left(\frac{1}{1-x}\right)\underset{x\rightarrow 1-}{\sim} c'(1-x)^{\alpha}s\big(1/(1-x)\big) \underset{ x\rightarrow 1-}{\sim} c'' (x-f(x))$
where $c''$ is a positive constant. Simple integral comparisons 	and a change of variable give the equivalences:
\begin{align*}
\sum_{n=2}^{\infty}\frac{n}{\Phi(n)}\bar{\mu}(n)<\infty &\Longleftrightarrow \int^{1-}\frac{1/(1-x)}{\Phi\big(1/(1-x)\big)} (1-x)^{\alpha}s\big(1/(1-x)\big)\frac{\ddr x}{(1-x)^2}<\infty \\
&\Longleftrightarrow \int^{1-}\frac{x-f(x)}{(1-x)^3\Phi\big(1/(1-x)\big)} \ddr x<\infty.
\end{align*}
\qed
\end{proof}

%By \cite[Theorem 8.1.6, page 333]{regularvariation}, the map $k\mapsto \bar{\mu}(k)$ is slowly varying at $\infty$ if and only if $\lambda\mapsto \kappa(\lambda)$ is slowly varying at $0$ and moreover in this case, $\kappa(\lambda)\underset{\lambda \rightarrow 0}{\sim} \bar{\mu}(1/\lambda)$.
%
%, there exist constants $d$ and $b$ such that $\Phi(n)\underset{n\rightarrow \infty}{\sim} dn^{1+\beta}$ and $\mu(n)\underset{n\rightarrow \infty}{\sim} \frac{b}{n^{\alpha+1}}$.
%This a direct application of Theorem \ref{stablefragtheorem} and Theorem \ref{} :
\subsection{A $\Lambda$-Wright-Fisher process with selection and boundary $1$ regular.}
When the splitting measure $\mu$ and the coalescence measure $\Lambda$ satisfy certain properties of regular variation at $\infty$ and $0$, respectively, some sharp conditions  for classifying the boundary $\infty$ of the block counting process $(N_t,t\geq 0)$ have been established in \cite{explosion}. 

We transfer those results  to results for the $\Lambda$-Wright-Fisher process with selection.
%Recall $\Phi$.  We now consider mechanism $\Lambda$ of resampling  with certain property of regular variation at $0$. The next

\begin{theorem}\label{regularcase} 
Let $\alpha,\beta\in (0,1)$ and $\sigma,\rho>0$. Assume 
\begin{center}
$\Lambda(\ddr z)=h(z)\ddr z \text{ with } h(z)\underset{x\rightarrow 0+}{\sim}\rho z^{-\beta}$ and $\mu(\mathbb{N})\big(x-f(x)\big)\underset{x\rightarrow 1-}{\sim} \sigma (1-x)^{\alpha}$.
\end{center}  The boundary $1$ of $(X_t^{\mathrm{r}},t\geq 0)$ is classified as follows:
\begin{itemize}
\item[i)] if $\alpha+\beta<1$, then 
%$(\#\Pi(t),t\geq 0)$ explodes and stays infinite almost surely:
%\vspace*{1mm}
%\begin{center}
$1$ is an entrance boundary,
%\end{center}
\vspace*{1mm}
\item[ii)] if $\alpha+\beta>1$, then 
%$(\#\Pi(t),t\geq 0)$ does not explode and 
%when started from a partition with infinitely many blocks of infinite size, 
%comes down from infinity instantaneously almost surely:
%\begin{center}
$1$ is an exit boundary,
% \end{center}
\vspace*{1mm}
\item[iii)] if $\alpha+\beta=1$ and further,
%$\sigma:=\frac{b}{d}\frac{\pi}{\alpha \sin(\pi\alpha)}$ and $\theta:=\frac{b}{d}\frac{1}{\alpha(1-\alpha)}$ then, $\sigma>\theta$ and
\vspace*{2mm}
\begin{itemize}
\item if $\sigma/\rho>\frac{\pi}{(2-\alpha)\sin(\pi \alpha)}$, then
% $(\#\Pi(t),t\geq 0)$ explodes and stays infinite almost surely:
%\vspace*{1mm}
%\begin{center} 
$1$ is an entrance boundary,
\vspace*{2mm}
%\end{center}
%\vspace*{2mm}
\item if $\frac{1}{(1-\alpha)(2-\alpha)}<\sigma/\rho<\frac{\pi}{(2-\alpha)\sin(\pi \alpha)}$, then
% $(\#\Pi(t),t\geq 0)$ explodes and then comes down from infinity instantaneously almost surely:
%\vspace*{2mm}
%\begin{center}
$1$ is a regular reflecting boundary,
\vspace*{2mm}
%\end{center}
%\vspace*{2mm}
\item if $\sigma/\rho<\frac{1}{(1-\alpha)(2-\alpha)}$, then 
%$(\#\Pi(t),t\geq 0)$ does not explode and
% when started from a partition with infinitely many blocks of infinite size,
%comes down from infinity instantaneously almost surely:
%\vspace*{2mm}
%\begin{center}
$1$ is an exit boundary.
%\end{center}
\end{itemize}
\end{itemize}
\end{theorem}

\begin{proof} 
Tauberian theorems ensure that the conditions over $\Lambda$ and $f$ are equivalent to
%\begin{equation*}
$\Phi(n)\underset{n\rightarrow \infty}{\sim} d n^{\beta+1}
\text{ and } 
\mu(n)\underset{n\rightarrow \infty}{\sim} \frac{b}{n^{1+\alpha}}$
%\end{equation*}
with 
%\[
$d:=\frac{\Gamma(1-\beta)}{\beta(1+\beta)}\rho \text{ and } b:=\frac{\alpha}{\Gamma(1-\alpha)}\sigma,$
%\]
see e.g. \cite[Section 2.2, page 12]{cdiEFC} for the equivalent of function $\Phi$.
Cases i) and ii) are obtained by applying respectively Theorem \ref{theorementrancef} and  Theorem \ref{suffcondexitf}. \cite[Theorem 3.7]{explosion} classifies the cases when $\alpha=1-\beta$ according to the ratio $b/d$. The three cases in iii) are obtained by noticing that $\sigma/\rho>\frac{\pi}{(2-\alpha)\sin(\pi \alpha)}$ is equivalent to $\frac{b}{d}>\alpha(1-\alpha)$ and $\sigma/\rho>\frac{1}{(1-\alpha)(2-\alpha)}$ is equivalent to $b/d>\frac{\alpha \sin(\pi \alpha)}{\pi}$. In particular, when $\frac{1}{(1-\alpha)(2-\alpha)}<\sigma/\rho<\frac{\pi}{(2-\alpha)\sin(\pi \alpha)}$, the process $(N_t,t\geq 0)$ has $\infty$ as a regular non-absorbing boundary. Therefore, the process $(N^{\mathrm{min}}_t,t\geq 0):=(N_{t\wedge \zeta_{\infty}},t\geq 0)$ has $\infty$ as a regular absorbing boundary, and Theorem \ref{thmmomentduality2} ensures that the process $(X_t^{\mathrm{r}},t\geq 0)$ has boundary $1$ regular \textit{non-absorbing}.  We now need to check that the boundary $1$ is regular reflecting. \cite[Proposition 3.9]{explosion} ensures that the dual process $(N_t,t\geq 0)$ has boundary $\infty$ regular for itself, the fact that $1$ is reflecting is a consequence of Theorem \ref{regularforitself}. 
 \qed
\end{proof}
We now precise the behavior of the process $(X_t^{\mathrm{r}},t\geq 0)$ at its boundary $1$ when it is regular reflecting by showing that the boundary is regular for itself.
\begin{theorem}\label{regularforitselfstable} Let $\sigma,\rho>0$ and $\alpha \in (0,1)$.
Assume 
%\begin{center}
$\Lambda(\ddr z)=h(z)\ddr z \text{ with } h(z)\underset{x\rightarrow 0+}{\sim}\rho z^{-(1-\alpha)}$ and $\mu(\mathbb{N})\big(x-f(x)\big)\underset{x\rightarrow 1-}{\sim} \sigma (1-x)^{\alpha}$.
%\end{center} 
If $\frac{1}{(1-\alpha)(2-\alpha)}<\sigma/\rho<\frac{\pi}{(2-\alpha)\sin(\pi \alpha)}$, then $(X_t^{\mathrm{r}},t\geq 0)$ has its boundary $1$  regular for itself.
\end{theorem}
\begin{remark}
Since the process $(X_t^{\mathrm{r}},t\geq 0)$ is Feller,  when boundary $1$ is regular reflecting and regular for itself, standard theory, see e.g. \cite[Chapter IV]{Bertoin96} ensures the existence of a local time of the process $(X_t^{\mathrm{r}},t\geq 0)$ at $1$ whose inverse subordinator has no drift.
\end{remark}
\begin{proof} Recall $\mathcal{A}$ in \eqref{generator} and $\mathcal{A}^{\mathrm{s}}$ in \eqref{generatorWFs}.  Let $\epsilon>0$. We look for a positive continuous function on $[0,1]$, $g$ which satisfies the following conditions: $g\in C^{2}([0,1])$, $g(1)=0$ and there is $c>0$, such that $\mathcal{A}^{\mathrm{s}}g(x)\leq -c$ for any $x\in (1-\epsilon,1)$. 
By applying \cite[Proposition 6.3.2, page 281]{MarkovProcessesSemigroupsandGenerators}, we will have for any $t>0$, $\mathbb{P}_x(\sigma_1\leq t)\underset{x\rightarrow 1-}{\longrightarrow} 1.$
Observe that
\begin{align*}
\mathcal{A}g(x)= & \ x\int_{0}^{1}\big(g(x+z(1-x))-g(x)-z(1-x)g'(x)\big) z^{-2}\Lambda(\ddr z)\\
&+(1-x)\int_{0}^{1}\big(g(x(1-z))-g(x)+zxg'(x)\big) z^{-2}\Lambda(\ddr z)\\
&:=\mathcal{A}^{+}g(x)+\mathcal{A}^{-}g(x).
\end{align*}
Recall $\mathcal{A}^{\mathrm{s}}g(x)=\mathcal{A}g(x)+\mu(\mathbb{N})(f(x)-x)g'(x)$ for al $x\in [0,1]$, see  \eqref{generatorWFs}. 
Note that for any function $g\in C^{2}((0,1))$, one has for any $y, u$,
\[g(y+u)-g(y)-ug'(y)=u^2\int_0^1 g''(y+vu)(1-v)\ddr v.\]
Let $g(x)= (1-x)^\delta$ for $0<x<1$ and $0<\delta<1$. Plainly, 
\[g(1)=0,\quad g'(x)= -\delta(1-x)^{\delta-1}<0, \quad g''(x)=\delta(\delta-1)(1-x)^{\delta-2}<0.\]
Since by assumption $\mu(n)\underset{n\rightarrow \infty}{\sim} \frac{b}{n^{1+\alpha}}$, then by Remark \ref{tauberianremark}-ii), we see that $x-f(x) \underset{x\rightarrow 1-}{\sim} C\log(1/x)^{\alpha}.$
Hence, the generating function $f$ satisfies 
$\limsup_{x\rightarrow 1-}\frac{x-f(x)}{(1-x)^\alpha}<\infty$.  Choosing $\delta=1-\alpha$, we have
%\begin{equation}
%	\begin{split}
$(f(x)-x)g'(x)
=-\delta(f(x)-x)(1-x)^{\delta-1}<c$
%				\end{split}
%\end{equation}
for some $c>0$ and all $x$ close to $1$. In addition, for all $x\in [0,1]$,
\begin{equation}\label{A+}
	\begin{split}
\mathcal{A}^{+}g(x)&=x\int_0^1 (g(x+z(1-x))-g(x)-z(1-x) g'(x))z^{-2}\Lambda(\ddr z) \\
&=\delta(\delta-1)x(1-x)^2\int_0^1 \Lambda(\ddr z)\int_0^1(1-x-vz(1-x) )^{\delta-2}(1-v)\ddr v\\
&=\delta(\delta-1)x(1-x)^\delta\int_0^1\Lambda(\ddr z)\int_0^1(1-vz )^{\delta-2}(1-v)\ddr v<0,
\end{split}
\end{equation}
and for all $x$ close to $1$,
\begin{equation}\label{A-}
	\begin{split}
	\mathcal{A}^{-}g(x)&=(1-x)\int_0^1(g(x(1-z))-g(x)+zxg'(x))z^{-2}\Lambda(\ddr z) \\
	&=\delta(\delta-1)(1-x)x^2	\int_0^1 \Lambda(\ddr z)\int_0^1 (1-x+vxz )^{\delta-2}(1-v)\ddr v\\
    &\leq c_1(\delta-1)(1-x)\int_0^1 (1-x+z)^{\delta-2}  \Lambda(\ddr z)\int_0^1(1-v)\ddr v  \\
    &\leq  c_2(\delta-1)(1-x)	\int_0^{1-x} (1-x+z)^{\delta-2}  \Lambda(\ddr z)  \\
    &\leq -c_3(1-x)(1-x)^{\delta-2}\int_0^{1-x}\Lambda(\ddr z) \leq -c_3(1-x)^{\delta-1}\int_0^{1-x}\Lambda(\ddr z),
\end{split}
\end{equation}		
where $c_i, i\in \{1,2,3\}$ are positive constants. By assumption  $\Lambda(\ddr z)=h(z)\ddr z$ with $h(z)\underset{z\rightarrow 0}{\sim} \rho z^{\alpha-1}$. Thus, $\int_0^{1-x}\Lambda(\ddr z)\underset{x\rightarrow 1-}{\sim} C(1-x)^{\alpha}$ for some constant $C>0$. Combining \eqref{A+} and \eqref{A-}, and recalling that 
$\delta=1-\alpha$, we see that  $\underset{x\rightarrow 1-}{\limsup}\ A^sg(x)\leq -c$ for some positive constant $c$. By Lemma \ref{almostsureinx}, for any $t>0$, $X^{\mathrm{r}}_t(x)\underset{x\rightarrow 1-}{\longrightarrow} X^{\mathrm{r}}_t(1-)=X^{\mathrm{r}}_t(1)$ almost surely. Therefore $\mathbb{P}_{1}(\sigma_1\leq t)=\underset{x\rightarrow 1-}{\lim}\ \mathbb{P}_x(\sigma_1\leq t)=1$, and since $t$ is arbitrary small, $\sigma_1=0$, $\mathbb{P}_1$-a.s, i.e. $1$ is regular for itself.
%For $\beta=1-\alpha$, if $\Lambda(0, 1)$ is large enough we also have 
%$A^sg(x)\rightarrow -\infty$ as $x\rightarrow 1-$.
\qed
\end{proof}
%\begin{example}  Let $\alpha\in (0,1)$ and $a>0$, set $\Lambda_{\alpha,a}(\ddr x)=\frac{1}{\mathrm{B}(\alpha,a)}x^{\alpha-1}(1-x)^{a-1}\ddr x$ with $\mathrm{B}(\alpha,a):=\int_{0}^{1}x^{\alpha-1}(1-x)^{a-1}\ddr x$. Set $f$ be the generating function $f:x\mapsto 1-(1-x)^{\alpha}$ defined for $x\in [0,1]$. {\blue{To continue}}
%\end{example}

We finally state a corollary of Theorem \ref{regularforitselfstable} for the  block counting process $(N_t,t\geq 0)$ of a simple EFC process $(\Pi(t),t\geq 0)$ whose splitting measure $\mu$ and coalescence measure $\Lambda$ are regularly varying. By Theorem \ref{Markovblockcounting}, $(N_t,t\geq 0):=(\#\Pi(t),t\geq 0)$ is a Markov process with state-space $\bar{\mathbb{N}}$. The first assertion i) below answers a question raised but left unadressed in \cite[page 31]{explosion}.
\begin{corollary}\label{inftyregularreflecting} Let $\alpha\in (0,1)$. Assume $\Phi(n)\underset{n\rightarrow \infty}{\sim} dn^{2-\alpha}$ with $d>0$ and $\mu(n)\underset{n\rightarrow \infty}{\sim} \frac{b}{n^{1+\alpha}}$ with $b>0$. 
\begin{itemize}
\item[i)] If $\frac{\alpha \sin(\pi \alpha)}{\pi}<b/d<\alpha(1-\alpha),$ then the boundary $\infty$ of the process $(\#\Pi(t),t\geq 0)$ is regular reflecting.
\item[ii)] If $b/d<\alpha(1-\alpha)$, then the process $(\#\Pi(t),t\geq 0)$ is positive recurrent and admits a stationary distribution carried over $\mathbb{N}$.
\end{itemize}
\end{corollary}
\begin{proof}
By Theorem \ref{regularforitselfstable}, the boundary $1$ of $(X_t^{\mathrm{r}},t\geq 0)$ is regular for itself. Then statement i) follows  by applying Theorem \ref{regularforitself}. Statement ii) is a consequence of Theorem \ref{rec}. \qed
\end{proof}

Other explicit cases have been found in \cite{explosion} and have their counterparts for the dual $\Lambda$-Wright-Fisher processes with frequency-dependent selection. We refer for instance to \cite[Theorem 3.11]{explosion} for the case $\Phi(n)\underset{n\rightarrow \infty}{\sim} d n(\log n)^{\beta}$ and $\mu(n)\underset{n\rightarrow \infty}{\sim} b \frac{(\log n)^{\alpha}}{n^2}$.

We finally mention a question that has not been addressed in the present article. The equivalence stated in Theorem \ref{regularforitself} entails that if one of the processes has its boundary \textit{irregular for itself} then the other process has its boundary \textit{sticky}, in the sense that the process stays at the boundary a set of times of positive Lebesgue measure.  More precisely, for the process $(X_t^{\mathrm{r}}(x),t\geq 0)$, the duality relation \eqref{dualabso} implies that for any time $t\geq 0$,  $\underset{n\rightarrow \infty}{\lim}\mathbb{E}[x^{N_t^{\mathrm{min},(n)}}]=\mathbb{P}(X_t^{\mathrm{r}}(x)=1)$, which might be positive if the process $(N^{(\infty)}_t,t\geq 0)$ does not return to $\infty$ instantaneously i.e. $\underset{n\rightarrow \infty}{\lim}\mathbb{P}_n(\zeta_\infty> t)>0$ for some $t>0$. Whether there are some resampling measure $\Lambda$ and selection function $f$ for which boundaries $1$ or $\infty$ are indeed sticky or irregular for themselves remains an unsolved question. 

\noindent \textbf{Acknowledgements:} C.F's research is partially  supported by LABEX MME-DII (ANR11-LBX-0023-01). X.Z.'s research is supported by Natural Sciences and Engineering Research Council of Canada (RGPIN-2016-06704) and by  National Natural Science Foundation of China (No.\  11731012).
\providecommand{\MR}{\relax\ifhmode\unskip\space\fi MR }
\providecommand{\MRhref}[2]{%
  \href{http://www.ams.org/mathscinet-getitem?mr=#1}{#2}
}

\end{document}